\documentclass[11pt]{amsart}
\usepackage{amsmath,amsthm,amssymb,amscd}

\newtheorem{theorem}{Theorem}[section]
\newtheorem{corollary}[theorem]{Corollary}
\newtheorem{proposition}[theorem]{Proposition}
\newtheorem{lemma}[theorem]{Lemma}

\theoremstyle{definition}    
\newtheorem{definition}[theorem]{Definition}
\newtheorem{example}[theorem]{Example}
\newtheorem{remark}[theorem]{Remark}
\newtheorem*{acks}{Acknowledgments}

\theoremstyle{remark}

\def\Q{\mathbb{Q}}
\def\R{\mathbb{R}}
\def\C{\mathbb{C}}
\def\P{\mathbb{P}}
\def\Z{\mathbb{Z}}
\def\N{\mathbb{Z}_{\geq 0}}

\def\GL{\operatorname{GL}}
\def\SL{\operatorname{SL}}
\def\gl{\operatorname{\mathfrak{gl}}}
\def\g{\mathfrak{g}}

\def\Re{\operatorname{Re}}
\def\Lie{\operatorname{Lie}}
\def\Spec{\operatorname{Spec}}
\def\Proj{\operatorname{Proj}}
\def\End{\operatorname{End}}
\def\Hom{\operatorname{Hom}}
\def\Ker{\operatorname{Ker}}
\def\Coker{\operatorname{Coker}}
\def\Im{\operatorname{Im}}
\def\codim{\operatorname{codim}}
\def\rank{\operatorname{rank}}
\def\corank{\operatorname{corank}}
\def\Tr{\operatorname{Tr}}
\def\Ad{\operatorname{Ad}}
\def\Res{\mathrm{Res}}
\def\gr{\operatorname{gr}}
\def\lset#1#2{\left\{ \, \left. #1 \hspace{0.25em} \right| \, #2 \, \right\} }
\def\rset#1#2{\left\{ \, #1 \, \left| \hspace{0.25em} #2\right. \, \right\} } 
\def\ov{\overline}
\def\mc#1{\mathcal{#1}}

\def\vin{\operatorname{in}} 
\def\vout{\operatorname{out}} 
\def\bA{\mathbf{A}} 
\def\bC{\mathbf{C}} 
\def\bv{\mathbf{v}} 
\def\bw{\mathbf{w}} 
\def\mbe{\mathbf{e}} 
\def\bM{\mathbf{M}} 
\def\bMM{\mathbf{M}^{\circ}} 
\def\bMs{\mathbf{M}^\mathrm{s}} 
\def\bMss{\mathbf{M}^\mathrm{ss}} 
\def\bF{\mathbb{F}} 
\def\GIT{/\hspace{-3pt}/} 
\def\MM{\mathcal{M}} 
\def\MMreg{\mathcal{M}^\mathrm{s}} 
\def\LL{\mathcal{L}} 
\def\M{\mathfrak{M}} 
\def\Mreg{\mathfrak{M}^\mathrm{s}} 
\def\bU{\mathbf{U}} 
\def\PP{\mathcal{P}} 
\def\SS{\mathcal{S}} 
\def\CF{\operatorname{CF}} 
\def\FF{\mathcal{F}} 
\def\RH{\mathrm{RH}} 
\def\bMl{\mathbf{M}^{\ell}} 
\def\MMl{\mathcal{M}^{\ell}} 
\def\MMlreg{\mathcal{M}^{\ell \, \rm s}}

\begin{document}
\title{Geometry of Multiplicative Preprojective Algebra}
\author{Daisuke Yamakawa}
\date{}
\address{Department~of~Mathematics, %
Faculty~of~Science, Kyoto~University, Kyoto, 606-8502, Japan}
\email{yamakawa@math.kyoto-u.ac.jp}


\begin{abstract}
Crawley-Boevey and Shaw recently 
introduced a certain multiplicative analogue 
of the deformed preprojective algebra, 
which they called the multiplicative preprojective algebra. 
In this paper we study the moduli 
space of (semi)stable representations of 
such an algebra (the multiplicative 
quiver variety), which in fact 
has many similarities to the 
quiver variety. We show that 
there exists a complex analytic 
isomorphism between the nilpotent subvariety 
of the quiver variety and 
that of the multiplicative quiver 
variety (which can be extended 
to a symplectomorphism between these 
tubular neighborhoods). We also show 
that when the quiver is 
star-shaped, the multiplicative quiver variety 
parametrizes Simpson's (poly)stable filtered 
local systems on a punctured 
Riemann sphere with prescribed filtration 
type, weight and associated graded 
local system around each puncture.        
\end{abstract}

\maketitle

\tableofcontents

\section{Introduction}\label{1}

In this paper we study the geometry of multiplicative preprojective relation. 

First let us recall the notion of (deformed) preprojective relation. 
Let $Q=(I,\Omega)$ be a finite quiver with vertex set $I$ and arrow set $\Omega$, 
and let $(I,H)$ be its ``double''; 
that is obtained by adding a reverse arrow $\ov h$ to $\Omega$ for each $h \in \Omega$. 
For $h\in H$, we denote by $\vout(h), \vin(h) \in I$ 
the outgoing, incoming vertex of $h$, respectively. 
A representation of $(I,H)$ is given by a pair $(V,x)$ 
of an $I$-graded vector space $V=\bigoplus V_i$ 
and a family $x=(x_h)_{h\in H}$ of linear maps 
$x_h \colon V_{\vout(h)} \to V_{\vin(h)}$. 
Then for $\zeta=(\zeta_i) \in \C^I$, the equation  
\[
(\mu_V)_i(x) := \sum_{h\in H; \vin(h)=i} \epsilon(h) x_h x_{\ov h} = \zeta_i 1_{V_i} 
\qquad (i\in I)
\] 
is called the {\em (deformed) preprojective relation}.  
Here $\epsilon(h)=1$ if $h \in \Omega$ and $\epsilon(h)=-1$ otherwise. 
One of the most important properties of this relation is that for a fixed $V$,  
the map 
\[
\mu_V \colon \bM(V) :=\bigoplus_{h \in H} \Hom(V_{\vout(h)},V_{\vin(h)}) 
\to \bigoplus_{i \in I} \End(V_i)
\]
satisfies the defining property of moment map 
for the natural action of $G_V :=\prod_i \GL(V_i)$. 
Thus taking a stability condition on $\bM(V)$ 
in the sense of geometric invariant theory, 
the quotient $\mu_V^{-1}(\zeta)^{\rm s} /G_V$ 
of the stable locus carries naturally a symplectic structure. 
Such a quotient is so-called the {\em quiver variety}. 
Strictly speaking, 
there are various choices of stability condition parametrized by $\theta \in \Q^I$ 
($\theta$-stability), and  
the quiver variety $\M_{\zeta,\theta}(V)=\mu_V^{-1}(\zeta)^{\theta-\mathrm{ss}} \GIT G_V$ 
is defined as the quotient of the $\theta$-semistable locus.   
In general its stable locus $\Mreg_{\zeta,\theta}(V)$ does not coincide with the whole space.     

An importance of quiver variety in geometry was firstly found by Kronheimer~\cite{Kro}. 
He described the minimal resolution $\widetilde{\C^2/\Gamma}$   
of the Kleinnian singularity  
as the quiver variety associated to a quiver of the extended Dynkin type 
corresponding to $\Gamma \subset \SL_2(\C)$ via the McKay correspondence. 
Motivated by this fact and the ADHM description of moduli spaces of 
instantons on such spaces by Kronheimer and him~\cite{KN}, 
Nakajima introduced the notion of quiver variety  
in his celebrated paper~\cite{Nak-Duke1}.  
In the same paper, developing Lusztig's idea he constructed geometrically 
all irreducible highest weight representations of Kac-Moody Lie algebras. 
For further developments in this direction, see \cite{Nak-Duke2, Nak-AMS}. 

On the other hand, Crawley-Boevey and Shaw  
recently introduced a certain ``multiplicative'' 
analogue of the preprojective relation, 
called the {\em multiplicative preprojective relation}~\cite{CS};   
that is    
\[
(\Phi_V)_i(x) := \prod_{h\in H; \vin(h)=i} (1+ x_h x_{\ov h})^{\epsilon(h)} 
= q_i 1_{V_i} \qquad (i\in I), 
\]         
where we have fixed $q=(q_i) \in (\C^\times)^I$ and an ordering for taking product, 
and have assumed that $\det (1+x_h x_{\ov h} ) \neq 0$ for all $h \in H$. 
They considered such a relation motivated by the 
{\em Deligne-Simpson problem}. 
Fix a number of conjugacy classes $\mc{C}_1, \dots , \mc{C}_n$ in $\GL(r,\C)$. 
Then the problem asks if irreducible solutions of the equation
\[
A_1 A_2 \cdots A_n = 1 \qquad (A_i \in \mc{C}_i)
\]  
exist. Here the word ``irreducible'' means that 
$A_i$'s have no common invariant non-zero proper subspace.
There is also an additive version of it;  
replace conjugacy classes $\mc{C}_i$ by coadjoint orbits $\mc{O}_i \subset \gl(r,\C)$, 
and replace the above equation by 
\[
A_1 + A_2 + \cdots + A_n = 0 \qquad (A_i \in \mc{O}_i).
\]  
It is well-known that the closure of any coadjoint orbit in $\gl(r,\C)$ 
can be described as the quiver variety associated to a quiver of type $A$, 
where the stability is nothing so that the resulting quiver variety 
is the affine quotient $\mu_V^{-1}(\zeta) \GIT G_V$. 
Based on this fact, Crawley-Boevey observed that the quiver variety  
associated to a {\em star-shaped} quiver: 
 
\vspace{10pt}
 
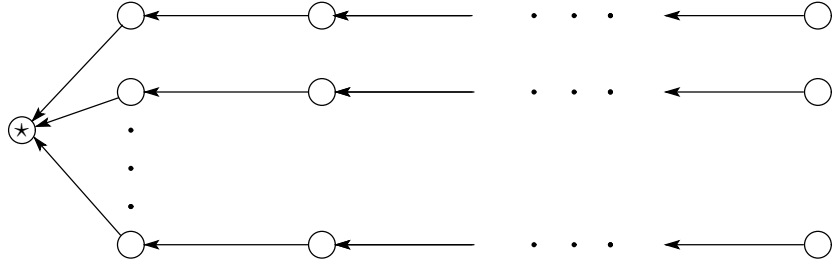
\begin{center} 
\unitlength 0.1in
\begin{picture}( 45.5000, 13.4000)(  5.0500,-16.7000)
\put(8.2000,-10.0000){\makebox(0,0){$\star$}}%
%
\special{pn 8}%
\special{ar 820 1000 70 70  0.0000000 6.2831853}%
%
\special{pn 8}%
\special{ar 1390 400 70 70  0.0000000 6.2831853}%
%
\special{pn 8}%
\special{ar 2390 400 70 70  0.0000000 6.2831853}%
%
\special{pn 8}%
\special{ar 4986 400 70 70  0.0000000 6.2831853}%
%
\special{pn 8}%
\special{ar 1390 800 70 70  0.0000000 6.2831853}%
%
\special{pn 8}%
\special{ar 2390 800 70 70  0.0000000 6.2831853}%
%
\special{pn 8}%
\special{ar 4986 800 70 70  0.0000000 6.2831853}%
%
\special{pn 8}%
\special{ar 1390 1600 70 70  0.0000000 6.2831853}%
%
\special{pn 8}%
\special{ar 2390 1600 70 70  0.0000000 6.2831853}%
%
\special{pn 8}%
\special{ar 4986 1600 70 70  0.0000000 6.2831853}%
%
\special{pn 8}%
\special{pa 1336 1550}%
\special{pa 886 1040}%
\special{fp}%
\special{sh 1}%
\special{pa 886 1040}%
\special{pa 914 1104}%
\special{pa 920 1080}%
\special{pa 944 1078}%
\special{pa 886 1040}%
\special{fp}%
%
\special{pn 8}%
\special{pa 2316 400}%
\special{pa 1466 400}%
\special{fp}%
\special{sh 1}%
\special{pa 1466 400}%
\special{pa 1532 420}%
\special{pa 1518 400}%
\special{pa 1532 380}%
\special{pa 1466 400}%
\special{fp}%
%
\special{pn 8}%
\special{pa 3166 400}%
\special{pa 2456 400}%
\special{fp}%
\special{sh 1}%
\special{pa 2456 400}%
\special{pa 2522 420}%
\special{pa 2508 400}%
\special{pa 2522 380}%
\special{pa 2456 400}%
\special{fp}%
\special{pa 3176 400}%
\special{pa 2456 400}%
\special{fp}%
\special{sh 1}%
\special{pa 2456 400}%
\special{pa 2522 420}%
\special{pa 2508 400}%
\special{pa 2522 380}%
\special{pa 2456 400}%
\special{fp}%
%
\special{pn 8}%
\special{pa 2316 800}%
\special{pa 1466 800}%
\special{fp}%
\special{sh 1}%
\special{pa 1466 800}%
\special{pa 1532 820}%
\special{pa 1518 800}%
\special{pa 1532 780}%
\special{pa 1466 800}%
\special{fp}%
%
\special{pn 8}%
\special{pa 2316 1600}%
\special{pa 1466 1600}%
\special{fp}%
\special{sh 1}%
\special{pa 1466 1600}%
\special{pa 1532 1620}%
\special{pa 1518 1600}%
\special{pa 1532 1580}%
\special{pa 1466 1600}%
\special{fp}%
%
\special{pn 8}%
\special{pa 3176 800}%
\special{pa 2466 800}%
\special{fp}%
\special{sh 1}%
\special{pa 2466 800}%
\special{pa 2532 820}%
\special{pa 2518 800}%
\special{pa 2532 780}%
\special{pa 2466 800}%
\special{fp}%
\special{pa 3186 800}%
\special{pa 2466 800}%
\special{fp}%
\special{sh 1}%
\special{pa 2466 800}%
\special{pa 2532 820}%
\special{pa 2518 800}%
\special{pa 2532 780}%
\special{pa 2466 800}%
\special{fp}%
%
\special{pn 8}%
\special{pa 3176 1600}%
\special{pa 2466 1600}%
\special{fp}%
\special{sh 1}%
\special{pa 2466 1600}%
\special{pa 2532 1620}%
\special{pa 2518 1600}%
\special{pa 2532 1580}%
\special{pa 2466 1600}%
\special{fp}%
\special{pa 3186 1600}%
\special{pa 2466 1600}%
\special{fp}%
\special{sh 1}%
\special{pa 2466 1600}%
\special{pa 2532 1620}%
\special{pa 2518 1600}%
\special{pa 2532 1580}%
\special{pa 2466 1600}%
\special{fp}%
%
\special{pn 8}%
\special{pa 4910 400}%
\special{pa 4190 400}%
\special{fp}%
\special{sh 1}%
\special{pa 4190 400}%
\special{pa 4258 420}%
\special{pa 4244 400}%
\special{pa 4258 380}%
\special{pa 4190 400}%
\special{fp}%
%
\special{pn 8}%
\special{pa 4910 800}%
\special{pa 4190 800}%
\special{fp}%
\special{sh 1}%
\special{pa 4190 800}%
\special{pa 4258 820}%
\special{pa 4244 800}%
\special{pa 4258 780}%
\special{pa 4190 800}%
\special{fp}%
%
\special{pn 8}%
\special{pa 4910 1600}%
\special{pa 4190 1600}%
\special{fp}%
\special{sh 1}%
\special{pa 4190 1600}%
\special{pa 4258 1620}%
\special{pa 4244 1600}%
\special{pa 4258 1580}%
\special{pa 4190 1600}%
\special{fp}%
%
\special{pn 8}%
\special{pa 1326 830}%
\special{pa 896 980}%
\special{fp}%
\special{sh 1}%
\special{pa 896 980}%
\special{pa 966 978}%
\special{pa 946 962}%
\special{pa 952 940}%
\special{pa 896 980}%
\special{fp}%
%
\special{pn 8}%
\special{pa 1346 450}%
\special{pa 876 950}%
\special{fp}%
\special{sh 1}%
\special{pa 876 950}%
\special{pa 936 916}%
\special{pa 912 912}%
\special{pa 906 888}%
\special{pa 876 950}%
\special{fp}%
%
\special{pn 8}%
\special{sh 1}%
\special{ar 1390 1000 10 10 0  6.28318530717959E+0000}%
\special{sh 1}%
\special{ar 1390 1200 10 10 0  6.28318530717959E+0000}%
\special{sh 1}%
\special{ar 1390 1400 10 10 0  6.28318530717959E+0000}%
\special{sh 1}%
\special{ar 1390 1400 10 10 0  6.28318530717959E+0000}%
%
\special{pn 8}%
\special{sh 1}%
\special{ar 3500 400 10 10 0  6.28318530717959E+0000}%
\special{sh 1}%
\special{ar 3710 400 10 10 0  6.28318530717959E+0000}%
\special{sh 1}%
\special{ar 3900 400 10 10 0  6.28318530717959E+0000}%
\special{sh 1}%
\special{ar 3900 400 10 10 0  6.28318530717959E+0000}%
%
\special{pn 8}%
\special{sh 1}%
\special{ar 3500 800 10 10 0  6.28318530717959E+0000}%
\special{sh 1}%
\special{ar 3710 800 10 10 0  6.28318530717959E+0000}%
\special{sh 1}%
\special{ar 3900 800 10 10 0  6.28318530717959E+0000}%
\special{sh 1}%
\special{ar 3900 800 10 10 0  6.28318530717959E+0000}%
%
\special{pn 8}%
\special{sh 1}%
\special{ar 3500 1600 10 10 0  6.28318530717959E+0000}%
\special{sh 1}%
\special{ar 3710 1600 10 10 0  6.28318530717959E+0000}%
\special{sh 1}%
\special{ar 3900 1600 10 10 0  6.28318530717959E+0000}%
\special{sh 1}%
\special{ar 3900 1600 10 10 0  6.28318530717959E+0000}%
\end{picture}%
\end{center}

\vspace{10pt}

\noindent
with no stability and an appropriate parameters $V$ and $\zeta$, 
is isomorphic to the variety 
\[
\mc{Q} :=\{\, (A_1,A_2, \dots ,A_n) \in \ov{\mc{O}}_1 \times \cdots \times \ov{\mc{O}}_n 
\mid A_1 + \cdots + A_n =0 \,\} \GIT \GL(r,\C).
\]  
Here the equation $\sum_i A_i =0$ arises as the preprojective relation at the vertex $\star$. 
He solved the additive version~\cite{Cra-add} 
using this idea,   
and in \cite{CS}, he and Shaw observed that 
the ``multiplicative quiver variety'' $\Phi_V^{-1}(q) \GIT G_V$ describes 
the multiplicative analogue of the above variety:
\begin{equation*}
\mc{R} :=\{\, (A_1,A_2, \dots ,A_n) \in \ov{\mc{C}}_1 \times \cdots \times \ov{\mc{C}}_n 
\mid A_1 \cdots A_n =1 \,\} \GIT \GL(r,\C).
\end{equation*}%
Note that fixing distinct $n$ points $p_1, \dots ,p_n$ in the Riemann sphere $\P^1$, 
this variety can be considered as the moduli space of representations of 
the fundamental group (the {\em character variety}) of $\P^1 \setminus \{ p_i \}$ 
whose local monodromy around each $p_i$ belongs to $\ov{\mc{C}}_i$. 

We have mentioned that the preprojective relation can be understood as a moment map. 
In fact, the multiplicative preprojective relation can be also understood as a 
``multiplicative analogue'' of moment map, called  
the {\em group-valued moment map}. 
The notion of group-valued moment map 
was introduced by Alekseev-Malkin-Meinrenken~\cite{AMM}, 
and Van den Bergh~\cite{V-poi, V-ham} observed that the map $\Phi_V$ 
together with an appropriate 2-form 
satisfies the defining properties of group-valued moment map. 
A general theory of group-valued moment map allows us to take the ``quotient'' like as 
moment map; the quotient space 
$\MMreg_{q,\theta}(V):=\Phi_V^{-1}(q)^{\theta-\mathrm{s}}/ G_V$ 
of the $\theta$-stable locus has naturally a symplectic structure.   
We call the quotient $\MM_{q,\theta}(V):=\Phi_V^{-1}(q)^{\theta-\mathrm{ss}}\GIT G_V$ 
of the semistable locus the {\em multiplicative quiver variety}, 
which and its stable locus are the main objects in this paper.  

Note that if we consider $\Phi_V(x)$ as a formal series in $x_h$, 
then it can be written as  
\[
\Phi_V(x) = 1 + \mu_V(x) + ( \text{higher order terms in $x_h$} ).
\]
Thus we may expect a certain direct relation between the quiver variety and 
the multiplicative quiver variety. 
In fact, in the case of star-shaped quivers there is a 
{\em monodromy map} between them.   
If each $\mc{O}_i$ is semi-simple and eigenvalues are generic, 
then the variety $\mc{Q}$ becomes smooth and 
there is a map from $\mc{Q}$ to the variety $\mc{R}$ with 
$\mc{C}_i:= \exp \mc{O}_i$ given by: 
\begin{align*}
(A_1, \cdots ,A_n) \longmapsto 
\text{the monodromy representation of the connection}& \\ 
d- \frac{1}{2\pi \sqrt{-1}}\sum_i \frac{A_i}{z-p_i} dz 
&\quad \text{on} \quad  \P^1 \setminus \{ p_i \}. 
\end{align*}
Such a map was considered by Hitchin~\cite{Hit} and Hausel~\cite{Hau} 
(Boalch~\cite{Boa1, Boa2} considered its generalization to the case of irregular singularity). 
Hitchin showed that the monodromy map is a local analytic isomorphism and interchanges 
the symplectic structures.   
Hausel conjectured that under this map, 
the cohomology of $\mc{Q}$ is isomorphic to the pure part of 
one of $\mc{R}$. In this direction, he and Rodriguez-Villegas~\cite{HR} 
suggested several interesting conjectures 
for the mixed Hodge polynomial of twisted character varieties of 
compact Riemann surfaces.  

In this paper, using a property of group-valued moment map we show that: 
\begin{theorem}[Corollary~\ref{3.3.4}]\label{1.1} 
There exist an open neighborhood   
$U$ {\rm (}resp.\ $U'${\rm )} of $[0] \in \MM_{1,0}(V)$ 
{\rm (}resp.\ $[0] \in \M_{0,0}(V)${\rm )}  
and a commutative diagram   
\[
\begin{CD}
\MM_{1,\theta}(V) \supset @.\, \pi^{-1}(U) @>{\tilde f}>> \pi^{-1}(U')\, @. \subset \M_{0,\theta}(V) \\
@. @V{\pi}VV @V{\pi}VV @. \\
@. U @>{f}>> U' @. 
\end{CD}
\]
such that $f([0])=[0]$ and both $\tilde{f}$ and $f$ are complex analytic isomorphisms. 
Moreover $\tilde{f}$ maps the stable locus symplectomorphically onto the stable locus. 
\end{theorem}

Let us consider a star-shaped quiver again. 
Then the associated $\MM_{q,0}(V)$ with appropriate $q, V$ 
and $\theta =0$ gives the variety $\mc{R}$.   
Now by definition there is a natural projective morphism 
$\pi \colon \MM_{q,\theta}(V) \to \MM_{q,0}(V)=\Phi_V^{-1}(q)\GIT G_V$.    
We show that: 
\begin{theorem}[Theorem~\ref{4.2.5}]\label{1.2}
Suppose that $\theta_i >0$ for any $i \neq \star$. 
Then the variety $\MM_{q,\theta}(V)$   
parametrizes Simpson's polystable filtered local systems on $(\P^1,\{ p_i \})$   
of which the filtration type, weight and 
the monodromy of the graded local systems 
are prescribed by $V, \theta$ and $q$, respectively.  
Moreover $\pi$ can be understood as the map taking the monodromy representation 
of the underlying local system.  
\end{theorem}

For the notion of filtered local system, see \cite{Sim} 
or \S\ref{4} in this paper.  
This notion naturally arises as an object which should correspond to 
a {\em parabolic connection} by the Riemann-Hilbert correspondence.  
In fact Simpson constructed such a correspondence. 
On the other hand, the moduli space of parabolic connections on 
a compact Riemann surface with marked points 
was constructed by Inaba-Iwasaki-Saito~\cite{IIS-par1} in the case 
of genus 0 and rank 2, and by Inaba~\cite{Ina} in the case of 
general genus, rank and full filtrations.  
We show that under certain conditions on a stability parameter, 
Simpson's Riemann-Hilbert correspondence gives 
a complex analytic symplectomorphism between such a moduli space and 
a star-shaped multiplicative quiver variety (see Theorem~\ref{4.3.5}).     

The paper is organized as follows: 
\begin{itemize}
\item In \S\ref{2}, we give a quick review of some basic facts 
about quiver variety and group-valued moment map.
\item In \S\ref{3}, we define the multiplicative quiver variety, 
and give some properties of it. 
\item \S\ref{4} is devoted to the study in the case of star-shaped quivers.
\end{itemize}  
Results in the rest two sections are little bit modifications of 
the known results.
\begin{itemize}
\item In \S\ref{5}, we show that a functor introduced by Crawley-Boevey and Shaw 
induces an isomorphism between two multiplicative quiver varieties 
whose parameters relate by certain reflections.  
This is a multiplicative version of Maffei's result~\cite{Maf}.    
\item In \S\ref{6}, by the same method as Nakajima~\cite{Nak-Duke1},  
we construct all irreducible highest weight representations of a Kac-Moody Lie algebra 
using the vector spaces of constructible functions on subvarieties of the  
multiplicative quiver varieties.    
\end{itemize}

\begin{acks}
The author is extremely grateful to Professor Hiraku Nakajima 
for prompting his interest in the geometry of 
multiplicative preprojective algebra, 
and for valuable advice and discussions. 
Also, the author is much obliged to Professor Masa-Hiko Saito 
for answering his questions about moduli space of parabolic connections. 
\end{acks}

\section{Preliminaries}\label{2}

\subsection{Notation and convention}\label{2.1}

Throughout this paper we use the following:

\begin{itemize}
\item $(I,\Omega)$ --- a finite quiver whose vertex set is $I$ and arrow set is $\Omega$.
\item $(I,\ov \Omega)$ --- the quiver obtained by reversing all arrows in $\Omega$.  
We set $H:= \Omega \sqcup \ov \Omega$. 
\item $\ov h \in H~(h\in H)$ --- the reverse arrow of $h$. 
\item $\epsilon \colon H \to \{ -1,1 \}$ --- the map defined by 
$\epsilon (h) =1=-\epsilon (\ov h)$ for $h \in \Omega$. 
\item $\vin (h), \vout (h) \in I$ --- the incoming, outgoing vertex of $h \in H$, respectively. 
\item $H_i~(i\in I)$ --- the subset of $H$ consisting of all $h$ with $\vin(h)=i$.
\item $\alpha \cdot \beta$ --- the standard inner product on $\Z^I$; 
$\alpha \cdot \beta = \sum_{i\in I} \alpha_i \beta_i$. 
\item $( \alpha ,\beta ):= 2\alpha \cdot \beta -\sum_{h \in H} \alpha_{\vout(h)} \beta_{\vin(h)}$. 
\item $\mbe_i~(i\in I)$ --- the $i$-th coordinate vector in $\Z^I$.
\end{itemize}   

A variety is a complex algebraic variety, not required to be irreducible or reduced. 
We always work over $\C$, and use the Zariski topology 
unless otherwise specified. 

On a smooth variety, 
we will treat symplectic structures both in the algebraic sense and 
in the complex analytic sense. 
We call the former ``algebraic symplectic structures'', 
and the latter ``holomorphic symplectic structures''.  
We use the word ``algebraic symplectic manifold''   
as a smooth variety endowed with an algebraic symplectic structure. 
A morphism or a holomorphic map $f\colon X \to Y$ between algebraic symplectic manifolds 
is {\em symplectic} if the pull-back $f^* \omega_Y$ 
of the symplectic form $\omega_Y$ coincides with $\omega_X$. 

$I$-graded vector spaces are always finite dimensional,
and whose subspaces are always $I$-graded.   

\subsection{Preliminaries to quiver variety}\label{2.2}

Take a non-zero $I$-graded vector space $V = \bigoplus_{i\in I} V_i$.   
We denote by $\dim V \in \N^I$ its dimension vector, 
i.e., $\dim V := ( \dim V_i)_{i\in I}$.

Set 
\[
\bM(V) := \bigoplus_{h \in H} \Hom (V_{\vout (h)}, V_{\vin (h)}).
\]
The reductive group $G_V :=\prod_{i\in I} \GL(V_i)$
acts on $\bM(V)$ by 
\[
g\cdot x := \left( 
g_{\vin(h)}x_h g^{-1}_{\vout(h)} \right)  
\quad \text{for}\ \ g=(g_i) \in G_V,\ x=(x_h) \in \bM(V).
\]
We consider $\C^\times$ as a subgroup of $G_V$ by 
\[
\C^\times \ni \lambda \longmapsto (\lambda 1_{V_i})_{i \in I} \in G_V.
\]
Clearly this subgroup acts trivially on $\bM(V)$.

Here we recall the notion of $\theta$-stability introduced by King~\cite{King}. 

Take and fix $\theta =(\theta_i) \in \Q^I$ such that $\theta \cdot \dim V =0$.
For $x \in \bM(V)$,
we say a subspace $S \subset V$ is $x$-{\em invariant} 
if $x_h (S_{\vout(h)}) \subset S_{\vin(h)}$ for all $h \in H$. 

\begin{definition}\label{2.2.1}
We say that a point $x\in \bM(V)$ is {\em $\theta$-semistable} 
if any subspace $S \subset V$ satisfies $\theta \cdot \dim S \leq 0$.
A point $x$ is {\em $\theta$-stable} if the strict inequality holds 
unless $S=0$ or $S=V$.
\end{definition}

Here we have changed the sign convention from \cite{King} 
(this agrees with \cite{Nak-ref}).  
King showed that the above stability condition is 
equivalent to Mumford's stability condition with respect to 
the linearization given by 
the trivial bundle  
with the $G_V$-action determined by the character  
$\chi(g) := \prod_i \det (g_i)^{-m\theta_i}$~(see below), 
where $m$ is any positive integer such that $m\theta \in \Z^I$ 
(note that the condition for $\theta$-(semi)stability 
and the one for $m\theta$-(semi)stability are identical). 

Set
\begin{align*}
\bMss_\theta(V) &:=
\lset{x\in \bM(V)}{\text{$x$ is $\theta$-semistable}}, \\
\bMs_\theta(V) &:=
\lset{x\in \bM(V)}{\text{$x$ is $\theta$-stable}}. 
\end{align*}
Both subsets are $G_V$-invariant and open.

Let $A_\theta(V)$ be the set consisting of regular functions 
$f\in \C[\bM(V)]$ on $\bM(V)$ such that 
\[
f(g\cdot x)=\chi(g)f(x) 
\quad \text{for any} \ (g,x) \in G_V \times \bM(V),
\]
and set $R_\theta(V):=\bigoplus_{n\in \N} A_{n\theta}(V)$. 
Then the variety $\Proj R_\theta(V)$ gives a {\em good quotient} of $\bMss_\theta(V)$;
namely, there is a surjective affine $G$-invariant morphism 
$\varphi \colon \bMss_\theta(V) \to \Proj R_\theta(V)$ such that the induced map 
$\varphi^*\colon \C[U] \to \C[\varphi^{-1}(U)]^G$ 
is an isomorphism for any affine open subset $U \subset \Proj R_\theta(V)$. 
Moreover, a point $x \in \bMss_\theta(V)$ is $\theta$-stable if and only if 
the fiber $\varphi^{-1}(\varphi(x))$ consists of a single $G_V$-orbit 
and its dimension is equal to $\dim G_V/\C^\times$.
In particular $\varphi(\bMs_\theta(V))$ can be identified with 
the set-theoretical orbit space $\bMs_\theta(V)/G_V$. 
By the last statement of the proposition below, $\varpi(\bMs_\theta(V))$ 
is an open subset of $\Proj R_\theta(V)$. 

\begin{remark}\label{2.2.2}
Both $\theta$-stability and $\theta$-semistability  
are purely topological conditions. 
Indeed, let $G_V$ act on $\bM(V) \times \C$ by 
$g \cdot (x,z) := (g \cdot x, \chi(g)^{-1}z)$. 
Then fixing a non-zero $z \in \C$,  
a point $x \in \bM(V)$ is $\theta$-semistable if and only if 
\[
\ov{G_V \cdot (x,z)} \cap \bM(V) \times \{ 0 \} = \emptyset,
\]
and $x$ is $\theta$-stable if and only if $G_V \cdot (x,z)$ 
is closed and its dimension is equal to $\dim G_V/\C^\times$ (see \cite{King}).

Thus if $f\colon \bM(V) \to \bM(V)$ is a $G_V$-equivariant 
homeomorphism in the sense of usual topology, 
$f$ preserves both $\theta$-stability and $\theta$-semistability. 
\end{remark}

We use a standard notation $\GIT$ for good quotient spaces, e.g., 
\[
\bMss_\theta(V) \GIT G_V 
=\Proj R_\theta(V).
\]

A good quotient $\varphi\colon X \to Y$ of a $G$-variety $X$ is called 
a {\em geometric quotient} if the induced map $X/G \to Y$ 
is bijective.   
$\bMs_\theta(V)$ has a geometric 
quotient $\bMs_\theta(V)/G_V$ by restricting $\varphi$.  

Here we introduce several properties of good quotients (see e.g.\ \cite{New}). 

\begin{proposition}\label{2.2.3} 
Let $X$ be a variety acted on by a reductive group $G$.
Suppose that a good quotient $\varphi\colon X \to Y=X\GIT G$ exists. 
 
{\rm (i)} A good quotient $(Y,\varphi)$ is a {\em categorical quotient}; 
namely, $(Y,\varphi)$ has the following universal property:
if $Z$ is a $G$-variety and $f\colon X\to Z$ is a $G$-invariant morphism,
then there exists a unique morphism $\psi \colon Y \to Z$
such that $f=\psi \circ \varphi$.
In particular $Y$ is unique up to isomorphism. 

{\rm (ii)} Two points $x, x' \in X$ have the same image 
$\varphi(x)=\varphi(x')$ if and only if 
the closures of the two orbits intersect;
$\ov{G\cdot x} \cap \ov{G\cdot x'} \neq \emptyset$. 

{\rm (iii)} If $Z \subset X$ is a closed $G$-invariant subset,
then $\varphi(Z)\subset Y$ is closed and 
the restriction $\varphi\colon Z\to \varphi(Z)$ is a good quotient. 

{\rm (iv)} If $U \subset X$ is open and $\varphi$-saturated 
{\rm (}namely, $\varphi^{-1}(\varphi(U))=U${\rm )},
then $\varphi(U) \subset Y$ is open and 
the restriction $\varphi\colon U\to \varphi(U)$ is a good quotient.
\end{proposition}

By the above proposition, 
two points $x, x' \in \bMss_\theta(V)$ have the same image under $\varphi$ 
if and only if 
\[
\ov{G_V \cdot x} \cap \ov{G_V \cdot x'} 
\cap \bMss_\theta(V) \neq \emptyset.
\] 
Since any orbit has a unique closed orbit in its closure (see e.g.\ \cite{Bor}), 
the space $\bMss_\theta(V) \GIT G_V$ parameterizes all 
closed $G_V$-orbits in $\bMss_\theta(V)$, where ``closed'' means ``closed in $\bMss_\theta(V)$''.     
A $\theta$-semistable point $x \in \bMss_\theta(V)$ whose orbit is closed in $\bMss_\theta(V)$ 
is said to be $\theta$-{\em polystable}. 

\begin{proposition}[{\cite[Proposition 3.2]{King}}]\label{2.2.4}
{\rm (i)} A point $x \in \bMss_\theta(V)$ is $\theta$-polystable if and only if 
there is a direct sum decomposition 
\[
V = V^1 \oplus V^2 \oplus \cdots \oplus V^n \quad \text{where}\ \ \theta \cdot \dim V^i =0,  
\]
and a $\theta$-stable point $x^i \in \bMs_\theta(V^i)$ for each $i$ such that 
$x=x^1 \oplus x^2 \oplus \cdots \oplus x^n$, i.e., $x_h$ is the direct sum of $x^i_h$'s as a linear map 
for any $h \in H$.  
  
{\rm (ii)} Every point $x \in \bMss_\theta(V)$ has a filtration 
\[
V = V^0 \supset V^1 \supset \cdots \supset V^N=0
\]
such that each $V^i$ is $x$-invariant, $\theta \cdot \dim V^i=0$ 
and each point $\gr^i x \in \bM(V^i/V^{i+1})$ induced from $x$ is $\theta$-stable. 
Let us set $\gr V := \bigoplus_i V^i/V^{i+1}$ and $\gr x := \bigoplus_i \gr^i x \in \bM(\gr V)$.  
Then under an identification $V \simeq \gr V$, the orbit $G_V \cdot \gr x$ 
is a unique closed orbit contained in $\ov{G_V \cdot x} \cap \bMss_\theta(V)$. 
Here ``closed'' means ``closed in $\bMss_\theta(V)$''.            
\end{proposition}  

We will often write a point in $\bMss_\theta(V)$ like as $[x]$, 
where $x \in \bMss_\theta(V)$ is its representative.

Note that clearly $\bMss_0 (V) = \bM(V)$ and hence 
the quotient $\bMss_0(V)\GIT G_V$ must be equal to the 
affine quotient of $\bM(V)$;
\[
\bM(V)\GIT G_V = \Spec \C[\bM(V)]^{G_V}.
\]
This space parameterizes all closed $G_V$-orbits in $\bM(V)$. 
By $\C[\bM(V)]^{G_V} = A_0(V)$, 
there is a natural projective morphism 
\[
\pi\colon \bMss_\theta(V) \GIT G_V \to \bM(V)\GIT G_V.
\] 
Set-theoretically, $\pi$ sends a point $[x]$ to 
a unique closed orbit in the closure of $G_V \cdot x$.

\begin{proposition}\label{2.2.5}
The restriction $\pi\colon \pi^{-1}(\bMs_0(V)/G_V) \to \bMs_0(V)/G_V$ 
is an isomorphism. 
\end{proposition}

\begin{proof}
By definition we have $\bMs_0(V) \subset \bMs_\theta(V)$. 
Thus the following diagram is commutative:
\[
\begin{CD}
\bMs_\theta(V) @<{\text{inclusion}}<< \bMs_0(V) \\
@VVV @VVV \\
\bMss_\theta(V) \GIT G_V @>{\pi}>> \bM(V)\GIT G_V
\end{CD}
\]
Both of the vertical arrows are the geometric quotients. 
Hence the assertion follows.
\end{proof}  

The following fact is also well-known.

\begin{proposition}\label{2.2.6}
The stabilizer of any $\theta$-stable point $x \in \bMs_\theta(V)$ is equal to $\C^\times$.  
\end{proposition}

\begin{proof}
Suppose that $g=(g_i) \in G_V$ stabilizes $x\in \bMs_\theta(V)$.
Then both $\bigoplus \Im(g_i -\lambda 1_{V_i})$ and $\bigoplus \Ker(g_i -\lambda 1_{V_i})$ 
are $x$-invariant subspaces of $V$ for any $\lambda \in \C^{\times}$.
Take $\lambda$ to be an eigenvalue of $g_i$ for some $i\in I$.
By the stability condition, we have
\begin{equation}
\sum_{i\in I} \theta_i \dim \Im(g_i-\lambda 1_{V_i}) \leq 0, 
\qquad \sum_{i\in I} \theta_i \dim \Ker(g_i -\lambda 1_{V_i}) \leq 0.
\end{equation}
On the other hand, we have
\[
\sum_{i\in I} \theta_i (\dim \Im(g_i -\lambda 1_{V_i}) +\dim \Ker(g_i -\lambda 1_{V_i})) 
=\theta \cdot \dim V = 0.
\]
Thus the previous two inequalities must be equalities.
By the stability condition and 
the choice of $\lambda$, we must have $g=\lambda$.
\end{proof} 

For each subspace $S \subset V$, 
the natural inclusion $\bM(S) \hookrightarrow \bM(V)$ induces 
a morphism
\[
\bM(S)\GIT G_S \to \bM(V)\GIT G_V. 
\]
It is a closed immersion. 
This follows immediately from the following fact.

\begin{proposition}[{\cite[Theorem 1.3]{Lus-quiv}}]\label{2.2.7} 
The invariant subring $\C[\bM(V)]^{G_V}$ is generated by 
functions of the form $\Tr \left(  x_{h_1}x_{h_2} \cdots x_{h_n} \right)$, 
where $(h_1, h_2, \dots ,h_n)$ is a {\rm cycle} in $H$; namely,  
a sequence in $H$ such that    
\[
\vout(h_1) = \vin(h_2), ~\vout(h_2) = \vin(h_3),~\dots ,~\vout(h_{n-1}) = \vin(h_n), 
~\vout(h_n) = \vin(h_1).
\] 
\end{proposition}

\subsection{Quiver variety}\label{2.3}

Let us define the quiver variety.
First we define a map $\mu_V \colon \bM(V) \to \Lie G_V :=\bigoplus \gl(V_i)$ by 
\[
\mu_V (x) := \left(  \sum_{h\in H_i}
\epsilon(h) x_h x_{\ov h} \right)  _{i\in I}.
\]
It is equivariant with respect to the action of $G_V$. 
Thus for a central element $\zeta \in \C^I$ of $\Lie G_V$,
the subset $\mu(\zeta)$ is a $G_V$-invariant
closed subvariety of $\bM(V)$.
\begin{definition}\label{2.3.1}
For a given $(\zeta, \theta) \in \C^I \times \Q^I$ with $\theta \cdot \dim V=0$, 
the {\em quiver variety} is the good quotient
\[
\M_{\zeta,\theta}(V) 
:= \bMss_\theta(V)\cap \mu_V^{-1}(\zeta) \GIT G_V.
\] 
\end{definition}
It is well-defined by Proposition~\ref{2.2.3}.
The equation $\mu_V (x)=\zeta$ is called
the {\em (deformed) preprojective relation}. 

The map $\mu_V$ has a remarkable property which we explain from now.
 
Let $M$ be a smooth variety acted on by 
a reductive algebraic group $G$. 
For $\xi \in \Lie G$, we denote by  
$\xi^*$ the vector field induced from the
infinitesimal action of $\xi$; namely,
\[
\xi^*_x := \left. \frac{d}{dt} \exp (t\xi) \cdot x \,\right|_{t=0}
\quad \text{for}\ x \in M.
\]

\begin{definition}\label{2.3.2}
A {\em Hamiltonian $G$-structure} on $M$ is 
a pair consisting of a $G$-invariant 2-form $\omega$ on $M$ and 
a morphism $\mu\colon M\to (\Lie G)^*$, which is equivariant with respect to  
the coadjoint action on $(\Lie G)^*$, such that: 
\begin{flalign*}
{\rm (H1)}\  & d\omega =  0; & \\
{\rm (H2)}\  & \iota(\xi^*)\omega = d\langle \mu,\xi \rangle
\quad \text{for any}\ \xi \in \Lie G; & \\
{\rm (H3)}\  & \Ker \omega_x= 0
\quad \text{for each}\ x \in M. &
\end{flalign*}
Here $\Ker \omega_x := \{\, v \in T_x M \mid \iota(v)\omega_x =0 \,\}$. 
The triple $(M,\omega,\mu)$ is called a {\em Hamiltonian $G$-space} and 
$\mu$ is called the {\em moment map}.
\end{definition}

We define an algebraic symplectic form $\omega$ on
$\bM(V)$ by
\[
\omega := \sum_{h\in \Omega} \Tr dx_h \wedge dx_{\ov h} 
+\sum_{i\in I}\Tr da_i \wedge db_i.
\] 
Note that $\Lie G_V$ can be identified with its dual by the trace.
It is easy to see that, under this identification, 
the triple $(\bM(V),\omega, \mu)$
is a Hamiltonian $G_V$-space.

Thus the open subvariety
\[
\Mreg_{\zeta,\theta}(V) 
:= \bMs_\theta(V) \cap \mu_V^{-1}(\zeta)/ G_V
\]
of $\M_{\zeta,\theta}(V)$ is an algebraic symplectic manifold by 
the following well-known fact.
\begin{theorem}\label{2.3.3}
Let $(M,\omega,\mu)$ be a Hamiltonian $G$-space 
and $\zeta \in (\Lie G)^*$ be a fixed point with respect to the coadjoint action. 
Suppose that the stabilizer of each point in $\mu^{-1}(\zeta)$ is trivial.   
Then $\mu^{-1}(\zeta)$ is smooth. 
Moreover if a geometric quotient $\mu^{-1}(\zeta)/G$ exists, 
then it becomes an algebraic symplectic manifold,
and for each point $x \in \mu^{-1}(\zeta)$, 
the tangent space of $\mu^{-1}(\zeta)/G$ at the point represented by $x$ 
can be naturally identified with the quotient space $\Ker d_x \mu /T_x(G\cdot x)$.  
\end{theorem}

In the case $\zeta =0$, we will denote 
$\M_\theta(V)=\M_{0,\theta}(V),~\Mreg_\theta(V)=\Mreg_{0,\theta}(V)$. 

\subsection{Quasi-Hamiltonian structure}\label{2.4}

A notion of {\em quasi-Hamiltonian structure},
which was introduced by Alekseev-Malkin-Meinrenken~\cite{AMM} 
for $C^{\infty}$-manifolds with a compact Lie group action,
is a ``multiplicative'' analogue of Hamiltonian structure.  
This subsection is a quick review of its complex algebraic version
which was already treated by Boalch~\cite{Boa3} and Van den Bergh~\cite{V-poi, V-ham}.

Let $G$ be a reductive algebraic group
and $\Lie G$ be its Lie algebra.
For simplicity, 
we assume that $G$ is a closed subgroup of $\GL(N,\C)$ for some $N$,  
and that the symmetric form $\Tr \colon \Lie G \otimes \Lie G \to \C$ 
induced from the trace is non-degenerate.  

We define
\[
\chi :=\frac{1}{6}\Tr\, (g^{-1}dg \wedge g^{-1}dg \wedge g^{-1}dg) 
=\frac{1}{6}\Tr\, (dg\, g^{-1} \wedge dg\, g^{-1} \wedge dg\, g^{-1}),
\]
where $g^{-1}dg$ (resp.\ $dg\, g^{-1}$) 
is the left-invariant (resp.\ right-invariant) Maurer-Cartan form on $G$.

\begin{definition}\label{2.4.1}
A {\em quasi-Hamiltonian $G$-space} is 
a smooth $G$-variety $M$ together with 
a $G$-invariant 2-form $\varpi$ on $M$ and 
a $G$-equivariant morphism $\Phi\colon M\to G$ (where we have let $G$ act on itself 
by the conjugation) such that: 
\begin{flalign*}
{\rm (QH1)}\  & d\varpi =  - \Phi^*\chi; & \\
{\rm (QH2)}\  & \iota(\xi^*) \varpi 
=\tfrac{1}{2}\Tr\, \xi (\Phi^{-1}d\Phi +d\Phi \, \Phi^{-1})
\quad \text{for any}\ \xi \in \Lie G; & \\
{\rm (QH3)}\  & \Ker \varpi_x= \{\, \xi^*_x \mid \xi\in\Ker(\Ad_{\Phi(x)}+1)\,\}
\quad \text{for each}\ x \in M. &
\end{flalign*}
$\Phi$ is called the {\em group-valued moment map}.
\end{definition}

There are two typical examples of quasi-Hamiltonian $G$-space. 

\begin{example}[{\cite[Proposition 3.1]{AMM}}]\label{2.4.2}
Let $\mc{C} \subset G$ be a conjugacy class with
the conjugation action of $G$.
Then there is a quasi-Hamiltonian $G$-structure on
$\mc{C}$ whose group-valued moment map is just the
inclusion $\mc{C} \to G$.
Indeed the 2-form is uniquely determined by the condition (QH2) 
since the action is transitive, 
and one can easily check that it actually exists. 
\end{example}

\begin{example}[{\cite[Proposition 3.2]{AMM}}]\label{2.4.3}
Consider the direct product $G \times G$. 
Let $G\times G$ act on itself by $(g,h) \cdot (a,b) := (gah^{-1}, hbg^{-1})$.  
Define a 2-form $\varpi$ on $G \times G$ by 
\[
\varpi = \frac12 \Tr \left(  a^{-1}da \wedge db \, b^{-1} \right)   
- \frac12 \Tr \left(   b^{-1}db \wedge da \, a^{-1} \right)   . 
\]
Then $\varpi$ together with the map 
\[
G \times G \to G \times G;\qquad (a,b) \mapsto (ab, a^{-1}b^{-1})
\]
gives a quasi-Hamiltonian $G \times G$-structure on $G \times G$.  
\end{example}

Recall that for a Hamiltonian $G$-space and a closed subgroup $K \subset G$, 
the induced $K$-action is also Hamiltonian in a natural way. 
Unfortunately, this is not true for a quasi-Hamiltonian $G$-space in general. 
However, if $G = K \times K$ and $K \subset G$ is the diagonal subgroup, 
then an analogous statement holds.  

\begin{theorem}[{\cite[\S 6]{AMM}}]\label{2.4.4}
Let $(M,\varpi,\Phi=(\Phi_1,\Phi_2,\Psi))$ be a 
quasi-Hamiltonian $G\times G\times K$-space. 
Let $G\times K$ act by the diagonal embedding $(g,k)\mapsto (g,g,k)$. 

{\rm (1)} $M$ with a 2-form 
\[ 
\varpi_{12}:=\varpi+ \frac12 \Tr\, (\Phi_1^{-1}d\Phi_1 \wedge d\Phi_2 \, \Phi_2^{-1})
\]
and a morphism
\[
(\Phi_{12},\Psi) := (\Phi_1\cdot \Phi_2,\Psi)\colon M\to G\times K
\]
is a quasi-Hamiltonian $G\times K$-space {\rm (}This space is called the {\em (internal) fusion}{\rm )}.

{\rm (2)} If we define
\begin{align*}
\varpi_{21} &:=\varpi+ \frac12 \Tr\, (\Phi_2^{-1}d\Phi_2 \wedge d\Phi_1 \, \Phi_1^{-1}), \\
\Phi_{21} &:= \Phi_2\cdot \Phi_1\colon M\to G,
\end{align*}
then $(\varpi_{21},\Phi_{21})$ also defines a 
quasi-Hamiltonian $G\times K$-structure on $M$.
Moreover it is isomorphic to $(M,\varpi_{12},(\Phi_{12},\Psi))$. 

{\rm (3)} Let $(M,\varpi,\Phi)$ be a quasi-Hamiltonian 
$G\times G\times G\times K$-space.
Let $(\varpi_{(12)3},\Phi_{(12)3})$ be 
the quasi-Hamiltonian $G\times K$-structure 
obtained by first fusioning 
the first two $G$-factors, and let $(\varpi_{1(23)},\Phi_{1(23)})$ 
be that obtained by first fusioning the last two $G$-factors. 
Then the two structures coincide.
\end{theorem}

The following theorem is a quasi-Hamiltonian version of Theorem~\ref{2.3.3}, 
which provides a new method to construct algebraic symplectic manifolds.

\begin{theorem}[cf. {\cite[Theorem 5.1]{AMM}}]\label{2.4.5}
Let $(M,\varpi,(\Phi_1,\Phi_2))$ be 
a quasi-Hamiltonian $G_1\times G_2$-space 
and $f$ be a central element of $G_1$.
Suppose that the stabilizer of each point in $\Phi_1^{-1}(f)$ is trivial. 
Then $\Phi_1^{-1}(f)$ is a smooth subvariety of $M$.
Moreover if a geometric quotient $\Phi_1^{-1}(f)/G_1$ exists, 
then $\Phi_1^{-1}(f)/G_1$ becomes a quasi-Hamiltonian $G_2$-space,
and for each point $x \in \Phi_1^{-1}(f)$, 
the tangent space of $\Phi_1^{-1}(f)$ at the point represented by $x$ 
can be naturally identified with the quotient space 
$\Ker d_x \Phi_1/T_x(G_1 \cdot x)$.   
\end{theorem}

Note that if $G_2$ is trivial, then the resulting quotient space 
carries a quasi-Hamiltonian $\{ 1\}$-structure, 
which is nothing but an algebraic symplectic structure.

We will use the following example in the next section. 

\begin{example}[\cite{V-poi, V-ham}]\label{2.4.6}
Let $V,W$ be two $\C$-vector spaces. Set 
\begin{align*}
M &=\Hom (W,V) \oplus \Hom(V,W), \\
M^{\circ} &= \{\, (a,b)\in M \mid \det(1+ab) \neq 0\,\}.
\end{align*}
We define a 2-form $\varpi$ on $M^{\circ}$ by
\[
\varpi = \frac12 \Tr \,(1+ab)^{-1}da \wedge db
-\frac12 \Tr \,(1+ba)^{-1}db \wedge da,
\]
and we define a map $(\phi,\psi)\colon M^{\circ} \to \GL(V)\times \GL(W)$ by
\[
\phi(a,b)=1+ab,\qquad \psi(a,b)=(1+ba)^{-1}.
\]
Then $(M^{\circ},\varpi,\Phi=(\phi,\psi))$ is a quasi-Hamiltonian 
$\GL(V)\times \GL(W)$-space. 
The proof needs a long calculation (see \cite{V-poi}).
We remark that this quasi-Hamiltonian structure is invertible; 
the map $\iota \colon M^{\circ} \to M^{\circ}$ defined by
$\iota (a,b) := (-(1+ab)^{-1}a,b)$
satisfies 
$\iota^* \varpi=-\varpi$ and $\iota^*(\phi,\psi)=(\phi^{-1},\psi^{-1})$. 
It was given by Crawley-Boevey and Shaw~\cite{CS}. 
\end{example}

\section{Multiplicative quiver variety}\label{3}

\subsection{Definition}\label{3.1}

Let us define the main objects in this paper. 
It is motivated by the paper~\cite{CS} of Crawley-Boevey and Shaw, 
who considered a ``multiplicative'' analogue of the preprojective relation. 
We require the stability condition for solutions of this equation 
to obtain a new variety.

Set
\[
\bMM (V) :=\{\, x\in \bM(V) \mid \det (1+x_h x_{\ov h}) \neq 0 \ \text{for all}\ h\in H \,\}.
\]
Since the function $x \mapsto \prod_{h\in H}\det (1+x_h x_{\ov h})$ 
is constant along each $G_V$-orbit,
it is a $G_V$-invariant open subset of $\bM(V)$,
and the intersection $\bMM (V) \cap \bMss_\theta(V)$
is $\varphi$-saturated.

Fix a total order $<$ on $H$.
We define a map $\Phi=(\Phi_i)_{i\in I}\colon \bMM (V) \to G_V$
by
\[
\Phi_i (x) := \prod^{<}_{h\in H_i}
(1+x_h x_{\ov h})^{\epsilon(h)}.
\]
We sometimes write $\Phi=\Phi_{V}$ to emphasize 
the vector space $V$.  
$\Phi_{V}$ is $G_V$-equivariant
with respect to the conjugation.
Hence, for any $q \in (\C^{\times})^I \subset G_V$,
$\Phi_{V}^{-1}(q)$ is a $G_V$-invariant 
closed subvariety of $\bMM(V)$.
Thus by Proposition~\ref{2.2.3},
the subvariety $\bMss_\theta(V) \cap \Phi_{V}^{-1}(q)$ has a good quotient, 
and the subvariety $\bMs_\theta(V) \cap \Phi_{V}^{-1}(q)$ has a geometric quotient.

\begin{definition}\label{3.1.1}
We define 
\[
\MM_{q,\theta}(V) 
:= \left(  \bMss_\theta(V) \cap \Phi_{V}^{-1}(q) \right)   \GIT G_V,
\]
which we call the {\em multiplicative quiver variety}.

We also define
\[
\MMreg_{q,\theta}(V) 
:= \left( \bMs_\theta(V) \cap \Phi_{V}^{-1}(q) \right)   / G_V.
\]
\end{definition}

The equation $\Phi(x) =q$ is called the {\em multiplicative preprojective relation}. 
We also use the following notation:    
\[
\MM_\theta (V) =\MM_{1,\theta}(V), \qquad \MMreg_\theta (V)=\MMreg_{1,\theta}(V).
\]

$\bMM(V)$ has a quasi-Hamiltonian $G_V$-structure.
 
\begin{proposition}[\cite{V-poi, V-ham}]\label{3.1.2}
We define a 2-form $\varpi$ on $\bMM(V)$ by
\[
\begin{split}
\varpi &:= \frac12 \sum_{h  \in H}\epsilon(h )\Tr\,(1+x_h x_{\ov h })^{-1}dx_h 
\wedge dx_{\ov h} \\
&\quad +\frac12 \sum_{h  \in H}\Tr \,\Phi_h^{-1}d\Phi_h 
\wedge d(1+x_h x_{\ov h})^{\epsilon(h)}(1+x_h x_{\ov h})^{-\epsilon(h)},
\end{split} 
\]
where
\[
\Phi_h := \prod_{h' \in H_i ; h' <h}^< 
(1+x_{h'} x_{\ov{h'}})^{\epsilon(h')} 
\quad \text{for}\ h \in H_i.
\]
Then $(\bMM(V),\varpi,\Phi)$ is a 
quasi-Hamiltonian $G_V$-space.
\end{proposition}

This proposition was proved by Van den Bergh as the following. 

\begin{proof}
Set
\[
M_h :=\Hom(V_{\vout(h)},V_{\vin(h)})
\oplus\Hom(V_{\vin(h)},V_{\vout(h)})
\quad \text{for}\ h \in \Omega,
\]
and define $M_h^\circ$ as in~Example~\ref{2.4.6}. 
Then $M_h^{\circ}$ has a quasi-Hamiltonian 
$\GL(V_{\vin(h)})\times \GL(V_{\vout(h)})$-structure whose
group-valued moment map is
\[
(x_h ,x_{\ov h}) \mapsto 
\left(  1+x_h x_{\ov h},(1+x_{\ov h}x_h)^{-1}\right)  .
\]
Taking a direct product, we obtain a quasi-Hamiltonian 
$G$-structure on $\bMM(V)$,
where $G$ is given by
\[
\begin{split}
G &=G_V \times 
\prod_{h  \in \Omega} \GL(V_{\vin(h )}) \times \prod_{h \in \Omega} 
\GL(V_{\vout(h )}) \\
&=G_V \times \prod_{h  \in H} \GL(V_{\vin(h )}). 
\end{split}
\]
Take the internal fusion among the $\GL(V_{\vin(h)})$-factors 
inductively on the total order $<$. 
Then we get a quasi-Hamiltonian 
$G_V \times G_V$-structure on $\bMM(V)$.
Fusioning further the $G_V$-factors, 
we obtain finally the desired structure.   
\end{proof}

Note that the above construction of a quasi-Hamiltonian structure
depends both on the total order $<$ on $H$ and on the orientation $\epsilon$. 
(Here an {\em orientation} is a function $\epsilon' \colon H \to \{ -1,1 \}$ 
satisfying $\epsilon' (\ov{h}) = -\epsilon'(h)$ for all $h\in H$.)
However the following holds.

\begin{proposition}\label{3.1.3}
A quasi-Hamiltonian structure obtained 
by the method in Proposition~\ref{3.1.2} does not
depend on the total order or the orientation
up to isomorphism.
\end{proposition}

\begin{proof}
The assertion follows immediately 
from Theorem~\ref{2.4.4} and the invertible property of $M^{\circ}$ 
mentioned in Example~\ref{2.4.6}.
\end{proof}

It is easy to see that any quasi-Hamiltonian $G_V$-structure 
$(\varpi, \Phi)$ on $\bMM(V)$ naturally 
descends to a quasi-Hamiltonian $G_V /\C^{\times}$-structure 
whose group-valued moment map $\ov \Phi$ is obtained by the composition 
of $\Phi$ with the projection $G_V \to G_V /\C^{\times}$.
Notice that for any $x \in \bMM(V)$, we have 
\[
\prod_i \det \Phi_i (x) 
= \prod_{h\in H_i \cap \Omega} \det (1+x_h x_{\ov h})
\prod_{h\in H_i \cap \ov{\Omega}} \det (1+x_h x_{\ov h})^{-1} 
=1.
\]
Hence if $\Phi_V^{-1}(q) \neq \emptyset$, $q$ must satisfies the equality  
\[
q^{\dim V} := \prod_{i\in I} q_i^{\dim V_i}=1.
\]
Moreover if $q^{\dim V}=1$ then the level set $\Phi_V^{-1}(q)$ 
coincides with $\ov{\Phi}_V^{-1}(q \mod \C^\times))$.  
Thus Theorem~\ref{2.4.5} and dimension count 
implies the following theorem.

\begin{theorem}\label{3.1.4}
$\MMreg_{q,\theta}(V)$ is a pure-dimensional algebraic symplectic manifold,  
and its dimension is $2-(\dim V , \dim V)$.
\end{theorem}

\begin{proof}
Since a quasi-Hamiltonian $\{ 1\}$-structure
is nothing but an algebraic symplectic structure,
the first assertion follows from Proposition~\ref{2.2.6}.
To compute the dimension of $\MMreg_{q,\theta}(V)$, note that
\[
(\dim V ,\dim V) = 2\sum_i \left(  \dim V_i \right)  ^2 - 
\sum_{h \in H} \left(  \dim V_{\vout(h)} \right)   \left(  \dim V_{\vin(h)}\right)  .
\]
Since $\dim \MMreg_{q,\theta}(V) = \dim \bM(V) - 2\dim G_V/\C^\times$, 
the assertion follows immediately.
\end{proof}

Finally we introduce a criterion for the smoothness of $\MM_{q,\theta}(V)$.

Set $\bv := \dim V$ and 
\begin{align*}
R_+ &:= \{\, \alpha \in \N^I \mid (\alpha , \alpha ) \leq 2\,\} 
\setminus \{ 0 \} , \\
R_+(\bv) &:= \{\, \alpha \in R_+ \mid \bv -\alpha \in \N^I\,\},\\
D_\alpha &:= \{\theta \in \Q^I \mid \theta \cdot \alpha =0\,\}, \\
E_\alpha &:= \{ z\in (\C^{\times})^I \mid z^{\alpha} =1\,\}
\quad \text{for}\ \alpha \in R_+.
\end{align*}

\begin{proposition}\label{3.1.5}
{\rm (1)}\ $\MMreg_{q,\theta}(V)$ is empty unless 
$\bv \in R_+$ and $(q,\theta) \in E_{\bv} \times D_{\bv}$. \\
{\rm (2)}\ If
\[
(q,\theta) \in E_{\bv} \times D_{\bv} \setminus 
\bigcup_{\alpha \in R_+(\bv)\setminus \{ \bv \} } E_\alpha \times D_\alpha ,
\]
then $\MMreg_{q,\theta}(V) = \MM_{q,\theta}(V)$. 
\end{proposition}

\begin{proof}
We have already proved the first assertion. 
Suppose that there exists a point $x \in \Phi_{V}^{-1}(q)$ 
which is $\theta$-semistable but not $\theta$-stable.
Then we can find a non-zero proper $x$-invariant subspace $S \subset V$
such that $\theta \cdot \dim S =0$.
We may assume that $S$ is minimal amongst all non-zero subspaces satisfying such conditions.
Set $\alpha :=(\dim S_i )$. Then $\theta \in D_{\alpha}$.
Let $x' \in \bM(S)$ be the element obtained by 
the restriction of $x$ to $S$. 
Then $\Phi_S (x) =q$ and hence $q \in E_{\alpha}$.
If $T \subset S$ is $x'$-invariant, then $\theta \cdot \dim T \leq 0$ by 
the $\theta$-semistability of $x$, 
and moreover if $\theta \cdot \dim T =0$, then 
$T=0$ or $T=S$ by the choice of $S$.
Thus $x'$ is $\theta$-stable. 
In particular $\MMreg_{q,\theta}(S)$ is non-empty, so
\[
0 \leq \dim \MMreg_{q,\theta}(S) = 2-( \alpha , \alpha).
\] 
Thus $\alpha \in R_+ (\bv)$.  
\end{proof}

\subsection{Some properties}\label{3.2}

By Proposition~\ref{3.1.3}, 
we may assume the following:
\begin{quote}
The total order $<$ satisfies $h < h'$ for any $h \in \Omega$ and $h'\in \ov\Omega$.
\end{quote}
Then we can decompose $\Phi_i=\Phi_i^+ (\Phi_i^-)^{-1}$, 
where
\begin{align*}
\Phi_i^+ (x) &:= \prod^<_{h\in H_i \cap \Omega} (1+x_h x_{\ov h}), \\
\Phi_i^- (x) &:= \prod^>_{h\in H_i \cap \ov\Omega} (1+ x_h x_{\ov h}).
\end{align*}
Thus the multiplicative preprojective relation $\Phi_i(x)=q_i$ at $i\in I$ 
is equivalent to 
\begin{equation}\label{3.2.1}
\Phi_i^+ (x) - q_i \Phi_i^-(x) =0.
\end{equation}
Here $\Phi_i^\pm$ is expanded as 
\begin{align*}
\Phi_i^+ &= 1 + \sum_{h\in H_i \cap \Omega} \Phi_h^+ x_h x_{\ov h}, \\
\Phi_i^- &= 1 + \sum_{h\in H_i \cap \ov\Omega} x_h x_{\ov h} \Phi_h^-, 
\end{align*}
where
\begin{align*}
\Phi_h^+ (x) &:= 
\prod^<_{\substack{h'\in H_i \cap \Omega ;\\ h'<h}}
(1+x_{h'} x_{\ov{h'}}),\\
\Phi_h^- (x) &:= \prod^>_{\substack{h'\in H_i \cap \ov\Omega ;\\ h'<h}}
(1+x_{h'} x_{\ov{h'}}) \quad \text{for}\ h\in H_i. 
\end{align*}
Thus \eqref{3.2.1} is equivalent to 
\begin{equation}\label{3.2.2} 
\sum_{h\in H_i \cap \Omega} \Phi_h^+ x_h x_{\ov h}
-q_i \sum_{h\in H_i \cap \ov\Omega} x_h x_{\ov h} \Phi_h^- = q_i -1.
\end{equation}

Set $\widehat{V}_i := \bigoplus_{h \in H_i} V_{\vout(h)}$. 
For $h \in H_i$, 
let $\iota_h \colon V_{\vout(h)} \to \widehat{V}_i$ be the natural inclusion 
and $\pi_h \colon \widehat{V}_i \to V_{\vout(h)}$ be the projection.  
We define  
\begin{align*}
\sigma_i(x) 
&:=\sum_{h\in H_i \cap \Omega} \iota_h x_{\ov h}  
+ \sum_{h \in H_i \cap \ov\Omega} \iota_h x_{\ov h} \Phi_h^-  
\colon V_i \to \widehat{V}_i, \\
\tau_i(x) 
&:=\sum_{h \in H_i\cap \Omega} \Phi^+_h x_h \pi_h  
- q_i \sum_{h \in H_i\cap \ov\Omega} x_h \pi_h 
\colon \widehat{V}_i \to V_i.
\end{align*}
Then by \eqref{3.2.2}, 
the multiplicative preprojective relation at $i\in I$
is equivalent to $\tau_i \sigma_i = q_i -1$.
In particular, the sequence 
\[
\begin{CD}
V_i @>{\sigma_i}>> \widehat{V}_i @>{\tau_i}>> V_i
\end{CD}
\]
is a complex if $\Phi_i(x) =1$.

\begin{lemma}\label{3.2.3}
Take $x\in \bMM(V)$ and $i\in I$. 
Suppose that a subspace $S \subset V$
satisfies: 
\begin{enumerate}
\item[\rm (i)] $\sigma_i (S_i) \subset \widehat{S}_i$; and
\item[\rm (ii)] $\tau_i ( \widehat{S}_i ) \subset S_i$. 
\end{enumerate}
Then $x_h (S_{\vout(h)}) \subset S_{\vin(h)}$ for $h \in H_i \cup \ov{H_i}$.
\end{lemma}

\begin{proof}
Suppose that $S \subset V$ satisfies the conditions (i) and (ii).
The condition (i) means
\begin{align}
&x_{\ov h} (S_i) \subset S_{\vout(h)} \quad \text{for}\ h\in H_i \cap \Omega ,
\label{3.2.4} \\
&x_{\ov h}\Phi_h^- (S_i) \subset S_{\vout(h)} \quad \text{for}\ h\in H_i \cap \ov\Omega , 
\label{3.2.5}
\intertext{and the condition (ii) means}
&\Phi_h^+ x_h (S_{\vout(h)}) \subset S_i \quad \text{for}\ h\in H_i \cap \Omega ,
\label{3.2.6} \\
&x_h (S_{\vout(h)}) \subset S_i \quad \text{for}\ h\in H_i \cap \ov\Omega .
\label{3.2.7}
\end{align}
Let $h \in H_i \cap \Omega$ be  
the minimum element in $H_i \cap \Omega$ with respect to $<$.
Then $\Phi_h^+=1$ and hence $x_h (S_{\vout(h)}) \subset S_i$
by \eqref{3.2.6}.
Thus $(1+x_{h}x_{\ov{h}})(S_i) \subset S_i$ by \eqref{3.2.4}, and hence 
one can use induction on $<$ 
to obtain $x_h (S_{\vout(h)}) \subset S_{\vin(h)}$ for all $h\in H_i \cap \Omega$. 

Similarly, if we denote by $h' \in H_i \cap \ov\Omega$ the 
minimum element in $H_i \cap \ov\Omega$, then $\Phi_{h'}^- =1$ and hence 
$x_{\ov{h'}}(S_i) \subset S_{\vout(h')}$ by \eqref{3.2.5}.
Thus $(1+x_{h'}x_{\ov{h'}})(S_i) \subset S_i$ by \eqref{3.2.7}, and hence 
one can use induction again to obtain $x_{\ov{h}}(S_i) \subset S_{\vout(h)}$ 
for all $h\in H_i \cap \ov\Omega$. 
\end{proof}

\begin{proposition}\label{3.2.8}
Take $x \in \bMM(V)\cap \bMs_\theta(V)$ and $i\in I$. 
Suppose that $\dim V \neq \mbe_i$. 

{\rm (i)} If $\theta_i \geq 0$, then $\sigma_i$ is injective. 

{\rm (ii)} If $\theta_i \leq 0$, then $\tau_i$ is surjective.
\end{proposition}

\begin{proof}
Suppose that $\theta_i \geq 0$.
Set
\[
S_j = \begin{cases} 0 & \text{if}\ j\neq i, \\
\Ker \sigma_i & \text{if}\ j=i.
\end{cases}
\]
By Lemma~\ref{3.2.3}, $S$ is $x$-invariant.  
Hence we have $\theta \cdot \dim S \leq 0$ by the stability condition.
However 
\[
\theta \cdot \dim S = \theta_i \dim \Ker \sigma_i \geq 0,
\]
so we must have $\theta \cdot \dim S =0$. 
Thus $S=0$ or $S=V$ by the stability condition again. 
If $S=0$ we are done. So assume that $S=V$. 
Then $V_j =0$ for all $j\neq i$, and hence $x=0$. 
Thus any subspace of $V$ is $x$-invariant, 
and hence $V_i =\C$ by the stability condition again. 
This contradicts. 
  
Next suppose that $\theta_i \leq 0$.
Set
\[
T_j = \begin{cases} V_j & \text{if}\ j\neq i, \\
\Im \tau_i & \text{if}\ j=i.
\end{cases}
\]
By Lemma~\ref{3.2.3}, $T$ is $x$-invariant.
Hence we have $\theta \cdot \dim T \leq 0$ by the stability condition. 
However we have also
\[
\theta \cdot \dim T = \theta \cdot \dim V - \theta_i \dim \Coker \tau_i \geq \theta \cdot \dim V =0.
\]
Thus $\theta \cdot \dim T = 0$, which implies that $T=0$ or $T=V$. 
If $T=V$ we are done. If $T=0$ one can deduce a contradiction as above.
\end{proof}

\subsection{Singularity at the origin; relation to the quiver variety}\label{3.3}

In this subsection we work in the complex analytic category. 
The following proposition implies that the
singularity at the origin of $\Phi_V^{-1}(1)$  
and that of $\mu_V^{-1}(0)$ are the same. 
Recall that we let $\varphi \colon \bM(V) \to \bM(V) \GIT G_V$ be 
the quotient morphism.  

\begin{proposition}\label{3.3.1}
There are $\varphi$-saturated open neighborhoods 
$\mc{U}$, $\mc{U}'$ of $0 \in \bM(V)$,  
and a $G_V$-equivariant 
biholomorphic map $f\colon \mc{U} \to \mc{U}'$ such that
\[
f(0)=0, \qquad f(\Phi_V^{-1}(1)\cap \mc{U})=\mu_V^{-1}(0) \cap \mc{U}',
\qquad (f^* \omega - \varpi) |_{\Ker d\Phi_V} =0.
\] 
\end{proposition}

\begin{proof}
We use the following result of 
Alekseev-Malkin-Meinrenken~(they proved it in the case that $G$ is a compact Lie group, 
but the proof can be extended immediately to the case of complex reductive group).

\begin{lemma}[{\cite[Lemma 3.3]{AMM}}]\label{3.3.2}
Let $G \subset \GL(N,\C)$ be a complex reductive Lie group. 
For $s\in [0,1]$, let $\exp_s \colon \Lie G \to G$ denote a map given by 
$\exp_s(\xi) := \exp (s\xi)$. 
Define a holomorphic 2-form $\rho$ on $G$ by  
\[
\rho := \frac12 \int_0^1 ds\, \Tr \left[ \exp_s^* (dg\, g^{-1}) \wedge 
\frac{\partial}{\partial s} \exp_s^* (dg\, g^{-1}) \right] .
\] 
Then $\rho$ is $G$-invariant and satisfies 
\[
d\rho =-\exp^* \chi, \quad \iota(\xi^*)\rho 
= -d \Tr (\xi \, \cdot ) + \frac12 \exp^* \Tr \xi (g^{-1}dg + dg\, g^{-1} ).
\] 
\end{lemma}

Using this, they observed that if $(M,\varpi,\Phi)$ is a quasi-Hamiltonian $G$-space 
and $\Phi(M)$ is contained in an open subset of $G$ on which an 
$G$-equivariant right-inverse $\log$ of $\exp$ exists, 
then the triple $(M,\varpi -\Phi^* \log^* \rho, \log \circ \Phi)$ satisfies the conditions 
(H1) and (H2)~\cite[Remark 3.2]{AMM} as follows:   
\begin{align*}
d(\varpi -\Phi^* \log^* \rho ) 
&= -\Phi^* \chi + \Phi^* \log^* \exp^* \chi =0,\\
\iota(\xi^*) (\varpi - \Phi^* \log^* \rho )
&= \frac12 \Phi^* \Tr \xi (g^{-1}dg + dg\, g^{-1}) + \Phi^* \log^* d\Tr (\xi \, \cdot ) \\
&\quad -\frac12 \Phi^* \log^* \exp^* \Tr \xi (g^{-1}dg + dg\, g^{-1}) \\
&= \Tr (d(\log \Phi)(\cdot )\, \xi).
\end{align*}
In fact, we can always find a $G$-invariant open neighborhood   
$O$ of 0 in $\Lie G$ such that 
the restriction $\exp \colon O \to \exp(O)$ has the inverse $\log := (\exp)^{-1}$. 
Clearly we can take $O$ to be saturated with respect to the quotient map 
$\Lie G \to (\Lie G) \GIT G$. 
Then we can apply the above fact to $(\Phi^{-1}(O),\varpi,\Phi)$.  
  
Let us back to our situation. 
First take a $\varphi$-saturated open subset $\mc{U}$ to be such that 
for any $x \in \mc{U}$ and $h\in H$, 
$1+x_h x_{\ov h} \in \exp (O)$, 
where $O \subset \Lie \GL(V_{\vin(h)})$ is the subset taken as above. 
Then $(\mc{U},\varpi -\Phi_V^* \log^* \rho, \log \circ \Phi_V)$ satisfies (H1) and (H2). 
Since $x_h x_{\ov h}=0$ and $d (1+x_h x_{\ov h}) =0$ at the origin, 
we have $(\Phi_V^* \log^* \rho )_0 =0$ and 
\begin{align*}
\varpi_0 &= \frac12 \sum_{h  \in H}\epsilon(h )\Tr dx_h \wedge dx_{\ov h} 
 +\frac12 \sum_{h  \in H}\Tr \,d\Phi_h \wedge d(1+x_h x_{\ov h})^{\epsilon(h)} \\
 &=\frac12 \sum_{h  \in H}\epsilon(h )\Tr dx_h \wedge dx_{\ov h} = \omega_0.
\end{align*}  
Thus the 2-form $\varpi -\Phi_V^* \log^*\rho$ coincides 
with the symplectic form $\omega$ at the origin. 
By the equivariant Darboux theorem~(see Remark~\ref{3.3.3} below), 
taking $\mc{U}$ to be small enough if necessary, 
there is a $G_V$-equivariant biholomorphic map $f \colon \mc{U} \to \mc{U}'$ 
to some $\varphi$-saturated open neighborhood $\mc{U}'$ such that 
\[
f(0)=0, \qquad f^* \omega = \varpi -\Phi_V^* \log^*\rho, \qquad \mu_V \circ f = \log \circ \Phi_V.
\]
This gives the desired map since the form $\Phi_V^* \log^*\rho$ vanishes on 
$\Ker d\Phi_V$.
\end{proof}

\begin{remark}\label{3.3.3}
The equivariant Darboux theorem asserts for $C^{\infty}$-manifolds with a compact Lie group action. 
However we can generalize this theorem to our case  
as the following~(This is due to Nakajima. See \cite{Nak-AMS}).

It is easy to see that the equivariant Darboux theorem can be generalized for 
complex manifolds with a compact Lie group action. 
Thus, in order to show our claim, 
we first apply the theorem for the maximal compact subgroup 
$U_V := \prod \operatorname{\rm U}(V_i) \subset G_V$. 
Then there is an open ball $B \in \bM(V)$ centered at the origin and 
a $U_V$-equivariant open embedding $f \colon B \to \bM(V)$ 
such that $f(0)=0$.  
By \cite[Proposition 1.4, Lemma 1.14]{Sja}, 
we can extend uniquely this map to a $G_V$-equivariant open embedding 
$f\colon G_V \cdot B \to \bM(V)$. 
We claim that $G_V \cdot B$ is $\varphi$-saturated. This can be proved using 
the map $F_{\infty}\colon \bM(V) \to \bM(V)$ 
introduced in \cite{Sja}, which is $U_V$-equivariant and 
has the following property: 
for any point $x \in \bM(V)$, $F_\infty$ maps $G_V \cdot x$ onto 
$U_V \cdot y$, where $y$ is a point whose $G_V$-orbit is 
a unique closed orbit in $\ov{G_V \cdot x}$.  
Moreover $F_\infty (G_V \cdot B) \subset B$ and there is a continuous family $\{ F_t \}$ 
of diffeomorphisms whose limit is $F_\infty$ and the differential  
$dF_t(x)/dt |_{t=0}$ at any $x$ is tangent to the orbit $G_V \cdot x$.      
Thus if $x, x' \in \bM(V)$ have the same image under $\varphi$ and $x \in G_V \cdot B$, 
then $F_\infty  (x') \in U_V \cdot F_\infty (x) \subset B$, 
and hence $F_t(x') \in B$ for sufficiently large $t$. 
Thus $x' \in G_V \cdot B$. Hence $G_V \cdot B$ is $\varphi$-saturated.

We take some open ball $B'$ in $f(G_V \cdot B)$ centered at the origin, 
and set $\mc{U} :=f^{-1}(G_V \cdot B')$.  
Then $\mc{U}$ is also $\varphi$-saturated. 
To see this, suppose $x' \in \bM(V)$ is in the orbit closure $\ov{G_V \cdot x}$ of 
some $x \in \mc{U}$. 
Then $x' \in G_V \cdot B$ by the above argument, and 
$f(x') \in \ov{G_V \cdot f(x)}$ since $f$ is continuous. 
Since $G_V \cdot B'$ is $\varphi$-saturated and $f(x) \in G_V \cdot B'$, 
we see that $f(x') \in G_V \cdot B'$. Thus $x' \in \mc{U}$. 

Setting $\mc{U}' := G_V \cdot B'$, 
we obtain a desired map $f \colon \mc{U} \to \mc{U}'$. 
\end{remark}

Recall the projective morphisms $\pi \colon \MM_\theta(V) \to \MM_0(V)$ 
and $\pi \colon \M_\theta(V) \to \M_0(V)$. 

\begin{corollary}\label{3.3.4}
There exist an open neighborhood $U$ {\rm (}resp.\ $U'${\rm )} of $[0] \in \MM_0(V)$ 
{\rm (}resp.\ $[0] \in \M_0(V)${\rm )} 
and a commutative diagram 
\[
\begin{CD}
\MM_{\theta}(V) \supset @.\, \pi^{-1}(U) @>{\tilde f}>> \pi^{-1}(U')\, @. \subset \M_{\theta}(V) \\
@. @V{\pi}VV @V{\pi}VV @. \\
@. U @>{f}>> U' @. 
\end{CD}
\]
such that: 
\begin{enumerate}
\item[\rm (i)] $f([0])=[0]$; 
\item[\rm (ii)] both $\tilde{f}$ and $f$ are complex analytic isomorphisms;  
\item[\rm (iii)] $\tilde f$ maps $\pi^{-1}(U) \cap \MMreg_\theta(V)$ onto 
$\pi^{-1}(U') \cap \Mreg_\theta(V)$ as a symplectic biholomorphic map; and 
\item[\rm (iv)] if $x \in \varphi^{-1}(U)$ and $y \in \varphi^{-1}(U')$ 
have closed orbits and $f([x])=[y]$, then the stabilizers of the two are conjugate. 
Thus $f$ preserves the orbit-type.        
\end{enumerate} 
\end{corollary}

\begin{proof}
Since both $\mc{U}$ and $\mc{U}'$ are $\varphi$-saturated and $f \colon \mc{U} \to \mc{U}'$ is 
a biholomorphic map, $f$ sends a closed orbit to a closed orbit and 
a stable/semistable point to a stable/semistable point~(see Remark~\ref{2.2.2}). 
So the result follows.
\end{proof}

The fiber $\pi^{-1}([0]) \subset \MM_\theta(V)$ 
is called the {\em nilpotent subvariety}. 
The above corollary implies that the nilpotent subvarieties of 
the quiver variety and the multiplicative quiver variety are 
complex analytically isomorphic.   

\section{Moduli of filtered local systems and star-shaped quiver}\label{4}

This section is devoted to the study in the case of star-shaped quivers. 
In particular, we prove Theorem~\ref{1.2}. 

\subsection{Star-shaped quiver}\label{4.1} 

Suppose that conjugacy classes 
$\mc{C}_1, \dots \mc{C}_n$ in $\gl(r,\C)$ 
for a fixed $r >0$ are given. 
Choose $A_i \in \mc{C}_i$ 
and take $\xi_{i,j} \in \C^{\times}$
which satisfies
\[
(A_i-\xi_{i,0})(A_i-\xi_{i,1}) \cdots (A_i-\xi_{i,r}) =0.
\]
Set 
\[
v_{i,j} = \rank (A_i-\xi_{i,0}) \cdots (A_i-\xi_{i,j-1}),\qquad
l_i = \min \{ j; v_{i,j}>0\}. 
\]
Note that each $v_{i,j}$ does not depend on a choice of $A_i$. 

Following Crawley-Boevey, we associate to $\mc{C}_1, \dots \mc{C}_n$ 
the following quiver $(I,\Omega)$:     

\vspace{10pt}

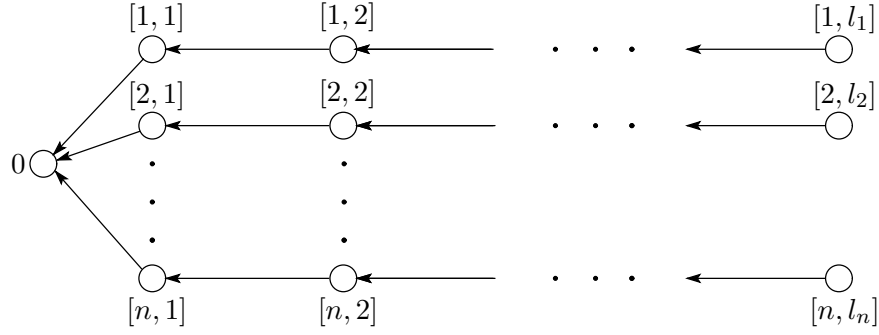
\begin{center}
\unitlength 0.1in
\begin{picture}( 45.0000, 15.4500)( 11.1000,-17.0000)
%
\special{pn 8}%
\special{ar 1376 1010 70 70  0.0000000 6.2831853}%
%
\special{pn 8}%
\special{ar 1946 410 70 70  0.0000000 6.2831853}%
%
\special{pn 8}%
\special{ar 2946 410 70 70  0.0000000 6.2831853}%
%
\special{pn 8}%
\special{ar 5540 410 70 70  0.0000000 6.2831853}%
%
\special{pn 8}%
\special{ar 1946 810 70 70  0.0000000 6.2831853}%
%
\special{pn 8}%
\special{ar 2946 810 70 70  0.0000000 6.2831853}%
%
\special{pn 8}%
\special{ar 5540 810 70 70  0.0000000 6.2831853}%
%
\special{pn 8}%
\special{ar 1946 1610 70 70  0.0000000 6.2831853}%
%
\special{pn 8}%
\special{ar 2946 1610 70 70  0.0000000 6.2831853}%
%
\special{pn 8}%
\special{ar 5540 1610 70 70  0.0000000 6.2831853}%
%
\special{pn 8}%
\special{pa 1890 1560}%
\special{pa 1440 1050}%
\special{fp}%
\special{sh 1}%
\special{pa 1440 1050}%
\special{pa 1470 1114}%
\special{pa 1476 1090}%
\special{pa 1500 1088}%
\special{pa 1440 1050}%
\special{fp}%
%
\special{pn 8}%
\special{pa 2870 410}%
\special{pa 2020 410}%
\special{fp}%
\special{sh 1}%
\special{pa 2020 410}%
\special{pa 2088 430}%
\special{pa 2074 410}%
\special{pa 2088 390}%
\special{pa 2020 410}%
\special{fp}%
%
\special{pn 8}%
\special{pa 3720 410}%
\special{pa 3010 410}%
\special{fp}%
\special{sh 1}%
\special{pa 3010 410}%
\special{pa 3078 430}%
\special{pa 3064 410}%
\special{pa 3078 390}%
\special{pa 3010 410}%
\special{fp}%
\special{pa 3730 410}%
\special{pa 3010 410}%
\special{fp}%
\special{sh 1}%
\special{pa 3010 410}%
\special{pa 3078 430}%
\special{pa 3064 410}%
\special{pa 3078 390}%
\special{pa 3010 410}%
\special{fp}%
%
\special{pn 8}%
\special{pa 2870 810}%
\special{pa 2020 810}%
\special{fp}%
\special{sh 1}%
\special{pa 2020 810}%
\special{pa 2088 830}%
\special{pa 2074 810}%
\special{pa 2088 790}%
\special{pa 2020 810}%
\special{fp}%
%
\special{pn 8}%
\special{pa 2870 1610}%
\special{pa 2020 1610}%
\special{fp}%
\special{sh 1}%
\special{pa 2020 1610}%
\special{pa 2088 1630}%
\special{pa 2074 1610}%
\special{pa 2088 1590}%
\special{pa 2020 1610}%
\special{fp}%
%
\special{pn 8}%
\special{pa 3730 810}%
\special{pa 3020 810}%
\special{fp}%
\special{sh 1}%
\special{pa 3020 810}%
\special{pa 3088 830}%
\special{pa 3074 810}%
\special{pa 3088 790}%
\special{pa 3020 810}%
\special{fp}%
\special{pa 3740 810}%
\special{pa 3020 810}%
\special{fp}%
\special{sh 1}%
\special{pa 3020 810}%
\special{pa 3088 830}%
\special{pa 3074 810}%
\special{pa 3088 790}%
\special{pa 3020 810}%
\special{fp}%
%
\special{pn 8}%
\special{pa 3730 1610}%
\special{pa 3020 1610}%
\special{fp}%
\special{sh 1}%
\special{pa 3020 1610}%
\special{pa 3088 1630}%
\special{pa 3074 1610}%
\special{pa 3088 1590}%
\special{pa 3020 1610}%
\special{fp}%
\special{pa 3740 1610}%
\special{pa 3020 1610}%
\special{fp}%
\special{sh 1}%
\special{pa 3020 1610}%
\special{pa 3088 1630}%
\special{pa 3074 1610}%
\special{pa 3088 1590}%
\special{pa 3020 1610}%
\special{fp}%
%
\special{pn 8}%
\special{pa 5466 410}%
\special{pa 4746 410}%
\special{fp}%
\special{sh 1}%
\special{pa 4746 410}%
\special{pa 4812 430}%
\special{pa 4798 410}%
\special{pa 4812 390}%
\special{pa 4746 410}%
\special{fp}%
%
\special{pn 8}%
\special{pa 5466 810}%
\special{pa 4746 810}%
\special{fp}%
\special{sh 1}%
\special{pa 4746 810}%
\special{pa 4812 830}%
\special{pa 4798 810}%
\special{pa 4812 790}%
\special{pa 4746 810}%
\special{fp}%
%
\special{pn 8}%
\special{pa 5466 1610}%
\special{pa 4746 1610}%
\special{fp}%
\special{sh 1}%
\special{pa 4746 1610}%
\special{pa 4812 1630}%
\special{pa 4798 1610}%
\special{pa 4812 1590}%
\special{pa 4746 1610}%
\special{fp}%
%
\special{pn 8}%
\special{pa 1880 840}%
\special{pa 1450 990}%
\special{fp}%
\special{sh 1}%
\special{pa 1450 990}%
\special{pa 1520 988}%
\special{pa 1500 972}%
\special{pa 1506 950}%
\special{pa 1450 990}%
\special{fp}%
%
\special{pn 8}%
\special{pa 1900 460}%
\special{pa 1430 960}%
\special{fp}%
\special{sh 1}%
\special{pa 1430 960}%
\special{pa 1490 926}%
\special{pa 1468 922}%
\special{pa 1462 898}%
\special{pa 1430 960}%
\special{fp}%
%
\special{pn 8}%
\special{sh 1}%
\special{ar 1946 1010 10 10 0  6.28318530717959E+0000}%
\special{sh 1}%
\special{ar 1946 1210 10 10 0  6.28318530717959E+0000}%
\special{sh 1}%
\special{ar 1946 1410 10 10 0  6.28318530717959E+0000}%
\special{sh 1}%
\special{ar 1946 1410 10 10 0  6.28318530717959E+0000}%
%
\special{pn 8}%
\special{sh 1}%
\special{ar 4056 410 10 10 0  6.28318530717959E+0000}%
\special{sh 1}%
\special{ar 4266 410 10 10 0  6.28318530717959E+0000}%
\special{sh 1}%
\special{ar 4456 410 10 10 0  6.28318530717959E+0000}%
\special{sh 1}%
\special{ar 4456 410 10 10 0  6.28318530717959E+0000}%
%
\special{pn 8}%
\special{sh 1}%
\special{ar 4056 810 10 10 0  6.28318530717959E+0000}%
\special{sh 1}%
\special{ar 4266 810 10 10 0  6.28318530717959E+0000}%
\special{sh 1}%
\special{ar 4456 810 10 10 0  6.28318530717959E+0000}%
\special{sh 1}%
\special{ar 4456 810 10 10 0  6.28318530717959E+0000}%
%
\special{pn 8}%
\special{sh 1}%
\special{ar 4056 1610 10 10 0  6.28318530717959E+0000}%
\special{sh 1}%
\special{ar 4266 1610 10 10 0  6.28318530717959E+0000}%
\special{sh 1}%
\special{ar 4456 1610 10 10 0  6.28318530717959E+0000}%
\special{sh 1}%
\special{ar 4456 1610 10 10 0  6.28318530717959E+0000}%
\put(19.7000,-2.4500){\makebox(0,0){$[1,1]$}}%
\put(29.7000,-2.4000){\makebox(0,0){$[1,2]$}}%
\put(55.7000,-2.5000){\makebox(0,0){$[1,l_1]$}}%
\put(19.7000,-6.5500){\makebox(0,0){$[2,1]$}}%
\put(29.7000,-6.4500){\makebox(0,0){$[2,2]$}}%
\put(55.7000,-6.5500){\makebox(0,0){$[2,l_2]$}}%
\put(19.7000,-17.8500){\makebox(0,0){$[n,1]$}}%
\put(29.7000,-17.8500){\makebox(0,0){$[n,2]$}}%
\put(55.7000,-17.8500){\makebox(0,0){$[n,l_n]$}}%
\put(12.4500,-10.1000){\makebox(0,0){$0$}}%
%
\special{pn 8}%
\special{sh 1}%
\special{ar 2950 1010 10 10 0  6.28318530717959E+0000}%
\special{sh 1}%
\special{ar 2950 1210 10 10 0  6.28318530717959E+0000}%
\special{sh 1}%
\special{ar 2950 1410 10 10 0  6.28318530717959E+0000}%
\special{sh 1}%
\special{ar 2950 1410 10 10 0  6.28318530717959E+0000}%
\end{picture}%
\end{center}

\vspace{10pt}

\noindent
Such a quiver is called a {\em star-shaped} quiver.
We denote the vertex set by $I =\{ 0 \}\cup \{ [i,j] \}$ as in the picture,
and set $I_0 := I \setminus \{ 0 \}$.
We define an $I$-graded vector space $V$ by 
\[
V_0 := \C^r, \qquad V_{i,j} := \C^{v_{i,j}} \quad \text{for}\ [i,j] \in I_0, 
\]
and use the convention $V_{i,0}=V_0$ and $V_{i,l_i +1}=0$. 
For an element $x\in \bM(V)$ we will denote its components by
$a_{i,j} \in \Hom (V_{i,j+1},V_{i,j})$, $b_{i,j} \in \Hom (V_{i,j}, V_{i,j+1})$  
and write simply as $x=(a,b)$. 

The following proposition was proved by Crawley-Boevey (and Shaw). 
 
\begin{proposition}\label{4.1.1}
{\rm (i)}\ Define 
\[
\zeta_0 := -\sum_{i=1}^n \xi_{i,0}, \qquad \zeta_{i,j}:= \xi_{i,j-1}-\xi_{i,j}
\quad \text{for}\ [i,j] \in I_0.
\] 
Then the morphism
\begin{align*}
\M_{\zeta,0}(V) &\to
\{\, (B_1,\dots ,B_n) \in \ov{\mc{C}_1} \times \cdots \times \ov{\mc{C}_n} \mid 
B_1 + \cdots + B_n =0 \,\} \GIT \GL(r,\C), \\
x=(a,b) &\mapsto B_i=\xi_{i,0}+ a_{i,0}b_{i,0}
\end{align*}
is an isomorphism. 
Moreover, the variety of the right hand side includes   
\[
\mc{Q}:=\{\, (B_1,\dots ,B_n) \in (\mc{C}_1 \times \cdots \times \mc{C}_n)^{\rm irr} \mid  
B_1 + \cdots + B_n =0 \,\} / \GL(r,\C)
\]
as the image of $\Mreg_{\zeta, 0}(V)$ 
under the above map. Here, $(\mc{C}_1 \times \cdots \times \mc{C}_n)^{\rm irr}$ 
denotes the set consisting of all 
$(B_1, \dots ,B_n) \in \mc{C}_1 \times \cdots \times \mc{C}_n$ 
such that there is no non-zero proper subspace 
$S \subset \C^r$ which is preserved by $B_i$ for any $i$. 

{\rm (ii)}\ Suppose that each $\xi_{i,j}$ is non-zero. 
Define $q \in (\C^{\times})^I$ by 
\[
q_0 :=\prod_i \xi_{i,0}^{-1}, \qquad q_{i,j}:= \frac{\xi_{i,j-1}}{\xi_{i,j}}
\quad \text{for}\ [i,j] \in I_0.
\]  
Then the morphism
\begin{align*}
\MM_{q,0}(V) &\to
\{\, (B_1,\dots ,B_n) \in \ov{\mc{C}_1} \times \cdots \times \ov{\mc{C}_n} \mid 
B_1 \cdots B_n =1\,\} \GIT \GL(r,\C), \\
x=(a,b) &\mapsto B_i=\xi_{i,0}(1+a_{i,0}b_{i,0})
\end{align*}
is an isomorphism. Moreover, the variety of the right hand side includes   
\[
\mc{R}:=\{\, (B_1,\dots ,B_n) \in (\mc{C}_1 \times \cdots \times \mc{C}_n)^{\rm irr} \mid 
B_1 \cdots B_n =1 \,\} / \GL(r,\C)
\]
as the image of $\MMreg_{q, 0}(V)$ under the above map. 
\end{proposition}

\begin{proof}
For a proof of (i), see \cite{Cra-norm, Cra-add}. 
(ii) also can be proved similarly, 
using \cite[Theorem 2.1]{Cra-par} and    
the method of Kraft-Procesi~\cite{KP}.
\end{proof}

\begin{remark}\label{4.1.2}
Recall that every coadjoint orbit has a canonical symplectic structure. 
Thus identifying $\gl(r,\C)$ with its dual via the trace, 
each $\mc{C}_i$ carries naturally an algebraic symplectic structure. 
The product of these symplectic forms defines an algebraic symplectic structure on 
$\prod_{i=1}^n \mc{C}_i$, 
and it has a moment map for the $\GL(r,\C)$-action 
given by $(B_1, \dots ,B_n) \mapsto \sum B_i$. 
Thus $\mc{Q}$ also carries naturally an algebraic symplectic structure by 
Theorem~\ref{2.3.3}. 
Then one can prove that the restriction of 
the map defined in (i) is a symplectic isomorphism  
between $\Mreg_{\zeta,0}(V)$ and $\mc{Q}$.  

Recall further that every conjugacy class $\mc{C} \subset G$ of 
a complex reductive group has a canonical quasi-Hamiltonian $G$-structure~(see \cite{AMM}). 
So under the same assumption as in (ii), 
the product $\prod_i \mc{C}_i$ carries 
a quasi-Hamiltonian $\GL(r,\C)$-structure 
by Theorem~\ref{2.4.4}. 
Its group-valued moment map is $(B_1,\dots ,B_n) \mapsto \prod B_i$, 
and hence the variety $\mc{R}$ carries an  
algebraic symplectic structure by Theorem~\ref{2.4.5}. 
One can also prove that 
the map defined in (ii) induces a symplectic isomorphism  
between $\MMreg_{q,0}(V)$ and $\mc{R}$. 
\end{remark}

From now on, we assume that $\xi_{i,j} \neq 0$ for all $i,j$. 
Take $n$ distinct points $p_1,\dots , p_n$ in the Riemann sphere $\P^1$, 
and set $D=\{ p_1,\dots ,p_n \}$. 
Choose a base point $* \in \P^1 \setminus D$ and  
consider the fundamental group $\pi_1(\P^1 \setminus D,*)$.  
It has a presentation $\langle \gamma_1, \gamma_2, \dots ,\gamma_n \mid \gamma_1 \cdots \gamma_n =1 \rangle$,  
where $\gamma_i$ represents a loop going from $*$ toward near $p_i$,
once around counterclockwise and back to $*$. Hence the map 
\begin{align*}
\Hom (\pi_1(\P^1 \setminus D,*),\GL(r,\C)) &\to
\{\, (A_1,\dots ,A_n) \in \GL(r,\C)^n \mid A_1 \cdots A_n =1\,\}, \\
\rho &\mapsto A_i =\rho (\gamma_i)
\end{align*}
is bijective. 
If we consider $\Hom (\pi_1(\P^1 \setminus D,*),\GL(r,\C))$ as an affine algebraic variety
via this bijection, then the space $\MM_{q,0}(V)$ can be described as
\[
\{\,\rho\in \Hom (\pi_1(\P^1 \setminus D,*),\GL(r,\C))
\mid \rho (\gamma_i) \in \ov{\mc{C}_i}\,\} \GIT \GL(r,\C). 
\] 
In fact, the multiplicative quiver variety
$\MM_{q,\theta} (V)$ with $\theta_{i,j}>0$ can be 
also described as a moduli space of local systems on $\P^1 \setminus D$ 
equipped with a certain additional structure, called a {\em filtered structure}.  

\subsection{Filtered local system}\label{4.2}

Suppose that the stability parameter $\theta \in \Q^I$ satisfies
\[
\theta_{i,j} >0,\quad \text{and} \quad \theta \cdot \dim V =0, \ \text{i.e.,} \ 
\theta_0 = - \frac{\sum_{[i,j]\in I_0} \theta_{i,j} \dim V_{i,j}}{\dim V_0}.
\]
Let $x=(a,b) \in \Phi_{V}^{-1}(q)$ be a 
$\theta$-semistable point.
For $i=1,\dots ,n$, 
define a filtration $F_i=(F_i^j)$ of $V_0$ by 
\[ 
F_i^0(V_0) := V_0, \qquad 
F_i^j(V_0) := \Im a_{i,0} \cdots a_{i,j-1} 
\quad \text{for}\ j=1,2, \dots ,l_i+1,
\]
and set $A_i:=\xi_{i,0} (1+a_{i,0}b_{i,0}) \in \GL(V_0)$. 

\begin{lemma}\label{4.2.1} 
For each $[i,j] \in I_0$, we have:  
\begin{enumerate}
\item[\rm (i)] $(A_i -\xi_{i,j})(F_{i}^j) \subset F_{i}^{j+1}$; and 
\item[\rm (ii)] $\dim F_i^j = v_{i,j}$. 
\end{enumerate}
\end{lemma}

\begin{proof}
Using induction on $j$, we first prove the following formula which implies (i):
\[
(A_i -\xi_{i,j}) a_{i,0} \cdots a_{i,j-1} 
= \xi_{i,j} a_{i,0} \cdots a_{i,j-1} a_{i,j} b_{i,j}. 
\]
If $j=0$, by definition we have $A_i -\xi_{i,0} = \xi_{i,0}a_{i,0}b_{i,0}$. 
If $j>0$, using the hypothesis of induction we have  
\begin{align*}
(A_i -\xi_{i,j}) a_{i,0} \cdots a_{i,j-1}
&= (A_i - \xi_{i,j-1}) a_{i,0} \cdots a_{i,j-2}a_{i,j-1} 
+ (\xi_{i,j-1} -\xi_{i,j}) a_{i,0} \cdots a_{i,j-1} \\
&= \xi_{i,j-1} a_{i,0} \cdots a_{i,j-1} b_{i,j-1} a_{i,j-1}
+ (\xi_{i,j-1} -\xi_{i,j}) a_{i,0} \cdots a_{i,j-1} \\
&= \xi_{i,j-1}a_{i,0} \cdots a_{i,j-1} (1+b_{i,j-1} a_{i,j-1}) 
- \xi_{i,j} a_{i,0} \cdots a_{i,j-1}.
\end{align*}
By the multiplicative preprojective relation at $[i,j]$, we have 
$1+b_{i,j-1} a_{i,j-1} = q_{i,j}^{-1}(1+ a_{i,j}b_{i,j})$. Thus 
we obtain the desired formula as the following:
\begin{align*}
(A_i -\xi_{i,j}) a_{i,0} \cdots a_{i,j-1}
&= q_{i,j}^{-1}\xi_{i,j-1}a_{i,0} \cdots a_{i,j-1} (1+a_{i,j} b_{i,j}) 
- \xi_{i,j} a_{i,0} \cdots a_{i,j-1} \\
&= \xi_{i,j}a_{i,0} \cdots a_{i,j-1} (1+a_{i,j} b_{i,j}) 
- \xi_{i,j} a_{i,0} \cdots a_{i,j-1} \\
&=\xi_{i,j}a_{i,0} \cdots a_{i,j-1}a_{i,j} b_{i,j}.
\end{align*}

To prove (ii), it is enough to show that each $a_{i,j-1}$ is injective.
Fix $[i,j] \in I_0$ and  
define a subspace $S \subset V$ by 
\[
S_0 :=0, \qquad
S_{k,m} := 
\begin{cases} 0 & \text{if}\ k\neq i \ \text{or}\ m<j, \\
\Ker a_{i,j-1} & \text{if}\ [k,m]=[i,j], \\
b_{i,m-1}b_{i,m-2} \cdots b_{i,j}(\Ker a_{i,j-1}) 
& \text{if}\ k=i \ \text{and}\ m>j.
\end{cases}
\]
As above one can easily prove the following formula:
\[
\xi_{i,j} a_{i,m-1} b_{i,m-1}b_{i,m-2} \cdots b_{i,j}
= \xi_{i,m-2} b_{i,m-2} \cdots b_{i,j} b_{i,j-1}a_{i,j-1} 
+(\xi_{i,m-2} -\xi_{i,j} )b_{i,m-2} \cdots b_{i,j}.
\]
This implies that $S$ is $(a,b)$-invariant. 
By the stability condition we have
\[
\sum_{[k,m] \in I_0}\theta_{k,m}\dim S_{k,m}=\theta \cdot \dim S \leq 0.
\]
Since $\theta_{k,m} >0$ we must have $S_{k,m}=0$ for each $[k,m]$.
Thus $a_{i,j-1}$ is injective.
\end{proof}
 
The multiplicative preprojective relation at $0$ implies $\prod A_i =1$. 
Thus setting $\rho(\gamma_i)=A_i~(i=1,\dots n)$, we get  
a representation $\rho$ of $\pi_1(\P^1 \setminus D,*)$ on $V_0$.
Let $L$ be the corresponding local system on $\P^1 \setminus D$. 
For $i=1, \dots ,n$, let $U_i$ be a simply connected open neighborhood of $p_i$ 
which contains $\gamma_i$, and we set $U_i^* =U_i \setminus \{ p_i \}$.
Note that $\pi_1(U_i^*,*)$ is a free group generated by $\gamma_i$. 
Thus $\rho(\gamma_i)$ determines a representation of 
$\pi_1(U_i^*,*)$ on $V_0$ which corresponds to 
the restriction of $L$ on $U_i^*$. 
Since each $F_i^j \subset V_0$ is preserved by $\rho(\gamma_i)$, 
it induces a local subsystem $\bF_i^j(L)$ of $L|_{U_i^*}$.  
So we get a filtration 
\[
\bF_i \colon 
L |_{U^*_i}=\bF_{i}^0(L) \supset \bF_i^1(L) \supset \cdots \supset \bF_{i}^{l_i +1}(L)=0
\] 
by local subsystems of $L |_{U^*_i}$. 
Note that the local monodromy of    
$\bF_{i}^{j}(L) /\bF_{i}^{j+1}(L)~(j=0,1,\dots ,l_i)$ around $p_i$  
is given by the scalar multiplication by $\xi_{i,j}$. 

\begin{lemma}\label{4.2.2}
$(L,\bF)$ satisfies the following property:

$(\dagger)$ For any non-zero proper local subsystem $M \subset L$,
the following inequality holds: 
\[
\frac{\sum_{[i,j]\in I_0} \theta_{i,j} \rank \left(  M\cap \bF_i^j(L) \right)  }{\rank M}
\leq \frac{\sum_{[i,j]\in I_0} \theta_{i,j} \rank \bF_i^j(L)}{\rank L}.
\]  
If $(a,b)$ is $\theta$-stable, then the strict inequality holds. 
\end{lemma}

\begin{proof}
For a non-zero proper local subsystem $M \subset L$, 
define a subspace $S \subset V$ by
\[
S_0 = M_* \subset V_0 , \qquad
S_{i,j} =(a_{i,0} \cdots a_{i,j-1})^{-1}(M_* \cap \bF_i^j(L)_*),
\]
where $M_*, \bF_i^j(L)_*$ mean the stalks at $*$. 
Then $S$ is $(a,b)$-invariant and non-zero proper by the assumption. 
On the other hand, $\theta \cdot \dim V =0$ implies  
\begin{align*}
\theta \cdot \dim S 
&= \theta_0 \rank M + \sum_{i,j} \theta_{i,j} \rank (M \cap \bF_i^j(L) ) \\
&= - \frac{\sum_{i,j} \theta_{i,j} \rank \bF_i^j(L)}{\rank L} \rank M
+\sum_{i,j} \theta_{i,j} \rank (M\cap \bF_i^j(L) ). 
\end{align*}
Thus $\theta \cdot \dim S \leq 0$ (resp.\ $< 0$) 
if and only if the inequality (resp.\ the strict inequality) in $(\dagger)$ holds.
So the assertion follows. 
\end{proof}

Motivated on the above argument, we introduce the following notion.
\begin{definition}\label{4.2.3}
Let $X$ be a compact Riemann surface and let $D \subset X$ be a finite subset. 
Let $L$ be a local system on $X \setminus D$. 
For a tuple of non-negative integers $l=(l_p)_{p \in D}$, 
a {\em filtered structure} on $L$ of {\em filtration type} $l$ is a tuple 
$(U_p, \bF_p)_{p\in D}$, where for each $p\in D$:  
\begin{enumerate}
\item[\rm (i)] $U_p$ is a neighborhood of $p$ in $X$ (we set $U_p^* :=U_p \setminus \{ p \}$); and   
\item[\rm (ii)] $\bF_p$ is a filtration  
\[
L |_{U_p^*} = \bF_p^0(L) \supset \bF_p^1(L) \supset \cdots
 \supset \bF_p^{l_p}(L) \supset \bF_p^{l_p+1}(L)=0
\]
by local subsystems of $L |_{U_p^*}$.
\end{enumerate}

Two filtered structures $(U_p, \bF_p)_{p\in D}, (U'_p, \bF'_p)_{p\in D}$ of the same filtration type 
are {\em equivalent} if for each $p\in D$, there exists a neighborhood $V_p \subset U_p \cap U'_p$ 
of $p$ such that $\bF_p$ and $\bF'_p$ coincide on $V^*_p$.
A local system $L$ together with an equivalence class of filtered structures 
$\bF =[ (U_p, \bF_p)_{p\in D} ]$
is called a {\em filtered local system} on $(X,D)$ of filtration type $l$. 
\end{definition}

\begin{definition} \label{4.2.4}
Let $(L,\bF)$ be a filtered local system on $(X,D)$ of filtration type $l$. 
Let $\beta = (\beta_p^j \mid p\in D,~j=0,\dots ,l_p)$ be a tuple of rational numbers 
satisfying $\beta_p^i < \beta_p^j$ for any $p$ and $i <j$ 
(Such a tuple is called a {\em weight}). 
 
$(L,\bF)$ on $(X,D)$ is said to be {\em $\beta$-semistable} if 
for any non-zero proper local subsystem $M \subset L$ the following inequality holds:
\[
\sum_{p\in D} \sum_j \beta_p^j \frac{ \rank \left(  M \cap \bF_p^j(L) \right)  /\left( 
M \cap \bF_p^{j+1}(L) \right)  }{\rank M} 
\leq
\sum_{p\in D} \sum_j \beta_p^j \frac{ \rank \left(  \bF_p^j(L)/
\bF_p^{j+1}(L) \right)  }{\rank L}.
\]
$(L,\bF)$ is {\em $\beta$-stable} if the strict inequality always holds.
\end{definition}

Clearly, $(L, \bF)$ constructed from a 
$\theta$-semistable point $x=(a,b) \in \Phi_V^{-1}(q) \cap \bMss_\theta(V)$    
defines a filtered local system on $(\P^1,\{ p_i \})$, 
where the filtration type $l$ is given by $l_{p_i}:= l_i$. 
Moreover this filtered local system satisfies the stability condition. 
Fix arbitrary $\beta_i^0 \in \Q$ for each $i$
and set $\beta_{p_i}^j:=\beta_i^0 +\sum_{s=1}^j\theta_{i,s}$. 
Then we have 
\[
\sum_{i=1}^n \sum_{j \geq 0} 
\beta_{p_i}^j \frac{\rank \bF_{p_i}^j(L)/\bF_{p_i}^{j+1}(L)}{\rank L} 
= \sum_{i=1}^n \beta_{i,0} 
+ \sum_{[i,j]\in I_0} 
\theta_{i,j} \frac{\rank \bF_{p_i}^j(L)}{\rank L},
\]
so the $\beta$-semistability condition for filtered local systems on $(\P^1,\{ p_i \})$ 
is equivalent to the property $(\dagger)$, and the $\beta$-stability condition is equivalent to 
that the strict inequality always holds in $(\dagger)$.   
In particular our $(L,\bF)$ is $\beta$-semistable, and if $x$ is $\theta$-stable then 
$(L,\bF)$ is $\beta$-stable. 

It is easy to see that the above construction sends a $G_V$-orbit in 
$\Phi_V^{-1}(q) \cap \bMss_\theta(V)$ to 
an isomorphism class of filtered local systems, 
and preserves the direct sum operation, 
where the direct sum of two filtered local systems of the same filtration type  
means the direct sum of local systems with filtrations induced from those of the two. 
In particular, this map sends a $\theta$-polystable point 
$x=x_1 \oplus x_2 \oplus \cdots \oplus x_N$ to a direct sum 
of $\beta$-stable filtered local systems 
$(L,\bF)=(L_1,\bF_1)\oplus (L_2,\bF_2) \oplus \cdots \oplus (L_N,\bF_N)$. 
Note that each $(L_i,\bF_i)$ satisfies 
\[
\sum_{p\in D} \sum_j \beta_p^j \frac{ \rank \left(  (\bF_i)_p^j(L_i)/
(\bF_i)_p^{j+1}(L_i) \right)  }{\rank L_i}=
\sum_{p\in D} \sum_j \beta_p^j \frac{ \rank \left(  \bF_p^j(L)/
\bF_p^{j+1}(L) \right)  }{\rank L},
\]
since $\theta \cdot \dim V^i =0$ (see Proposition~\ref{2.2.4}).  
Such a filtered local system is said to be $\beta$-{\em polystable}.    

Conversely, suppose that a $\beta$-semistable filtered local system $(L,\bF)$ 
of filtration type $l$ with $\rank L=r$, $\rank L = v_{i,j}$ is given. 
Suppose that the local monodromy of $\bF_{p_i}^j(L)/\bF_{p_i}^{j+1}$ around $p_i$ 
is given by the scalar multiplication $\xi_{i,j}$ for all $i,j$.   
We define an $I$-graded vector space $V$ by 
$V_0 :=L_*,~V_{i,j}:=\bF_{p_i}^j(L)_*$,
and define a point $(a,b) \in \bM(V)$ by 
\[
b_{i,j} := (\xi_{i,j}^{-1}\rho(\gamma_i) -1)|_{V_{i,j}} \colon V_{i,j} \to V_{i,j+1}, \qquad  
a_{i,j} \colon V_{i,j+1}\hookrightarrow V_{i,j} \quad \text{the inclusion}.
\]
Then $(a,b)\in \bM(V)$ satisfies the multiplicative preprojective relation. 
To check the stability condition,  
suppose that a non-zero proper $(a,b)$-invariant subspace $S \subset V$ is given.
Then there is a local subsystem $M\subset L$ whose stalk at $*$ is $S_0$.
By the property $(\dagger)$, we have
\[
\frac{\sum_{[i,j]\in I_0} \theta_{i,j} \rank \left(  M\cap \bF_{p_i}^j(L)\right)  }{\rank M}
\leq \frac{\sum_{[i,j]\in I_0} \theta_{i,j} \rank \bF_{p_i}^j(L)}{\rank L}.
\]  
Since $a_{i,j}$'s are injective, we have 
$\dim S_{i,j} \leq \rank ( M \cap \bF_{p_i}^j(L) )$, 
and hence 
\[
\frac{\sum_{[i,j]\in I_0} \theta_{i,j} \dim S_{i,j}}{\dim S_0}
\leq \frac{\sum_{[i,j]\in I_0} \theta_{i,j} \dim V_{i,j}}{\dim V_0},
\]  
which implies $\theta \cdot \dim S \leq 0$. Thus $(a,b)$ is $\theta$-semistable.  
Clearly, if $(L,\bF)$ is $\beta$-stable then $(a,b)$ is $\theta$-stable. 
It is also easy to see that this construction sends an isomorphism class 
of filtered local systems to a $G_V$-orbit, 
and preserves the polystability.      

We have obtained maps of both directions between $\MM_{q,\theta}(V)$ and 
the set of isomorphism classes of $\beta$-polystable filtered local systems $(L,\bF)$ 
of filtration type $l$ satisfying $\rank L =r$, $\rank \bF_{p_i}^j(L) =v_{i,j}$ 
and that the local monodromy of $\bF_{p_i}(L)/\bF_{p_i}^{j+1}(L)$ is given by  
the scalar $\xi_{i,j}$ for all $i,j$. 
Clearly each one is the inverse of the other. 
So we get the following result. 
 
\begin{theorem}\label{4.2.5}
Let $D=\{ p_1, \dots ,p_n \}$ 
be a finite subset of $\P^1$ with cardinality $n$. 
Take an arbitrary $l \in \N^D$, 
and let $\xi =(\xi_p^j \mid p \in D,~j=0,\dots , l_p)$ be a tuple of 
non-zero complex numbers, 
$\beta = (\beta_p^j \mid p \in D,~j=0,\dots , l_p)$ be a tuple of 
rational numbers satisfying $\beta_p^i < \beta_p^j$ for any $p$ and $i <j$.
Take a star-shaped quiver $(I,\Omega)$ 
with $n$ arms such that the length $l_i$ of the $i$-th arm is equal to $l_{p_i}$.   
Then for any $I$-graded vector space $V$,  
setting $(q,\theta) \in (\C^{\times})^I \times \Q^I$ by
\begin{align*}
\theta_{i,j} &:= \beta_{p_i}^j -\beta_{p_i}^{j-1}, &      
\theta_0 &:= -\frac{\sum_{[i,j] \in I_0}\theta_{i,j} \dim V_{i,j}}{\dim V_0}, \\
q_{i,j} &:= \xi_{p_i}^{j-1}/\xi_{p_i}^{j}, &    
q_0 &:= \prod_i (\xi_{p_i}^0)^{-1}, 
\end{align*}
there is a natural bijection between
the multiplicative quiver variety $\MM_{q,\theta}(V)$
and the set of isomorphism classes of 
$\beta$-polystable filtered local systems $(L,\bF)$ on $(\P^1,D)$ 
satisfying: 
\begin{itemize}
\item $\rank L =\dim V_0,~\rank \bF_{p_i}^j(L)=\dim V_{i,j}$; 
\item the local monodromy of $\bF_{p_i}^j(L)/\bF_{p_i}^{j+1}(L)$ around $p_i$ is given by  
the scalar multiplication by $\xi_{p_i}^j$ for all $i,j$.
\end{itemize} 
Under this map, a point in $\MMreg_{q,\theta}(V)$ corresponds to 
an isomorphism class of $\beta$-stable filtered local systems. 
\end{theorem}

\begin{remark}\label{4.2.6}
The word ``filtered local system'' is originally due to Simpson~\cite{Sim}. 
Simpson's filtered local system $(L,\bF)$ is 
a pair of a local system $L$ on $X \setminus D$ and 
a tuple $\bF=(\bF_p)_{p\in D}$, 
where for each $p \in D$, 
$\bF_p= ( \bF_p^\beta )_{\beta \in \R}$ is 
a filtration of the restriction $L |_{U_p^*}$ of $L$ on 
some punctured neighborhood $U^*_p$ of $p$ indexed by real number $\beta \in \R$. 
$\bF_p^\beta$ is required to be left continuous, i.e., $\bF_p^{\beta -\varepsilon} = \bF_p^\beta$ 
for small $\varepsilon >0$. 
Filtered local systems in the sense of Simpson form a category  
on which direct sum, tensor product, dual, etc.\ are defined. 
Moreover the notion of ``degree'' for a filtered local system in the sense of Simpson 
is naturally defined and provides a slope stability condition. 
Our notion of filtered local system $(L,\bF)$ together with a weight $\beta$  
can be considered as Simpson's filtered local system as follows: 
for each $\beta \in \R$, 
define $\bF_p^\beta(L) \subset L|_{U^*_p}$ by 
\[
\bF_p^\beta (L) := \begin{cases} 
\,L|_{U^*_p} & \text{when}\ \beta \leq \beta_p^0, \\
\,\bF_p^j(L) & \text{when}\ \beta \in (\beta_p^{j-1},\beta_p^j], \\
\,0 & \text{when} \ \beta > \beta_p^{l_p}.
\end{cases}
\]
Then $(L,\{ \bF_p^\beta \})$ is a filtered local system in the sense of Simpson. 
Moreover one can easily check that  
if our $(L,\bF)$ is $\beta$-stable/semistable/polystable, then 
$(L,\{ \bF_p^\beta \})$ is stable/semistable/polystable. 
Note that in the previous theorem, 
the star-shaped multiplicative quiver varieties 
parametrize only polystable filtered local systems $(L,\{ \bF_p^\beta \})$    
such that the local monodromy of $\bF_p^\beta / \bF_p^{>\beta}$ around $p$ is 
{\em scalar} for each $p, \beta$. 
This is because we have considered only the case that all $\theta_{i,j}$'s are positive. 
In fact, if we allow $\theta_{i,j} = 0$ for some $i,j$, 
then a point in the multiplicative quiver variety represents a   
polystable filtered local system $(L, \{ \bF_p^\beta \})$ 
such that the local monodromy of $\bF_p^\beta / \bF_p^{>\beta}$ 
is in the closure of some fixed conjugacy class, 
which may not be a scalar.    
\end{remark}

\subsection{Riemann-Hilbert correspondence}\label{4.3}

\begin{definition}\label{4.3.1}
Let $X$ be a compact Riemann surface and 
let $D \subset X$ be a finite subset. 
A {\em logarithmic connection} $(E,\nabla)$ on $(X ,D)$ 
is a pair of a holomorphic vector bundle $E$ on $X$ 
and a morphism of sheaves   
$\nabla \colon E \to E \otimes \Omega^1_X (\log D)$ 
satisfying the Leibniz rule:
\[
\nabla (f s ) = df \otimes s + f \nabla (s) \quad \text{for}\ f \in \mc{O}_{X}, s \in E,
\]
where we have used the same symbol $E$ for the sheaf of holomorphic sections of $E$, 
$\Omega^1_X(\log D)$ is the sheaf of meromorphic 1-forms on $X$ 
with logarithmic poles on $D$ and no poles on $X \setminus D$,     
and $\mc{O}_X$ is the sheaf of holomorphic functions on $X$. 
\end{definition}

For each $p \in D$, a logarithmic connection $(E,\nabla)$ induces canonically 
an endomorphism 
\[
\Res_p \nabla \colon E|_p \to E|_p
\] 
of the fiber $E|_p$ of $E$ at $p$. 
Such an endomorphism is called the {\em residue} 
of $(E,\nabla)$ at $p$. 
Using a trivialization $E |_{U_p} \simeq U_p \times \C^r$ on 
a neighborhood of $p$ and 
a local coordinate $z$ centered at $p$, 
the logarithmic connection $\nabla$ is written as $\nabla = d + A(z)dz/z$ 
for some holomorphic function $A(z)$. 
Then $\Res_p \nabla =A(0)$. 

One can also define the notion of logarithmic connection in the algebro-geometric sense. 
However by GAGA, it is equivalent to the above notion.  

By Deligne's Riemann-Hilbert correspondence~\cite{Del}, 
there is a natural equivalence  
between the category of local systems $L$ on $X \setminus D$ 
and the category of logarithmic connections $(E,\nabla)$ on $(X,D)$ 
such that the real parts of eigenvalues of the residue $\Res_p \nabla$ 
are in $[0,1)$ for any $p \in D$. 
A ``filtered'' version of it was proved by Simpson~\cite{Sim}. 
To explain it, we introduce a ``filtered'' structure on 
a logarithmic connection, so-called a {\em parabolic structure}.

\begin{definition}\label{4.3.2}
Let $X$ be a compact Riemann surface and let $D \subset X$ be a finite subset. 
Let $(E,\nabla)$ be a logarithmic connection on $(X ,D)$. 

For $l=(l_p)_{p\in D} \in \N^D$, 
a {\em parabolic structure} on $(E,\nabla)$ of {\em filtration type} $l$ is a tuple 
$\FF = (\FF_p)_{p\in D}$, where for each $p\in D$, $\FF_p$ is a filtration  
\[
E|_p = \FF_p^0(E) \supset \FF_p^1(E) \supset \cdots
 \supset \FF_p^{l_p}(E) \supset \FF_p^{l_p+1}(E)=0
\]
by vector subspaces of the fiber $E|_p$ at $p$.

A logarithmic connection $(E,\nabla)$ 
together with a parabolic structure $\FF =(\FF_p)_{p\in D}$ 
is called a {\em parabolic connection} on $(X,D)$.
\end{definition}

\begin{definition}\label{4.3.3}
Let $\alpha = (\alpha_p^j \mid p\in D,~j=0,\dots ,l_p)$ be a tuple of rational numbers in $[0,1)$  
such that $\alpha_p^i < \alpha_p^j$ for any $p$ and $i <j$.  
A parabolic connection $(E,\nabla,\FF)$ is said to be {\em $\alpha$-semistable} if 
for any non-zero proper subbundle $F \subset E$ preserved by $\nabla$, 
the following inequality holds:
\[
\sum_{p\in D} \sum_j \alpha_p^j \frac{ \dim \left(  F|_p \cap \FF_p^j(E) \right)  /\left( 
F|_p \cap \FF_p^{j+1}(E) \right)  }{\rank F} 
\leq
\sum_{p\in D} \sum_j \alpha_p^j \frac{ \dim \left(  \FF_p^j(E)/
\FF_p^{j+1}(E) \right)  }{\rank E}.
\]
$(E,\nabla,\FF)$ is {\em $\alpha$-stable} if the strict inequality always holds.  
A direct sum $(E,\nabla,\FF)=\bigoplus_i (E_i, \nabla_i, \FF_i)$ of $\alpha$-stable parabolic connections  
satisfying  
\[
\sum_{p\in D} \sum_j \alpha_p^j \frac{ \dim \left(  (\FF_i)_p^j(E_i)/
(\FF_i)_p^{j+1}(E_i) \right)  }{\rank E_i}=
\sum_{p\in D} \sum_j \alpha_p^j \frac{ \dim \left(  \FF_p^j(E)/
\FF_p^{j+1}(E) \right)  }{\rank E}
\]
for all $i$ is said to be {\em $\alpha$-polystable}.  
\end{definition}

We now introduce the filtered version of Deligne's Riemann-Hilbert correspondence. 

\begin{theorem}[{\cite[Lemma 3.2]{Sim}}]\label{4.3.4}
Let $X$ be a compact Riemann surface and $D \subset X$ be a finite subset. 
Then there is a natural bijective correspondence between: 
\begin{enumerate}
\item[\rm (i)] isomorphism classes of filtered local systems $(L,\bF)$ on $(X,D)$ 
together with a weight $\beta$; and   
\item[\rm (ii)] isomorphism classes of parabolic connections $(E,\FF)$ on $(X,D)$ 
together with a weight $\alpha$.  
\end{enumerate}
For each $p\in D$, this correspondence induces a bijection between: 
\begin{align*}
&\rset{(\lambda,\alpha_p^j) \in \C \times [0,1)}{ 
\text{the action of $\Res_p \nabla$ on $\FF_p^j(E)/\FF_p^{j+1}(E)$ has an eigenvalue $\lambda$}}; 
\quad \text{and} \\ 
&\rset{(\xi,\beta_p^k) \in \C^\times \times \R \,}{
\begin{array}{l} 
\text{the monodromy of $\bF_p^k(L)/\bF_p^{k+1}(L)$} \\ 
\text{along a simple loop around $p$ (counterclockwise) has an eigenvalue $\xi$}
\end{array}
}, 
\end{align*}
which is explicitly given by $(\lambda, \alpha) \mapsto (\xi, \beta)$, where   
\[
\beta := \alpha - \Re \lambda, \qquad   
\xi := \exp (-2\pi \sqrt{-1} \lambda). 
\] 
Furthermore, if $(\lambda, \alpha_p^j)$ corresponds to $(\xi, \beta_p^k)$ under this bijection, 
then the generalized $\lambda$-eigen space of $\FF_p^j(E)/\FF_p^{j+1}(E)$ and  
the generalized $\xi$-eigen space of $\bF_p^k(L)/\bF_p^{k+1}(L)$  
have the same dimension.    
\end{theorem}

Recently, Inaba constructed the moduli space 
of $\alpha$-semistable {\em $\lambda$-parabolic connections} $(E,\nabla,\FF)$ on $(X,D)$~\cite{Ina} 
of rank $r>0$ for a given tuple $\lambda =(\lambda_p^j \mid p\in D, j=0, \dots ,r-1)$, 
where $\lambda$-parabolic connection means a parabolic connection of full filtration type 
(i.e., $l_p = r-1$ and $\dim \FF_p^j(E)=r -j$) and 
$(\Res_p \nabla -\lambda_p^j)(\FF_p^j(E)) \subset \FF_p^{j+1}(E)$ for each $p, j$.    
(We will use the word ``$\xi$-filtered local system'' by a similar manner.) 
We denote this moduli space by 
$\MM_{\lambda,\alpha}(X,D;r)$.  
Its stable locus $\MMreg_{\lambda,\alpha}(X,D;r)$ has naturally 
an algebraic symplectic structure. 

Now consider the case of $X=\P^1$. 
We can take $\alpha$ to be generic so that 
\[
\MM_{\lambda,\alpha}(\P^1,D;r)=\MMreg_{\lambda,\alpha}(\P^1,D;r). 
\]
Inaba showed that if $rn-2r-2 >0$ and $r \geq 2$ ($n$ is the cardinality of $D$), 
then $\MM_{\lambda,\alpha}(\P^1,D;r)$ is an irreducible variety of dimension 
$(r-1)(rn-2r-2)$~\cite[Proposition 4.3]{Ina}. 
We assume further that $\alpha_p^i - \Re \lambda_p^i \neq \alpha_p^j - \Re \lambda_p^j$ 
for $i \neq j$ so that one can take a permutation $\sigma_p \in \mathfrak{S}_{l_p+1}$ 
such that 
\[
i < j \quad \Longrightarrow \quad  
\alpha_p^{\sigma_p(i)} - \Re \lambda_p^{\sigma_p(i)} <   
\alpha_p^{\sigma_p(j)} - \Re \lambda_p^{\sigma_p(j)}.  
\] 
Then under Simpson's Riemann-Hilbert correspondence, 
an $\alpha$-semistable $\lambda$-parabolic connection correspond to 
a $\beta$-semistable $\xi$-filtered local system, 
where $\beta$ and $\lambda$ are given by
\[
\beta_p^j := \alpha_p^{\sigma_p(j)} - \Re \lambda_p^{\sigma_p(j)} \qquad 
\text{and} \qquad 
\xi_p^j := \exp (-2\pi \sqrt{-1} \lambda_p^{\sigma_p(j)}).  
\]
Assume $\Re \lambda_p^k \in \Q$ so that $\beta_p^j \in \Q$.  

\begin{theorem}\label{4.3.5}  
Under the above notation and assumptions, 
let $(I,\Omega)$ be a star-shaped quiver with $n$ arms  
such that the length $l_i$ of the $i$-th arm is equal to $r-1$ for any $i$.  
Set $q,\theta$ as in Theorem~\ref{4.2.5}, 
and take an $I$-graded vector space $V$ with $\dim V_0=r$, $\dim V_{i,j}=r-j$.  
Then Simpson's Riemann-Hilbert correspondence 
gives a symplectic biholomorphic map  
between $\MM_{\lambda,\alpha}(\P^1,D;r)$ and $\MMreg_{q,\theta}(V)=\MM_{q,\theta}(V)$.   
\end{theorem} 

\begin{proof}
First of all we recall Simpson's Riemann-Hilbert correspondence. 
This correspondence can be constructed locally, 
so we may replace $X=\P^1$ with the unit open disk $\{\, z \in \C \mid \lvert z \rvert <1\,\}$ 
and assume $D=\{ 0 \}$. 
Also, for simplicity we assume that the permutation $\sigma_p$ for $p=0$ is an identity. 
Let $(L,\bF)$ be a $\xi$-filtered local system on $(X,D)$. 
$L$ corresponds to a holomorphic bundle with connection $(E',\nabla)$ on $X \setminus D$. 
Take multi-valued flat sections $u_0, u_1, \dots ,u_{r-1}$ of $E'$ 
such that $u_j \in \bF^j(L) \setminus \bF^{j+1}(L)$ 
(we omit the subscript $0\in D$).   
Let $M \in \End E'$ be the monodromy operator 
and let $R$ be a unique operator such that $e^{-2\pi \sqrt{-1}R}=M$ 
and the eigenvalues of $R$ are $\lambda^0, \dots ,\lambda^{r-1}$.   
Then 
\[
v_j(z) := e^{R \log z}u_j(z)
\]
becomes a single-valued holomorphic section of $E'$, 
since when $z$ moves along a simple loop around $p$ 
once counterclockwise, 
$v_j$ goes to    
\[
e^{R \log z} e^{2\pi \sqrt{-1}R} M u_j = v_j.
\]   
If we denote by $\tilde{E'}$ the sheaf of meromorphic section of the Deligne extension of $E'$ 
having pole only at $D$, 
then $v_j$ can be considered as a section of $\tilde{E'}$.    
Let $E$ be the subsheaf of $\tilde{E'}$  
generated by $v_0, \dots ,v_{r-1}$. 
Then $E$ is locally free of rank $r$, and $\nabla$ defines a logarithmic connection on $E$ 
since $\nabla v_j = R v_j dz/z$  
(Note that since $e^{R \log z}$ commutes with $R$, 
the representation matrix of $R$ with respect to the framing $(v_0, \dots ,v_{r-1})$ 
is the same as the one with respect to $(u_0, \dots ,u_{r-1})$, 
and so it is a constant matrix).  
Moreover since $u_j \in \bF^j \setminus \bF^{j+1}$, 
if we let $N$ be the nilpotent part of $R$ then $Rv_j =(\lambda^j +N)v_j$. 
Thus $v_j(0) \in E|_0$ lies in the generalized eigenspace for 
$\Res_0 \nabla$ with eigenvalue $\lambda^j$.  
Hence setting 
\[
\FF^j(E) := \bigoplus_{k \geq j} \C v_k(0) \subset E|_0,
\]
we get a $\lambda$-parabolic connection $(E,\nabla,\FF)$ on $(X,D)$. 

This construction gives a bijection from the set of isomorphism classes 
of $\beta$-stable $\xi$-filtered local systems on $(\P^1,D)$ 
to the set of isomorphism classes of $\alpha$-stable $\lambda$-parabolic connections on $(\P^1,D)$ 
(see \cite{Sim}). Using this fact, let us consider the inverse map. 

Assume again that $X$ is the unit open disk in $\C$ and $D=\{ 0\}$. 
Let $(E,\nabla,\FF)$ be a $\lambda$-parabolic connection on $(X,D)$. 
Let $L$ be the corresponding local system on $X \setminus D$, 
and let $M, R, N$ be as in the previous paragraph.    
Take a basis $(e_0,e_1, \dots ,e_{r-1})$ of $E|_0$ compatible with the filtration $\FF$. 
By the above fact, we can take 
a framing $(v_0, \dots ,v_{r-1})$ of $E$ such that 
\[
\nabla v_j = Rv_j dz/z, \qquad v_j(0)=e_j.
\]    
Then setting $u_j := e^{-R \log z} v_j$,   
we get multi-valued flat sections of $L$. 
Since $(R v_j)(0) = (\Res_0 \nabla) e_j$, 
we have $(R-\lambda_j) u_j \in \sum_{k>j}\C u_k$. 
Thus if we set $\bF^j(L) \subset L$ by the subsheaf generated by $u_j, \dots ,u_{r-1}$, 
then $(L,\bF)$ is a $\xi$-filtered local system on $(X,D)$. 
Note that $v_j$ is uniquely determined by the differential equation 
$\nabla v_j =Rv_j dz/z$ and the condition $v_j(0)=e_j$. 
Now recall that if a differential equation has complex analytic parameters 
then a solution of it also depends complex analytically on the parameters.     
Thus if $(E,\nabla,\FF)$ varies complex analytically, 
then the corresponding local system $L$, the monodromy operator $M$ and 
the filtration on $L$ which determined by $v_j$ also vary complex analytically. 
This implies that the map $\RH \colon \MM_{\lambda,\alpha}(\P^1,D;r) \to \MM_{q,\theta}(V)$ 
given by Simpson's correspondence is complex analytic.  

Next we prove that the map $\RH$ is symplectic. 
First note that since the claim is a closed condition, 
we may assume that $\lambda$ is generic so that the morphism 
$\pi \circ \RH \colon \MM_{\lambda,\alpha}(\P^1,D;r) \to \MM_{q,0}(V)$ 
is a complex analytic isomorphism~(see \cite[Theorem 2.2]{Ina}), 
where $\pi \colon \MM_{q,\theta}(V) \to \MM_{q,0}(V)$ is the canonical 
projective morphism.  
This implies that $\MMreg_{q,0}(V)=\MM_{q,0}(V)$, 
that $\pi$ is a symplectic isomorphism,  
and that $\RH = \pi^{-1} \circ ( \pi \circ \RH )$ is biholomorphic. 
Now take an arbitrary $[\rho] \in \MM_{q,0}(V)$ and 
let $\mc{C}_i$ denote the conjugacy class of $\rho(\gamma_i)$.  
Then we can write $\MM_{q,0}(V)=\MMreg_{q,0}(V)$ as the variety $\mc{R}$ 
associated to $\mc{C}_i$ by Proposition~\ref{4.1.1}.  
We have remarked that $\mc{R}$ has naturally an algebraic symplectic structure, 
and it is isomorphic to $\MM_{q,0}(V)$ as an algebraic symplectic manifold   
via this identification~(see Remark~\ref{4.1.2}).  
Thus the remaining task is to compare the symplectic structure on 
$\mc{R}$ and that on $\MM_{\lambda,\alpha}(\P^1,D;r)$. 
To do this, we use the following fact proved by Alekseev-Malkin-Meinrenken. 
Let $\Sigma$ be the compact Riemann surface with boundary 
obtained by cutting out an open disk $U_p$ centered at $p$ for each $p \in D$.
Then we have $\pi_1(\P^1 \setminus D,*) \simeq \pi_1(\Sigma,*)$ canonically, 
and hence we can identify the variety $\mc{R}$ with 
the moduli space of irreducible flat $C^{\infty}$-connections on $\Sigma$ 
with the holonomy along $\partial U_{p_i}$ lying in $\mc{C}_i$ for each $i$. 
This moduli space is actually smooth, and 
by the method of Atiyah-Bott 
we can construct naturally a symplectic structure on it. 
Alekseev-Malkin-Meinrenken~\cite{AMM} showed that this symplectic structure coincides 
with the one on $\mc{R}$. 
On the other hand, Biquard~\cite{Biq} constructed a natural isomorphism between 
the Zariski tangent space of $\MM_{\lambda,\alpha}(\P^1,D;r)$ at a point $[(E,\nabla,\FF)]$ 
and the degree 1 $L^2$-cohomology of the complex 
$\Omega^\bullet (X \setminus D, \End E)$  
of the spaces of $C^\infty$-forms on $X \setminus D$ with coefficients in $\End E$, 
with the differential given by the flat $C^\infty$-connection $D=\nabla + \ov{\partial}$. 
One can easily check that Inaba's symplectic form on $\MM_{\lambda,\alpha}(\P^1,D;r)$ 
goes to the form on $L^2$-cohomology induced from $(u,v) \mapsto \int_X \Tr\, (u \wedge v)$, 
and this pairing comes from the Atiyah-Bott symplectic structure 
on $\mc{R}$ via the map $\MM_{\lambda,\alpha}(\P^1,D;r) \to \mc{R}$. 
Hence $\RH$ is symplectic. 

Since the determinant of the Jacobian of a symplectic map is 
everywhere non-vanishing, $\RH$ is biholomorphic.  
\end{proof}
 
\subsection{Higher genus case}\label{4.4}

In this subsection let us consider the higher genus case. 
In this case, we cannot describe the moduli space of filtered local systems 
as some multiplicative quiver variety,   
but the quasi-Hamiltonian method still goes through. 

First of all let us consider any quiver $(I,\Omega)$. 
Let $H^\ell$ be the subset of $H$ which consists of all loops in $H$; 
$H^\ell := \{\, h \in H \mid \vin(h)=\vout(h) \,\}$, 
and set $\Omega^\ell := H^\ell \cap \Omega$.  
For an $I$-graded vector space $V$,   
we define an open subset $\bMl(V) \subset \bM(V)$ by 
\[
\bMl(V) := \{\, x \in \bM(V) \mid 
\text{$\det x_h \neq 0$ for $h \in H^\ell$, and 
$\det (1+x_h x_{\ov h}) \neq 0$ for $h \in H \setminus H^\ell$}\,\}.
\]  
This is a $\varphi$-saturated open subset of $\bM(V)$, 
where $\varphi \colon \bM(V) \to \bM(V) \GIT G_V$ is the quotient morphism. 
For any $i \in I$, 
the variety $\GL(V_i) \times \GL(V_i)$ has a quasi-Hamiltonian $\GL(V_i) \times \GL(V_i)$-structure 
whose group-valued moment map is $(a,b) \mapsto (ab,a^{-1}b^{-1})$~(see Example~\ref{2.4.3}). 
Thus by fusioning, we can construct a quasi-Hamiltonian $G_V$-structure on $\bMl(V)$ 
whose group-valued moment map $\Psi_V \colon \bMl(V) \to G_V$ is given by  
\[
(\Psi_V)_i(x):= \prod^<_{h \in H_i \cap \Omega^\ell} [x_h, x_{\ov h}]^{\rm m} 
\prod^<_{h \in H_i \setminus H^\ell} (1+x_h x_{\ov h})^{\epsilon(h)},
\]
where $[ x_h, x_{\ov h} ]^{\rm m} := x_h x_{\ov h} x_h^{-1} x_{\ov h}^{-1}$ 
and we have fixed a total order $<$ on $\Omega^\ell$ and that on $H \setminus H^\ell$.   
Thus for each $(q,\theta) \in (\C^{\times})^I \times \Q^I$, we get a 
variety 
\[
\MMl_{q,\theta}(V)= \left(  \Psi_V^{-1}(q) \cap \bMss(V) \right)   \GIT G_V, 
\]
and its open subset 
\[
\MMlreg_{q,\theta}(V) = \left(  \Psi_V^{-1}(q) \cap \bMs(V) \right)   / G_V
\]
which carries an algebraic symplectic structure. 

Now consider a star-shaped quiver with $g$ loops $(I,\Omega)$ as the following picture: 

\vspace{10pt}

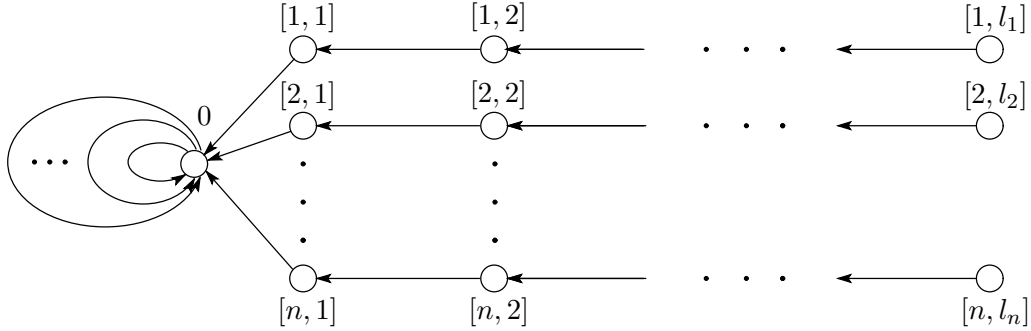
\begin{center}
\unitlength 0.1in
\begin{picture}( 52.1000, 15.4500)(  4.0000,-17.0000)
%
\special{pn 8}%
\special{ar 1376 1010 70 70  0.0000000 6.2831853}%
%
\special{pn 8}%
\special{ar 1946 410 70 70  0.0000000 6.2831853}%
%
\special{pn 8}%
\special{ar 2946 410 70 70  0.0000000 6.2831853}%
%
\special{pn 8}%
\special{ar 5540 410 70 70  0.0000000 6.2831853}%
%
\special{pn 8}%
\special{ar 1946 810 70 70  0.0000000 6.2831853}%
%
\special{pn 8}%
\special{ar 2946 810 70 70  0.0000000 6.2831853}%
%
\special{pn 8}%
\special{ar 5540 810 70 70  0.0000000 6.2831853}%
%
\special{pn 8}%
\special{ar 1946 1610 70 70  0.0000000 6.2831853}%
%
\special{pn 8}%
\special{ar 2946 1610 70 70  0.0000000 6.2831853}%
%
\special{pn 8}%
\special{ar 5540 1610 70 70  0.0000000 6.2831853}%
%
\special{pn 8}%
\special{pa 1890 1560}%
\special{pa 1440 1050}%
\special{fp}%
\special{sh 1}%
\special{pa 1440 1050}%
\special{pa 1470 1114}%
\special{pa 1476 1090}%
\special{pa 1500 1088}%
\special{pa 1440 1050}%
\special{fp}%
%
\special{pn 8}%
\special{pa 2870 410}%
\special{pa 2020 410}%
\special{fp}%
\special{sh 1}%
\special{pa 2020 410}%
\special{pa 2088 430}%
\special{pa 2074 410}%
\special{pa 2088 390}%
\special{pa 2020 410}%
\special{fp}%
%
\special{pn 8}%
\special{pa 3720 410}%
\special{pa 3010 410}%
\special{fp}%
\special{sh 1}%
\special{pa 3010 410}%
\special{pa 3078 430}%
\special{pa 3064 410}%
\special{pa 3078 390}%
\special{pa 3010 410}%
\special{fp}%
\special{pa 3730 410}%
\special{pa 3010 410}%
\special{fp}%
\special{sh 1}%
\special{pa 3010 410}%
\special{pa 3078 430}%
\special{pa 3064 410}%
\special{pa 3078 390}%
\special{pa 3010 410}%
\special{fp}%
%
\special{pn 8}%
\special{pa 2870 810}%
\special{pa 2020 810}%
\special{fp}%
\special{sh 1}%
\special{pa 2020 810}%
\special{pa 2088 830}%
\special{pa 2074 810}%
\special{pa 2088 790}%
\special{pa 2020 810}%
\special{fp}%
%
\special{pn 8}%
\special{pa 2870 1610}%
\special{pa 2020 1610}%
\special{fp}%
\special{sh 1}%
\special{pa 2020 1610}%
\special{pa 2088 1630}%
\special{pa 2074 1610}%
\special{pa 2088 1590}%
\special{pa 2020 1610}%
\special{fp}%
%
\special{pn 8}%
\special{pa 3730 810}%
\special{pa 3020 810}%
\special{fp}%
\special{sh 1}%
\special{pa 3020 810}%
\special{pa 3088 830}%
\special{pa 3074 810}%
\special{pa 3088 790}%
\special{pa 3020 810}%
\special{fp}%
\special{pa 3740 810}%
\special{pa 3020 810}%
\special{fp}%
\special{sh 1}%
\special{pa 3020 810}%
\special{pa 3088 830}%
\special{pa 3074 810}%
\special{pa 3088 790}%
\special{pa 3020 810}%
\special{fp}%
%
\special{pn 8}%
\special{pa 3730 1610}%
\special{pa 3020 1610}%
\special{fp}%
\special{sh 1}%
\special{pa 3020 1610}%
\special{pa 3088 1630}%
\special{pa 3074 1610}%
\special{pa 3088 1590}%
\special{pa 3020 1610}%
\special{fp}%
\special{pa 3740 1610}%
\special{pa 3020 1610}%
\special{fp}%
\special{sh 1}%
\special{pa 3020 1610}%
\special{pa 3088 1630}%
\special{pa 3074 1610}%
\special{pa 3088 1590}%
\special{pa 3020 1610}%
\special{fp}%
%
\special{pn 8}%
\special{pa 5466 410}%
\special{pa 4746 410}%
\special{fp}%
\special{sh 1}%
\special{pa 4746 410}%
\special{pa 4812 430}%
\special{pa 4798 410}%
\special{pa 4812 390}%
\special{pa 4746 410}%
\special{fp}%
%
\special{pn 8}%
\special{pa 5466 810}%
\special{pa 4746 810}%
\special{fp}%
\special{sh 1}%
\special{pa 4746 810}%
\special{pa 4812 830}%
\special{pa 4798 810}%
\special{pa 4812 790}%
\special{pa 4746 810}%
\special{fp}%
%
\special{pn 8}%
\special{pa 5466 1610}%
\special{pa 4746 1610}%
\special{fp}%
\special{sh 1}%
\special{pa 4746 1610}%
\special{pa 4812 1630}%
\special{pa 4798 1610}%
\special{pa 4812 1590}%
\special{pa 4746 1610}%
\special{fp}%
%
\special{pn 8}%
\special{pa 1880 840}%
\special{pa 1450 990}%
\special{fp}%
\special{sh 1}%
\special{pa 1450 990}%
\special{pa 1520 988}%
\special{pa 1500 972}%
\special{pa 1506 950}%
\special{pa 1450 990}%
\special{fp}%
%
\special{pn 8}%
\special{pa 1900 460}%
\special{pa 1430 960}%
\special{fp}%
\special{sh 1}%
\special{pa 1430 960}%
\special{pa 1490 926}%
\special{pa 1468 922}%
\special{pa 1462 898}%
\special{pa 1430 960}%
\special{fp}%
%
\special{pn 8}%
\special{sh 1}%
\special{ar 1946 1010 10 10 0  6.28318530717959E+0000}%
\special{sh 1}%
\special{ar 1946 1210 10 10 0  6.28318530717959E+0000}%
\special{sh 1}%
\special{ar 1946 1410 10 10 0  6.28318530717959E+0000}%
\special{sh 1}%
\special{ar 1946 1410 10 10 0  6.28318530717959E+0000}%
%
\special{pn 8}%
\special{sh 1}%
\special{ar 4056 410 10 10 0  6.28318530717959E+0000}%
\special{sh 1}%
\special{ar 4266 410 10 10 0  6.28318530717959E+0000}%
\special{sh 1}%
\special{ar 4456 410 10 10 0  6.28318530717959E+0000}%
\special{sh 1}%
\special{ar 4456 410 10 10 0  6.28318530717959E+0000}%
%
\special{pn 8}%
\special{sh 1}%
\special{ar 4056 810 10 10 0  6.28318530717959E+0000}%
\special{sh 1}%
\special{ar 4266 810 10 10 0  6.28318530717959E+0000}%
\special{sh 1}%
\special{ar 4456 810 10 10 0  6.28318530717959E+0000}%
\special{sh 1}%
\special{ar 4456 810 10 10 0  6.28318530717959E+0000}%
%
\special{pn 8}%
\special{sh 1}%
\special{ar 4056 1610 10 10 0  6.28318530717959E+0000}%
\special{sh 1}%
\special{ar 4266 1610 10 10 0  6.28318530717959E+0000}%
\special{sh 1}%
\special{ar 4456 1610 10 10 0  6.28318530717959E+0000}%
\special{sh 1}%
\special{ar 4456 1610 10 10 0  6.28318530717959E+0000}%
\put(19.7000,-2.4500){\makebox(0,0){$[1,1]$}}%
\put(29.7000,-2.4000){\makebox(0,0){$[1,2]$}}%
\put(55.7000,-2.5000){\makebox(0,0){$[1,l_1]$}}%
\put(19.7000,-6.5500){\makebox(0,0){$[2,1]$}}%
\put(29.7000,-6.4500){\makebox(0,0){$[2,2]$}}%
\put(55.7000,-6.5500){\makebox(0,0){$[2,l_2]$}}%
\put(19.7000,-17.8500){\makebox(0,0){$[n,1]$}}%
\put(29.7000,-17.8500){\makebox(0,0){$[n,2]$}}%
\put(55.7000,-17.8500){\makebox(0,0){$[n,l_n]$}}%
\put(14.3000,-7.6000){\makebox(0,0){$0$}}%
\special{pn 8}%
\special{sh 1}%
\special{ar 2950 1010 10 10 0  6.28318530717959E+0000}%
\special{sh 1}%
\special{ar 2950 1210 10 10 0  6.28318530717959E+0000}%
\special{sh 1}%
\special{ar 2950 1410 10 10 0  6.28318530717959E+0000}%
\special{sh 1}%
\special{ar 2950 1410 10 10 0  6.28318530717959E+0000}%
\special{pn 8}%
\special{ar 1110 1000 290 220  0.4187469 5.9693013}%
\special{pn 8}%
\special{pa 1368 1102}%
\special{pa 1376 1090}%
\special{fp}%
\special{sh 1}%
\special{pa 1376 1090}%
\special{pa 1324 1138}%
\special{pa 1348 1136}%
\special{pa 1360 1158}%
\special{pa 1376 1090}%
\special{fp}%
\special{pn 8}%
\special{ar 910 1000 510 340  0.2464396 6.0978374}%
\special{pn 8}%
\special{pa 1400 1096}%
\special{pa 1406 1084}%
\special{fp}%
\special{sh 1}%
\special{pa 1406 1084}%
\special{pa 1362 1138}%
\special{pa 1384 1132}%
\special{pa 1398 1152}%
\special{pa 1406 1084}%
\special{fp}%
\special{pn 8}%
\special{sh 1}%
\special{ar 540 1000 10 10 0  6.28318530717959E+0000}%
\special{sh 1}%
\special{ar 620 1000 10 10 0  6.28318530717959E+0000}%
\special{sh 1}%
\special{ar 700 1000 10 10 0  6.28318530717959E+0000}%
\special{pn 8}%
\special{ar 1200 1000 170 100  0.7298997 5.6860086}%
\special{pn 8}%
\special{pa 1314 1076}%
\special{pa 1328 1068}%
\special{fp}%
\special{sh 1}%
\special{pa 1328 1068}%
\special{pa 1260 1084}%
\special{pa 1282 1094}%
\special{pa 1280 1118}%
\special{pa 1328 1068}%
\special{fp}%
\end{picture}%
\end{center}

\vspace{10pt}

\begin{theorem}\label{4.4.1}
Let $X$ be a compact Riemann surface with genus $g >0$, 
and let $D=\{ p_1, \dots ,p_n \}$ 
be a finite subset of $X$ with cardinality $n$. 
Take an arbitrary $l \in \N^D$, 
and let $\xi =(\xi_p^j \mid p \in D,~j=0,\dots , l_p)$ be a tuple of 
non-zero complex numbers, 
$\beta = (\beta_p^j \mid p \in D,~j=0,\dots , l_p)$ be a tuple of 
rational numbers such that $\beta_p^i < \beta_p^j$ for any $p$ and $i <j$.
Take a star-shaped quiver $(I,\Omega)$ with $g$ loops as above, 
such that the number of arms is $n$ and the length of the $i$-th arm is $l_{p_i}$.   
Then for any $I$-graded vector space $V$,  
setting $(q,\theta) \in (\C^{\times})^I \times \Q^I$ by
\begin{align*}
\theta_{i,j} &:= \beta_{p_i}^j -\beta_{p_i}^{j-1}, \qquad 
\theta_0 := -\frac{\sum_{[i,j] \in I_0}\theta_{i,j} \dim V_{i,j}}{\dim V_0}, \\
q_{i,j} &:= \xi_{p_i}^{j-1}/\xi_{p_i}^{j}, \quad \qquad q_0 := \prod_i (\xi_{p_i}^0)^{-1}, 
\end{align*}
there is a natural bijection between $\MMl_{q,\theta}(V)$
and the set of isomorphism classes of 
$\beta$-polystable filtered local systems $(L,\bF)$ on $(X,D)$ 
satisfying: 
\begin{itemize}
\item $\rank L =\dim V_0,~\rank \bF_{p_i}^j(L)=\dim V_{i,j}$; 
\item the local monodromy of $\bF_{p_i}^j(L)/\bF_{p_i}^{j+1}(L)$ around $p_i$ is given by  
the scalar multiplication by $\xi_{p_i}^j$ for all $i,j$.
\end{itemize} 
Under this map, a point in $\MMlreg_{q,\theta}(V)$ corresponds to 
an isomorphism class of $\beta$-stable filtered local systems. 
\end{theorem}

\begin{theorem}\label{4.4.2}  
Let $X$ be a compact Riemann surface with genus $g >0$, 
and let $D=\{ p_1, \dots ,p_n \}$ 
be a finite subset of $X$ with cardinality $n$.   
Under the same notation and assumptions as in Theorem~\ref{4.3.4},   
assume further that $n>1$ if $g=1$,  
$l_p=r-1$ for all $p$ and fixed $r>0$, 
and that $\alpha$ is generic so that $\MMreg_{\lambda,\alpha}(X,D;r)=\MM_{\lambda,\alpha}(X,D;r)$.         
Let $(I,\Omega)$ be a star-shaped quiver with $g$ loops 
such that the number of arms is $n$ and the length of each arm is $r-1$, 
and set $q,\theta$ as in Theorem~\ref{4.4.1}. 
Then Simpson's Riemann-Hilbert correspondence 
gives a symplectic biholomorphic map  
between $\MM_{\lambda,\alpha}(X,D;r)$ and $\MMlreg_{q,\theta}(V)=\MMl_{q,\theta}(V)$, 
where $V$ is given by $V_0=\C^r$, $V_{i,j}=\C^{r-j}$. 
\end{theorem} 

We omit proofs of the above two theorem, 
since they are almost the same as in the previous two subsections. 
Notice only that the fundamental group of a punctured Riemann surface $X \setminus D$ 
of genus $g >0$ has a presentation  
\[
\langle \alpha_1,\beta_1, \dots ,\alpha_g,\beta_g, \gamma_1, \dots ,\gamma_n \mid 
\alpha_1 \beta_1 \alpha_1^{-1}\beta_1^{-1}
\alpha_2 \beta_2 \alpha_2^{-1}\beta_2^{-1} \cdots 
\alpha_g \beta_g \alpha_g^{-1}\beta_g^{-1}
\gamma_1 \cdots \gamma_n =1 \rangle.
\]

Let us back to the case of an arbitrary quiver and 
consider similarities between $\MMl$ and $\M$. 
Similar to Proposition~\ref{3.3.1} 
we can show the following. 

\begin{proposition}\label{4.4.3}  
Let $0^\ell \in \bM(V)$ denote the point whose component $x_h$ is given by 
$x_h =1$ for $h \in H^\ell$ and $x_h =0$ for $h \in H \setminus H^\ell$. 
Then there are $\varphi$-saturated open neighborhoods  
$\mc{U}$, $\mc{U}'$ of $0^\ell$ in $\bM(V)$    
and a $G_V$-equivariant biholomorphic map 
$f\colon \mc{U} \to \mc{U}'$ such that
\[
f(0^\ell )=0^\ell, \qquad f(\Psi_V^{-1}(1)\cap \mc{U})=\mu_V^{-1}(0) \cap \mc{U}',
\qquad (f^* \omega - \varpi^\ell) |_{\Ker d\Psi_V} =0,
\] 
where $\varpi^\ell$ is the 2-form associated to the quasi-Hamiltonian $G_V$-structure on $\bMl(V)$. 
\end{proposition}

\begin{proof}
First notice that $0^\ell \in \Psi_V^{-1}(1) \cap \mu_V^{-1}(0)$.  
The 2-form $\varpi^\ell$ on $\bMl(M)$ is given by 
\[
\begin{split}
\varpi^\ell &:= \frac12 \sum_{h  \in H \setminus H^\ell}\epsilon(h )\Tr\,(1+x_h x_{\ov h })^{-1}dx_h 
\wedge dx_{\ov h} \\
&\quad +\frac12 \sum_{h \in H^\ell} \epsilon(h) \Tr x_h^{-1}dx_h \wedge dx_{\ov h} \, x_{\ov h}^{-1} \\
&\quad + \frac12 \sum_{h \in \Omega^\ell} \Tr\, 
[x_{\ov h},x_h]^{\rm m} d(x_hx_{\ov h}) \wedge d(x_h^{-1}x_{\ov h}^{-1}) \\
&\quad +\frac12 \sum_{h  \in \Omega^\ell}\Tr \,\Psi_h^{-1}d\Psi_h 
\wedge d[ x_h, x_{\ov h} ]^{\rm m}\, [ x_{\ov h},x_h ]^{\rm m} \\
&\quad +\frac12 \sum_{h  \in H \setminus H^\ell}\Tr \,\Psi_h^{-1}d\Psi_h 
\wedge d(1+x_h x_{\ov h})^{\epsilon(h)}(1+x_h x_{\ov h})^{-\epsilon(h)},
\end{split} 
\]
where
\[
\Psi_h := 
\begin{cases}
\prod^<_{h \in H_i \cap \Omega^\ell} [ x_h, x_{\ov h} ]^{\rm m}
& \text{if}\ h \in \Omega^\ell, \\
\prod^<_{h \in H_i \cap \Omega^\ell} [ x_h, x_{\ov h} ]^{\rm m}
\prod_{h' \in H_i ; h' <h}^< 
(1+x_{h'} x_{\ov{h'}})^{\epsilon(h')} 
& \text{if}\ h \in H_i \setminus H^\ell .
\end{cases}
\]
For a loop $h \in \Omega^\ell$, 
we have $d(x_h x_{\ov h})_{0^\ell} = dx_h + dx_{\ov h} = - d(x_h^{-1} x_{\ov h}^{-1})_{0^\ell}$. 
Thus we get 
\[
\varpi^\ell_{0^\ell} = \frac12 \sum_{h  \in H \setminus H^\ell}\epsilon(h )\Tr\, dx_h \wedge dx_{\ov h} 
+\frac12 \sum_{h \in H^\ell} \epsilon(h) \Tr dx_h \wedge dx_{\ov h} = \omega_{0^\ell}. 
\]
Moreover we have $(d\Psi_V)_{0^\ell}=0$. 
Hence Lemma~\ref{3.3.2} and the equivariant Darboux theorem imply the assertion.  
\end{proof}

Since $0^\ell$ is a fixed point for the $G_V$-action and $\mu_V(0^\ell)=0$, 
applying the equivariant Darboux theorem again  
one can find a $\varphi$-saturated open neighborhood $\mc{U}''$ of  
$0 \in \bM(V)$ and a $G_V$-equivariant biholomorphic map 
$F \colon \mc{U}' \to \mc{U}''$ such that 
\[
F(0^\ell ) =0, \qquad F^* \omega = \omega, \qquad \mu_V \circ F = \mu_V .
\]  
Together with the above proposition, we get:
\begin{corollary}\label{4.4.4}
Let $\MMl_\theta(V) = \MMl_{1,\theta}(V)$ and 
$\pi \colon \MMl_\theta(V) \to \MMl_0(V)$ denote the natural projective morphism. 
Then there exist an open neighborhood $U$ {\rm (}resp.\ $U'${\rm )} of $[0^\ell] \in \MMl_0(V)$ 
{\rm (}resp.\ $[0] \in \M_0(V)${\rm )}
and a commutative diagram   
\[
\begin{CD}
\MMl_{\theta}(V) \supset @.\, \pi^{-1}(U) @>{\tilde f}>> \pi^{-1}(U')\, @. \subset \M_{\theta}(V) \\
@. @V{\pi}VV @V{\pi}VV @. \\
@. U @>{f}>> U' @. 
\end{CD}
\]
such that:
\begin{enumerate}
\item[\rm (i)] $f([0^\ell])=[0]$; 
\item[\rm (ii)] both $\tilde{f}$ and $f$ are complex analytic isomorphisms; 
\item[\rm (iii)] $\tilde f$ maps $\pi^{-1}(U) \cap \MMlreg_\theta(V)$ onto 
$\pi^{-1}(U') \cap \Mreg_\theta(V)$ as a symplectic biholomorphic map; and
\item[\rm (iv)] if $x \in \varphi^{-1}(U)$ and $y \in \varphi^{-1}(U')$ 
have closed orbits and $f([x])=[y]$, then the stabilizers of the two are conjugate. 
Thus $f$ preserves the orbit-type.      
\end{enumerate} 
\end{corollary} 

\section{Middle convolution}\label{5}

Multiplicative preprojective relation has a certain surprising similarity to preprojective relation. 
Let $i\in I$ be a {\em loop-free} vertex, i.e.,   
there is no $h \in H$ such that $\vout(h)=\vin(h)=i$. 
In this section we fix such an $i\in I$.  
Let $s_i \colon \C^I \to \C^I$ be the reflection defined by   
$s_i(\alpha) := \alpha - (\alpha,\mbe_i) \mbe_i$. 
There is a reflection $r_i \colon \C^I \to \C^I$ 
which is dual to $s_i$ with respect to the standard inner product: 
\[
r_i(\zeta) :=(\zeta'_j), \qquad  \zeta'_j =\zeta_j - (\mbe_i ,\mbe_j) \zeta_i.
\]
Then the $i$-th {\em reflection functor} is defined as 
a certain equivalence  
between the category of representations $(V,x)$ of $(I,H)$ 
satisfying the preprojective relation $\mu_V(x) =\zeta$ 
with a fixed $\zeta$ such that $\zeta_i \neq 0$,   
and the category of those $(V',x')$ satisfying the preprojective relation  
$\mu_{V'}(x') =r_i(\zeta)$. This functor transforms  
the dimension vector $\dim V$ to $\dim V' =s_i(\dim V)$.     
Crawley-Boevey and Shaw~\cite{CS} constructed its multiplicative analogue 
by generalizing an algebraic formulation of 
Katz' middle convolution given by Dettweiler-Reiter~\cite{DR}. 
In other words, they constructed an equivalence between the category 
of the representations $(V,x)$  
satisfying the relation $\Phi_V(x) =\zeta$ 
with a fixed $q$ such that $q_i \neq 1$, 
and the category of those $(V',x')$ 
satisfying $\Phi_{V'}(x') =u_i(q)$, where  
\[
u_i(q) :=(q'_j), \qquad q'_j = q_j q_i^{-(\mbe_i ,\mbe_j)}.
\] 
This functor is called the {\em middle convolution functor}.

On the other hand, Maffei~\cite{Maf} showed that 
the reflection functor sends 
a $\theta$-stable representation to a $r_i(\theta)$-stable representation 
if $\dim V \neq \mbe_i$. 
Thus the reflection functor induces an isomorphism 
\[
\Mreg_{\zeta, \theta}(V) \simeq \Mreg_{r_i(\zeta),r_i(\theta)}(V'),
\]
where $V'$ is an $I$-graded vector space with $\dim V' = s_i(\dim V)$.  
Moreover he proved that the above isomorphism can be defined in the case of $\zeta_i =0$.

In this section we show a multiplicative version of his result. 
\begin{theorem}\label{5.0.1}   
If $\dim V \neq \mbe_i$ and $s_i(\dim V) \notin \N^I$, then $\MMreg_{q,\theta}(V)$ is empty. 

If $s_i(\dim V) \in \N^I$, 
take an $I$-graded vector space $V'$ with $\dim V' =s_i(\dim V)$. 
Then there is an isomorphism of algebraic varieties
\[
\MMreg_{q,\theta}(V)  
\simeq \MMreg_{u_i(q),r_i(\theta)}(V').
\]
\end{theorem}

The proof of the first statement is easy. 
Indeed if $\dim V \neq \mbe_i$ and $\MMreg_{q,\theta}(V) \neq \emptyset$, 
then Proposition~\ref{3.2.8} implies 
\[
0 \geq \dim \widehat{V}_i  - \dim V_i = -(\mbe_i ,\dim V ) + \dim V_i.
\]
The right hand side is just the coefficient of $s_i (\dim V)$ in $\mbe_i$, 
so $s_i(\dim V) \in \N^I$. 
Moreover if we assume further that $\theta_i =0$ and $q_i =1$, 
then the sequence
\[
\begin{CD}
0 @>>> V_i @>{\sigma_i}>> \widehat{V}_i  @>{\tau_i}>> V_i @>>> 0
\end{CD}
\]
is exact at any point in $\MMreg_{q,\theta}(V)$ by Proposition~\ref{3.2.8} again. 
The exactness implies that the right hand side of the previous inequality is equal to $\dim V_i$, 
and hence that $s_i(\dim V) = \dim V$. 
Since we have $r_i(\theta)=\theta,~u_i(q)=q$ under the assumption, 
the second statement in the case $\theta_i =0, q_i=1$ is clear.     

The rest of this section is devoted to the proof of the general case. 
In fact, all the proofs are very similar to the case of reflection functor. 

\subsection{Middle convolution functor}\label{5.1}

First we rewrite the middle convolution functor in our context. 
See \cite{CS} for the original definition.

From now on, we assume that $s_i(\dim V) \in \N^I$. 
Note that $\dim V \neq \mbe_i$ under this assumption. 
For simplicity, we assume further that $H_i \subset \Omega$.

Let us recall the definitions of $\sigma_i(x)$ and $\tau_i(x)$:
\begin{align*}
\sigma_i(x) &=\sum \iota_h x_{\ov h}  
\colon V_i \to  \widehat{V}_i , \\
\tau_i(x) &=\sum_{h \in H_i} \Phi_h x_h \pi_h 
\colon \widehat{V}_i  \to V_i,
\end{align*}
where
\[
\Phi_h =\Phi_h(x) = \prod_{h'\in H_i; h'<h}^<
(1+x_{h'}x_{\ov{h'}}).
\]
Since $i$ is fixed, we will drop the subscript $i$; 
$\sigma=\sigma_i,~\tau=\tau_i$.

For a point $x \in \Phi_{V}^{-1}(q)$, 
we define
\[
\phi_h := \sum_{h'\in H_i; h'<h } 
\iota_{h'} x_{\ov{h'}}x_h
+\frac{1}{q_i}\sum_{h'\in H_i; h'\geq h} 
\iota_{h'} x_{\ov{h'}}x_h
+\frac{1-q_i}{q_i}\iota_h \colon V_{\vout(h)}\to \widehat{V}_i  \quad (h \in H_i).
\]
Then one can show that 
\begin{gather} 
\tau \phi_h =0 \quad \text{for all}\ h\in H_i,\label{5.1.1}\\
\intertext{and that}
\prod_{h\in H_i}(1+\phi_{h}\pi_{h}) 
= 1- \frac{1}{q_i}(q_i -1 -\sigma\tau).\label{5.1.2} 
\end{gather}
For the proof, see \cite{CS}.

Now suppose $q_i\neq 1$.
The equality $\tau \sigma =q_i -1$ implies that 
$\tau$ is surjective.
Thus if we set $V'_i := \Ker \tau$ and $V'_j := V_j$ for $j \neq i$, 
then $\dim V'_j = s_i (\dim V)$.
 
Using \eqref{5.1.1} we define
\begin{equation}\label{5.1.3}
x'_h := 
\begin{cases}
\,\phi_h \colon V_{\vout(h)}\to V'_i & \text{if}\ h\in H_i, \\
\,\pi_{\ov h} |_{\Ker \tau}\colon V'_i \to V_{\vin(h)} & \text{if}\ h\in \ov{H}_i, \\
\,x_h & \text{otherwise},
\end{cases} 
\end{equation}

Then \eqref{5.1.2} implies
\[
\Phi_i(x')
=\prod_{h\in H_i}(1+x'_{h}x'_{\ov h}) 
= \frac{1}{q_i} = q'_i.
\]
Moreover for any $h\in H_i$ we have
\[
x'_{\ov h}x'_h = \pi_h \phi_h
= \frac{1}{q_i}x_{\ov h}x_h + \frac{1-q_i}{q_i},
\]
and hence
\[
1+x'_{\ov h}x'_h = \frac{1}{q_i} (1+x_{\ov h}x_h).
\]
Thus we get
\[
\Phi_j(x') = q_i^{\bA_{ij}}\Phi_j(x) =q'_j
\]
for all $j\neq i$, where 
$\bA_{ij}$ is the number of $h\in H$ satisfying $\vin(h)=i$ and $\vout(h)=j$. 

Thus under the assumption $q_i \neq 1$,  
we have a map 
\[
S_i \colon \Phi_{V}^{-1}(q)/G_V \to \Phi_{V'}(q)/G_{V'}; \qquad 
G_V \cdot x \mapsto G_{V'} \cdot x' 
\]
between the set-theoretical orbit spaces. 
This is a set-theoretical definition of the {\em middle convolution functor}.  

Crawley-Boevey and Show observed that $S_i^2 = \mathrm{id}$.  
We use this map to prove Theorem~\ref{5.0.1} for the case $q_i \neq 1$ or $\theta_i <0$. 
Note that  
even if $q_i =1$,  
the above definition of $x'$ for a $\theta$-stable point $x \in \Phi_V^{-1}(q)$ with $\theta_i \leq 0$ 
makes sense since $\tau(x)$ is still surjective by Proposition~\ref{3.2.8}. 
Thus we have a map $S_i \colon \MMreg_{q,\theta}(V) \to \Phi_{V'}(q)/G_{V'}$ in 
the case $\theta_i \leq 0$.

\subsection{Lusztig's correspondence}\label{5.2}

To prove Theorem~\ref{5.0.1}, 
we modify a beautiful formulation of the reflection functor  
by Lusztig~\cite{Lus-ref} for the middle convolution.

From now on, we assume that $\theta_i \leq 0$ and 
$\epsilon(h)>0$ for all $h \in H_i$ as in the previous subsection. 
Both of the assumption lose no generality by $r_i^2 = \mathrm{id}$ and Proposition~\ref{3.1.3}. 
Moreover we exclude the case $\theta_i =0, q_i=1$ as we explained before. 

Set $q':=u_i(q)$ and $\theta':=r_i(\theta)$. 
Take an $I$-graded vector space $V'$ such that $\dim V' =s_i(\dim V)$ 
and $V'_j = V_j$ for all $j\neq i$.   

In this section we use the following notation:
\begin{gather*}
\begin{aligned}
\bM  
&=\bM (V), \\
\bM' 
&=\bM (V'), 
\end{aligned}\quad
\begin{aligned}
Z &=\bMs_\theta (V) \cap \Phi_{V}^{-1}(q), \\
Z'&=\bMs_{\theta'}(V') \cap \Phi_{V'}^{-1}(q').
\end{aligned}
\end{gather*}

\begin{definition}\label{5.2.1}
Let $P$ be the subvariety of $\bM \times \bM'$
which consists of all pairs $(x,x')\in \bM \times \bM'$ 
satisfying the following conditions:
\begin{enumerate}
\item[\rm (R1)] $x_h =x'_h$ for all $h \notin H_i \cup \ov{H}_i$. 
\item[\rm (R2)] The sequence 
\[
\begin{CD}
0 @>>> V'_i @>{\sigma'}>> \widehat{V}_i  @>{\tau}>> V_i @>>> 0
\end{CD}
\]
is exact. Here $\sigma'=\sigma(x')$. 
\item[\rm (R3)] $\sigma\tau = q_i \sigma' \tau' + q_i -1$. Here $\tau'=\tau(x')$. 
\item[\rm (R4)] $\det(1+x_hx_{\ov h})\neq 0$ for all $h \in H$. 
\item[\rm (R4')] $\det(1+x'_h x'_{\ov h})\neq 0$ for all $h\in H$.   
\item[\rm (R5)] $\Phi_{V}(x)=q$. 
\item[\rm (R5')] $\Phi_{V'}(x')=q'$.  
\item[\rm (R6)] $x$ is $\theta$-stable. 
\item[\rm (R6')] $x'$ is $\theta'$-stable. 
\end{enumerate}
\end{definition}

Let $r\colon P \to Z$ (resp.\ $r'\colon P\to Z'$) be
the map induced from the projection to the first (resp.\ the second) factor.
$P$ is naturally acted on by the reductive group
\[
G:=\GL(V_i) \times \GL(V'_i) \times \prod_{j\neq i} \GL(V_j),
\]
and $r$ (resp.\ $r'$) is equivariant through the projections 
$G \to G_V$ (resp.\ $G\to G_{V'}$).
Note that $P$ has a geometric quotient, 
because $Z\times Z'$ has a geometric quotient
for the action of $G_V \times G_{V'}$ 
and hence so for the action of its reductive subgroup $G$, 
and $P$ is a $G$-invariant subvariety of $Z\times Z'$.   

The second statement of Theorem~\ref{5.0.1} is deduced from the following fact. 

\begin{theorem}\label{5.2.2}
Suppose $\theta_i\leq 0$, and $q_i \neq 1$ if $\theta_i =0$.
Then
$r$ and $r'$ induce isomorphisms
\[
\MMreg_{q,\theta}(V) \simeq P/ G 
\simeq \MMreg_{q',\theta'}(V').
\]
\end{theorem}

We give a proof of this theorem in \S\ref{5.4}. 
In the next subsection,  
we give several properties of $P$, 
all of which are needed in \S\ref{5.4}.  

\subsection{Several lemmas}\label{5.3}

\begin{lemma}\label{5.3.1}
Suppose $\theta_i \leq 0$, and $q_i \neq 1$ if $\theta_i =0$. 
If a point $(x,x') \in \bM \times \bM'$ satisfies the conditions 
{\rm (R1)}, {\rm (R2)} and {\rm (R3)}, then
\[
\text{$(x,x')$ satisfies \rm (R6)} \quad \iff \quad 
\text{$(x,x')$ satisfies \rm (R6')}.
\]
\end{lemma}

\begin{proof}
We adapt a beautiful proof of Nakajima~\cite{Nak-ref} for the reflection functor to our case.

First we prove the direction $\Rightarrow$. 
Suppose that (R1-3) and (R6). 
If $\theta_i =0$, suppose further that $q_i \neq 1$.
 
Let $S'$ be a $x'$-invariant subspace of $V'$.
Then
\begin{equation}\label{5.3.2}
\sigma'(S'_i) \subset \widehat{S}'_i ,\qquad 
\tau' ( \widehat{S}'_i ) \subset S'_i.
\end{equation}
Set
\begin{equation}\label{5.3.3}
S_j :=
\begin{cases}
S'_j & \text{for}\ j\neq i, \\
\tau ( \widehat{S}'_i ) & \text{for}\ j=i.
\end{cases}
\end{equation}
Clearly $x_h(S_{\vout(h)}) \subset S_{\vin(h)}$ if 
$\vin(h) \neq i \neq \vout(h)$.
By \eqref{5.3.2}, we have
\[
\begin{split}
\sigma(S_i) &=\sigma\tau ( \widehat{S}_i ) \\
&=q_i \sigma'\tau' ( \widehat{S}_i ) 
+ (q_i-1) ( \widehat{S}_i ) \\
&\subset \sigma'(S'_i) +\widehat{S}_i 
\subset \widehat{S}_i . 
\end{split}
\]
Thus $S$ is $x$-invariant by Lemma~\ref{3.2.3}.

By the $\theta$-stability of $x$ we have
\begin{equation}\label{5.3.4}
0 \geq \theta \cdot \dim S = \sum_{j\neq i} \theta_j \dim S_j + \theta_i \dim S_i
\end{equation}
and the strict inequality holds unless $S=0$ or $S=V$.

Consider the following complex.
\[
\begin{CD}
S'_i @>{\sigma'}>> \widehat{S}_i  @>{\tau}>> S_i 
\end{CD}
\]
The left arrow is injective by (R2)
and the right arrow is surjective by the definition of $S_i$.
Hence we have
\begin{equation}\label{5.3.5}
\dim S_i \leq \sum \dim S_{\vout(h)} -\dim S'_i.
\end{equation}
Noticing $\theta_i \leq 0$, we substitute this inequality into~\eqref{5.3.4}.
Then we get
\[
0 \geq \sum_{j\neq i} (\theta_j +\bA_{ij}\theta_i) \dim S_j - \theta_i \dim S'_i
=\theta' \cdot \dim S'.
\]
If we have the equality, we must have the equality in
\eqref{5.3.4} which implies $S=0$ or $S=V$.
If $S=0$, then $S'_j=0$ for $j\neq i$.
Then~\eqref{5.3.2} and the injectivity of $\sigma'$
imply $S'_i=0$. Thus $S'=0$. 
We assume $S=V$. 
When $\theta_i \neq 0$, we must also have the equality in
\eqref{5.3.5}. Substituting $S=V$ into it,
we obtain $\dim S'_i = \dim V'_i$. Thus $S'=V'$.
When $\theta_i=0$, $q_i \neq 1$ by the assumption. 
Thus by (R2) and (R3) we have 
\[
0 =\sigma \tau \sigma' = q_i \sigma' \tau' \sigma' +(q_i -1)\sigma'. 
\]
By (R2) $\sigma'$ is injective, so we have $\tau'\sigma' = q_i^{-1} -1 \neq 0$.  
Thus $\tau'$ is surjective and hence $S'_i \supset \tau'(\widehat{S'_i}) = V'_i$. 
Thus $S'=V'$. 
Hence $x'$ is $\theta'$-stable.

The proof of the inverse direction $\Leftarrow$ 
also can be done similarly.
Let $S$ be a $x$-invariant subspace of $V$.
Set
\[
S'_j :=
\begin{cases}
S_j & \text{for}\ j\neq i, \\
(\sigma')^{-1} ( \widehat{S}_i ) & \text{for}\ j=i.
\end{cases}
\]
Then $S'$ is $x'$-invariant.

By the $\theta'$-stability of $x'$ we have
\begin{equation}\label{5.3.6}
0 \geq (\theta', \dim S') = \sum_{j\neq i} \theta'_j \dim S'_j + \theta'_i \dim S'_i
\end{equation}
and we have the strict inequality unless $S'=0$ or $S'=V'$.

Consider the following complex.
\[
\begin{CD}
S'_i @>{\sigma'}>> \widehat{S}_i  @>{\tau}>> S_i
\end{CD}
\]
The left arrow is injective by (R2)
and its image is equal to the kernel of the right arrow 
by the definition of $S'_i$ and (R2).
Hence we have
\begin{equation}\label{5.3.7}
\dim S'_i \geq \sum \dim S_{\vout(h )} -\dim S_i.
\end{equation}
Noticing $\theta'_i \geq 0$, we substitute this inequality into~\eqref{5.3.6}.
Then we get
\[
0 \geq \sum_{j\neq i} (\theta'_j +\bA_{ij}\theta'_i) \dim S_j - \theta'_i \dim S_i
=(\theta, \dim S).
\]
If we have the equality, we must have the equality in
\eqref{5.3.6} which implies $S'=0$ or $S'=V'$.
If $S'=V'$, then $S_j =S_j$ for $j\neq i$.
Thus $S_i \supset \tau ( \widehat{S}_i ) =V_i$
by the surjectivity of $\tau$. Hence $S=V$.
We assume $S'=0$. 
When $\theta_i \neq 0$, we must also have the equality in~\eqref{5.3.7}.
This implies $S_i=0$, and hence $S=0$.
When $\theta_i=0$, the conditions (R2), (R3) and the assumption $q_i \neq 1$ 
implies 
\[
0= \tau \sigma' \tau' = q_i^{-1}\tau \sigma \tau +(q_i^{-1}-1)\tau.
\]
By (R2), $\tau$ is surjective, so we have $\tau \sigma = q_i -1 \neq 0$. 
Thus $\sigma$ is injective and hence  
$S_i \subset \sigma^{-1}( \widehat{S_i} ) =0$. Thus $S=0$.
Hence $x$ is $\theta$-stable.
\end{proof}

\begin{lemma}\label{5.3.8}
Suppose $\theta_i \geq 0$, and $q_i \neq 1$ if $\theta_i =0$.
If a point $(x,x') \in \bM \times \bM'$ satisfies {\rm (R1)}, {\rm (R3)} and
\begin{enumerate}
\item[\rm (R2')] The sequence
\[
\begin{CD}
0 @>>> V_i @>{\sigma}>> \widehat{V}_i  @>{\tau'}>> V'_i @>>> 0
\end{CD}
\] 
is exact,
\end{enumerate}
then
\[
\text{$(x,x')$ satisfies \rm (R6)} \quad \iff \quad 
\text{$(x,x')$ satisfies \rm (R6')}.
\]
\end{lemma}

\begin{proof}
The proof is similar to the previous lemma.
\end{proof}

\begin{proposition}\label{5.3.9}
Suppose that $q_i \neq 1$ or $\theta_i < 0$.
For a point $x \in Z$,
let $x'$ be a representative of $S_i(G_V \cdot x)$ which is defined by \eqref{5.1.3}.
Then $(x,x') \in P$. 
\end{proposition}

\begin{proof}
(R1) is satisfied by the definition of $x'$.
Moreover both (R4') and (R5') are satisfied by the argument before. 

To check (R2),
first note that
\[
\sigma' =\sum_{h\in H_i} \iota_h x'_{\ov h } 
=\sum_{h\in H_i} \iota_h \pi_h 
\]
is equal to the inclusion $V'_i=\Ker \tau \hookrightarrow \widehat{V}_i$.
Thus $\sigma'$ is injective and $\tau \sigma' =0$.
Since $\tau$ is surjective  
and the Euler number of the complex in (R2) is zero,
$(x,x')$ satisfies (R2).

By \eqref{5.1.2}, we have
\[
\begin{split}
\tau' &= \sum_{h \in H_i} 
\prod_{h'\in H_i; h' <h }(1+x'_{h'}x'_{\ov{h'}})\, x'_{h}\pi_{h} \\
&= \sum_{h \in H_i} 
\prod_{h'\in H_i; h' <h }(1+\phi_{h'}\pi_{h'})\, \phi_{h}\pi_{h} \\
&= \prod_{h \in H_i} (1+\phi_{h}\pi_{h}) -1 \\
&= q_i^{-1}\sigma\tau + q_i^{-1}(1-q_i).
\end{split}
\]
Thus (R3) is satisfied.

Lemma~\ref{5.3.1} shows that $x'$ is $\theta'$-stable when $\theta_i < 0$.
So we assume that $\theta_i \geq 0$ and $q_i\neq 1$.
By (R3) and the equality $\tau'\sigma' =q'_i -1$, we have
\[
\begin{split}
\tau'\sigma \tau &= q_i \tau'\sigma'\tau' +(q_i -1)\tau' \\
&= q_i (q'_i -1) \tau' +(q_i -1)\tau' =0.
\end{split}
\]
Since $\tau$ is surjective, the above implies $\tau'\sigma =0$.
Note that $\tau'$ is surjective and $\sigma$ is injective 
by the equalities $\tau'\sigma'=q'_i-1$ and $\tau \sigma=q_i -1$.
Hence the sequence
\[
\begin{CD}
0 @>>> V_i @>{\sigma}>> \widehat{V}_i  @>{\tau'}>> V'_i @>>> 0
\end{CD}
\] 
is exact.
Thus $x'$ is $\theta'$-stable by Lemma~\ref{5.3.8}.
\end{proof}

\begin{proposition}\label{5.3.10}
Suppose $q_i \neq 1$ or $\theta'_i <0$.
For a point $x\in Z'$,
let $x'$ be a representative of $S_i(G_{V'} \cdot x)$.
Then $(x',x) \in P$. 
\end{proposition}

\begin{proof}
The proof is similar.
\end{proof}

\subsection{Proof of the main theorem}\label{5.4}

In this subsection we prove Theorem~\ref{5.2.2}.
First consider the case $q_i \neq 1$.

\begin{proof}[Proof of Theorem~\ref{5.2.2} for the case $q_i \neq 1$]
$r\colon P\to Z$ is surjective by Proposition~\ref{5.3.9}.
Let $x_0\in Z$. 
We construct a section of $r$ over a neighborhood of $x^0$.

Take an identification $\widehat{V}_i \simeq V_i \oplus V'_i$
such that the first projection coincides with $\tau(x^0)$.
Set
\[
Z_0 =\{ \, x\in Z \mid
\tau(x) |_{V_i} \colon V_i \to V_i
~\text{is an isomorphism}\,\} .
\]
Then $Z_0$ is a neighborhood of $x^0$,
and for any $x\in Z$, 
\[
\alpha:=
\begin{bmatrix} -(\tau(x) |_{V_i})^{-1}\tau(x) |_{V'_i} \\ 1 \end{bmatrix}
\colon V_i \to \Ker \tau(x)
\]
is an isomorphism.
We choose it for the identification $V_i \simeq \Ker \tau(x)$
to define the point $x' \in Z'$, 
i.e., we define 
\[
x'_h := \alpha^{-1} \phi_h \colon V'_{\vout(h)} \to V'_i, 
\qquad
x'_{\ov h} := \pi_h \alpha \colon V'_i \to V'_{\vout(h)}
\quad \text{for}\ h\in H_i,
\]
and define $x_h$ for $h\notin H_i \cup \ov{H}_i$
by the condition (R1).
Then $x \mapsto x'$ defines a section of $r$ over $Z_0$. 

Since $r$ has a local section,
the induced morphism $P/G \to Z/G_V$ 
is an isomorphism.
The proof for $r'$ is similar 
(use Proposition~\ref{5.3.10} instead of Proposition~\ref{5.3.9}).
\end{proof}

In the rest of this subsection
we assume that $q_i=1$
and $\theta_i < 0$.

\begin{lemma}\label{5.4.1}
If a pair $(x,x') \in \bM \times \bM'$ 
satisfies the conditions {\rm (R1)}, {\rm (R2)} and {\rm (R3)}, then 
\[
\text{$(x,x')$ satisfies \rm (R4)}
\quad \iff \quad \text{$(x,x')$ satisfies \rm (R4')}.
\]
And under these assumptions, 
the equality $x_{\ov h }x_h = x'_{\ov h }x'_h$ holds 
for all $h \in H_i$.
\end{lemma}

\begin{proof}
Let $H_i =\{\, h_1<h_2< \cdots <h_n \,\}$.
By (R3), 
\[
x_{\ov{h}_1} x_{h_1} =\pi_{h_1}\sigma\tau\iota_{h_1}
=\pi_{h_1}\sigma'\tau'\iota_{h_1}=x'_{\ov{h}_1}x'_{h_1},
\]
and 
\[
\det (1+x'_{h_1}x'_{\ov{h}_1})=\det (1+x'_{\ov{h}_1}x'_{h_1})
=\det (1+x_{\ov{h}_1}x_{h_1})=\det (1+x_{h_1}x_{\ov{h}_1}).
\]
Set $R_1 = \iota_{h_1}\pi_{h_1}\sigma\tau$.
Then 
\[
x_{h_1}x_{\ov{h}_1} \tau =\tau\iota_{h_1}\pi_{h_1}\sigma\tau =\tau R_1,
\]
and also 
\[
x'_{h_1}x'_{\ov{h}_1} \tau' =\tau'\iota_{h_1}\pi_{h_1}\sigma'\tau' =\tau' R_1.
\]
by (R3).
Suppose now that $(x,x')$ satisfies (R4).
Since $\det(1+R_1) =\det(1+\pi_{h_1}\sigma \tau \iota_{h_1})
=\det(1+x_{\ov{h}_1}x_{h_1}) \neq 0$,
$(1+R_1)$ is invertible and hence
\begin{align*}
x_{\ov{h}_2}x_{h_2} &= \pi_{h_2}\sigma (1+x_{h_1}x_{\ov{h}_1})^{-1} \tau \iota_{h_2} \\
&=\pi_{h_2}\sigma \tau (1+R_1)^{-1} \iota_{h_2} \\
&=\pi_{h_2}\sigma' \tau' (1+R_1)^{-1} \iota_{h_2} \\
&=\pi_{h_2}\sigma' (1+x'_{h_1}x'_{\ov{h}_1})^{-1}\tau' \iota_{h_2} \\
&=x'_{\ov{h}_2}x'_{h_2}.
\end{align*}
Next we define 
\[
R_2 =(1+R_1)^{-1}\iota_{h_2}\pi_{h_2}\sigma\tau.
\]
Then 
\begin{align*}
\det(1+R_2) &=\det(1+\iota_{h_2}\pi_{h_2}\sigma\tau (1+R_1)^{-1}) \\
&=\det(1+\iota_{h_2}\pi_{h_2}\sigma (1+x_{h_1}x_{\ov{h}_1})^{-1}\tau ) \\
&=\det(1+x_{\ov{h}_2}x_{h_2}) \neq 0,
\end{align*}
and 
\begin{align*}
x_{h_2}x_{\ov{h}_2}\tau 
&=(1+x_{h_1}x_{\ov{h}_1})^{-1}\tau\iota_{h_2}\pi_{h_2}\sigma\tau \\
&=\tau (1+R_1)^{-1}\iota_{h_2}\pi_{h_2}\sigma\tau 
=\tau R_2, \\
x'_{h_2}x'_{\ov{h}_2}\tau' 
&=(1+x'_{h_1}x'_{\ov{h}_1})^{-1}\tau'\iota_{h_2}\pi_{h_2}\sigma'\tau' \\
&=\tau' (1+R_1)^{-1}\iota_{h_2}\pi_{h_2}\sigma'\tau' 
=\tau' R_2. 
\end{align*}
By induction, one can easily show that
\[
R_k :=(1+R_{k-1})^{-1} \cdots (1+R_2)^{-1}(1+R_1)^{-1}
\iota_{h_k}\pi_{h_k}\sigma\tau
\]
is well-defined and 
\begin{align*}
\det (1+R_k) &=\det (1+x_{\ov{h}_k}x_{h_k}), \\
x_{h_k}x_{\ov{h}_k}\tau &= \tau R_k, \\
x'_{h_k}x'_{\ov{h}_k}\tau' &= \tau' R_k
\end{align*}
for $1\leq k \leq n$.
Hence for $1\leq k \leq n$,
\begin{align*}
x_{\ov{h}_k}x_{h_k} 
&= \pi_{h_k}\sigma (1+x_{h_{k-1}}x_{\ov{h}_{k-1}})^{-1} \cdots 
(1+x_{h_1}x_{\ov{h}_1})^{-1} \tau \iota_{h_k} \\
&=\pi_{h_k}\sigma \tau (1+R_{k-1})^{-1} \cdots (1+R_1)^{-1}
 \iota_{h_k} \notag\\
&=\pi_{h_k}\sigma' \tau' (1+R_{k-1})^{-1}\cdots (1+R_1)^{-1}
 \iota_{h_k} \\
&=\pi_{h_k}\sigma' (1+x'_{h_{k-1}}x'_{\ov{h}_{k-1}})^{-1}\cdots
(1+x'_{h_1}x'_{\ov{h}_1})^{-1}
\tau' \iota_{h_k} \\
&=x'_{\ov{h}_k}x'_{h_k}.
\end{align*}

The proof of the inverse direction (R4')\,$\Rightarrow$\,(R4) can be done similarly, so we omit it.
\end{proof}

\begin{lemma}\label{lem5.4.2}
If a pair $(x,x')$ satisfies {\rm (R1)}, {\rm (R2)} and {\rm (R3)},
then
\[
(x,x')~\text{satisfies {\rm (R4)} and \rm (R5)} \quad \iff \quad 
(x,x')~\text{satisfies {\rm (R4')} and \rm (R5')}.
\]
\end{lemma}

\begin{proof}
Under the conditions (R1), (R2), (R3)
and (R4) (or (R4')),
the above lemma implies that
$\Phi_j(x)=\Phi_j(x')$ for all $j\neq i$.
Since $q_j =q'_j$ for $j\neq i$, the result follows.
\end{proof}

\begin{proof}[Proof of Theorem~\ref{5.2.2} for the case $q_i=1$ and $\theta_i< 0$]
The proof for $r$
is the same that in the case $q_i \neq 1$.
To prove that $r'$ is a geometric quotient,
we will construct locally a section of $r'$,
as in the other case. 

Let $x' \in Z'$.
By Proposition~\ref{3.2.8} and $\theta'_i > 0$,
$\sigma'$ is injective.
Thus we can identify $V_i$ with $\widehat{V}_i /\Im \sigma'$.
Let $p$ be the projection $\widehat{V}_i \to V_i$.
Since $\tau'\sigma'=0$, $\tau'$ descends to a linear map 
$\ov{\tau}'\colon V_i \to V'_i$.
We define
\[
x_{\ov h} = x'_{\ov h} \ov{\tau}' \colon V_i \to V_{\vout(h)},\qquad
x_h = \Phi_h^{-1}p\iota_h \colon V_{\vout(h)} \to V_i
\]
for $h\in H_i$.
Here we use induction to define $x_h$.

We define $x_h$ for $h\notin H_i \cup \ov{H}_i$
by the condition (R1).
Then
\[
\sigma = \sigma'\ov{\tau}', \qquad \tau =p.
\]
Thus $\sigma \tau = \sigma' \ov{\tau}' p=\sigma' \tau'$
and $\tau \sigma'=p \sigma' =0$.
Clearly $\tau$ is surjective, so (R2) is satisfied. 
By Lemma~\ref{5.3.1} and Lemma~\ref{lem5.4.2}, 
the pair of $x$ and $x'$ is an element of $P$. 

The definition of $x$ depends on the identification 
$V_i \simeq \widehat{V}_i/\Im \sigma'$,
but we can choose it locally to be regular
in the variable $x'$,
as in the case of $r$.
Thus the induced morphism $P/G \to Z'/G_{V'}$
is an isomorphism.
\end{proof}

\section{Representations of Kac-Moody algebra}\label{6}

In \cite{Nak-Duke1}, Nakajima constructed all irreducible highest weight representations 
of a Kac-Moody Lie algebra using the vector spaces of constructible functions 
on the nilpotent subvarieties of the quiver varieties. 
In this section we observe that the same method can be applied to the case 
of the multiplicative quiver varieties. 

\subsection{Notation}\label{6.1} 

Suppose that the following data are given:
\begin{itemize}
\item $P$ --- a free $\Z$-module, called a {\em weight lattice}. 
\item $I$ --- an index set of simple roots.
\item $\alpha_i \in P~(i \in I)$ --- {\em simple root},
\item $h_i \in P^* := \Hom_\Z (P,\Z)~(i\in I)$ --- {\em simple coroot}.
\item $(\ , \ )$ --- a symmetric bilinear form on $P$.
\end{itemize}
These are required to satisfy:
\begin{enumerate}
\item[(i)] $\langle h_i, \lambda \rangle = 2(\alpha_i,\lambda)/(\alpha_i ,\alpha_i)$ for $i \in I$ and $\lambda\in P$; 
where $\langle \ , \ \rangle\colon P^* \otimes P \to \Z$ is the natural pairing;
\item[(ii)] $c_{ij} := \langle h_i, \alpha_j \rangle$ forms a {\em generalized Cartan matrix}, i.e., 
$c_{ii}=2$, $c_{ij} \in \Z_{\leq 0}\ (i \neq j)$ and $c_{ij}=0 \Leftrightarrow c_{ji}=0$; 
\item[(iii)] $(\alpha_i, \alpha_i ) \in 2\Z_{>0}$; 
\item[(iv)] $\{ \alpha_i \}_{i\in I}$ is linearly independent; and 
\item[(v)] there exists $\Lambda_i\in P$ ($i \in I$), called the {\em fundamental weight}, 
such that $\langle h_j, \Lambda_i\rangle = \delta_{ij}$.
\end{enumerate}

Such data are so-called {\em root data}, 
to which one associates a Kac-Moody Lie algebra $\g$~(see e.g.\ \cite{Kac}). 
Let $\bU$ be the universal enveloping algebra of $\g$. 
Recall the defining relations of it:
\begin{gather}
[h, h'] = 0 \quad \text{for}\ h, h' \in P^*, \label{6.1.1}\\
[h, e_i]= \langle h,\alpha_i \rangle e_i, \label{6.1.2}\\
[h, f_i]= -\langle h,\alpha_i \rangle f_i, \label{6.1.3}\\
[e_i, f_j]= \delta_{ij} h_i, \label{6.1.4}
\end{gather}
\begin{gather}
\sum_{n=0}^{1-c_{ij}} (-1)^n \binom{1-c_{ij}}{n}
 e_i^n e_j e_i^{1-c_{ij}-n} = 0 \quad (i \neq j), \label{6.1.5}\\
\sum_{n=0}^{1-c_{ij}} (-1)^n \binom{1-c_{ij}}{n}
 f_i^n f_j f_i^{1-c_{ij}-n} = 0 \quad (i \neq j). \label{6.1.6}
\end{gather}

We also use the following symbols: 
\begin{itemize}
\item $P^+ := \{ \lambda \in P \mid \langle h_i,\lambda \rangle \geq 0 \ \text{for any}\ i\in I \}$ 
(the semigroup of {\em dominant weights}),
\item $Q := \bigoplus_i \Z \alpha_i \subset P$ ({\em root lattice}),
\item $Q^+ := \sum_i \N \alpha_i \subset Q$.  
\end{itemize}

Let $(I,E)$ be the graph associated to $\bC$, i.e., 
the graph whose vertex set is $I$ and 
edge set $E$ is given by $2 \mathbf{I} - \bA = \bC$, 
where $\mathbf{I}$ is the identity matrix and $\bA$ is a matrix whose $(i,j)$ entry 
is just the number of edges joining $i$ and $j$.      
Let $(I,\Omega)$ be a quiver whose underlying graph is $(I,E)$.

\subsection{Framed multiplicative quiver variety}\label{6.2}

For $\bv \in Q^+$ and $\bw \in P^+$, 
we define a variety $\MM(\bv,\bw)$ 
which is a multiplicative analogue of the Nakajima quiver variety $\M(\bv,\bw)$. 

Following Crawley-Boevey~(see \cite[Introduction]{Cra-geom}), 
we associate to $(I,\Omega)$ and $\bw$ 
another quiver $(\tilde I, \tilde \Omega)$ by   
setting $\tilde I :=I \cup \{ \infty \}$ and  
letting $\tilde{\Omega}$ be the set obtained by adding $w_i$ arrows 
starting at $\infty$ toward $i$ for each $i \in I$ to $\Omega$, where $w_i := \langle h_i,\bw \rangle$.
 
Take an $I$-graded vector space $V$ such that $\sum_i (\dim V_i) \,\alpha_i =\bv$. 
To such $V$, we associate an $\tilde{I}$-graded vector space $\widetilde{V}$ 
by $\widetilde{V}_i := V_i$ and $\widetilde{V}_\infty := \C$.         
To a pair $(q,\theta) \in (\C^\times)^I \times \Z^I$, 
we associate a pair 
$(\tilde{q},\tilde{\theta}) \in (\C^\times)^{\tilde{I}} \times \Z^{\tilde{I}}$ 
by 
\begin{align*}
\tilde{q}_i &:= q_i, \qquad \tilde{q}_\infty := \prod_i q_i^{-\dim V_i}, \\
\tilde{\theta}_i &:= \theta_i, \qquad \tilde{\theta}_\infty := -\sum_i \theta_i \dim V_i.
\end{align*}

We define an $I$-graded vector space $W$ by 
$W_i := \C^{w_i}$. 
Then the vector space $\bM(\widetilde{V})$ can be identified with 
\[
\bM(V,W) := \bM(V) \oplus \bigoplus_{i\in I} \Hom (W_i, V_i) 
\oplus \bigoplus_{i\in I} \Hom (W_i,V_i).
\] 
For an element $x$ in $\bM(V,W)$, 
we usually denote its three components by $B=(B_h)~,a=(a_i),~b=(b_i)$. 
The multiplicative preprojective relation $\Phi_i(x) = q_i$ at $i \in I$ 
becomes 
\[
(1+a_{i,1}b_{i,1}) \cdots (1+a_{i,w_i}b_{i,w_i}) 
\prod_{i\in I} (1+B_h B_{\ov h})^{\epsilon(h)} =q_i,
\] 
where   
\[
a_i =
\begin{bmatrix} a_{i,1} & a_{i,2} & \cdots & a_{i,w_i} \end{bmatrix}
\colon \C^{w_i} \to V_i, \qquad 
b_i =
\begin{bmatrix} b_{i,1} \\ b_{i,2} \\ \vdots \\ b_{i,w_i} \end{bmatrix}
\colon V_i \to \C^{w_i}.
\]
The following can be checked easily.
\begin{proposition}\label{6.2.1}
A point $x=(B,a,b) \in \bM(V,W)$ is $\tilde{\theta}$-semistable 
if and only if the following conditions are satisfied:
\begin{enumerate}
\item[\rm (i)] For any $B$-invariant subspace $S \subset V$ contained in
$\Ker b := \bigoplus \Ker b_i$, the inequality $\theta \cdot \dim S \leq 0$ holds.  
\item[\rm (ii)] For any $B$-invariant subspace $T \subset V$ containing 
$\Im a := \bigoplus \Im a_i$,
the inequality $\theta \cdot \dim T \leq \theta \cdot \dim V$ holds.
\end{enumerate}
$x$ is $\tilde{\theta}$-stable
if and only if the strict inequalities hold in {\rm (i)}, {\rm (ii)} 
unless $S=0$, $T=V$ respectively.
\end{proposition}

For a subspace $S \subset V$ we usually identify $\dim S \in \N^I$ 
with $\sum_i (\dim S_i)\, \alpha_i \in Q^+$. 

We define 
\[
\MM_{q,\theta}(\bv,\bw) := \MM_{\tilde{q},\tilde{\theta}}(\widetilde{V}), 
\qquad 
\MMreg_{q,\theta}(\bv,\bw) := \MMreg_{\tilde{q},\tilde{\theta}}(\widetilde{V}),
\]
both of which we call the {\em framed multiplicative quiver varieties}. 
One can easily check that the dimension of $\MMreg_{q,\theta}(\bv,\bw)$ can be written as
\[
\dim \MMreg_{q,\theta}(\bv,\bw) = \langle \bv^{\vee} ,2\bw -\bv \rangle,
\]
where $\bv^{\vee} := \sum_i (\dim V_i)\, h_i$.

\subsection{Brill-Noether locus, Steinberg variety and Hecke correspondence}\label{6.3}

In this subsection, we assume that:
\begin{enumerate}
\item[\rm (i)] $h<h'$ for all $h\in \Omega, h'\in \ov\Omega$; and  
\item[\rm (ii)] $q_i=1$ and $\theta_i >0$ for all $i$.
\end{enumerate}
Note that the stability condition for $(B,a,b) \in \Phi^{-1}(1)$
then becomes
\begin{itemize}
\item If a subspace $S \subset V$ is $B$-invariant and contained in $\Ker b$,
then $S=0$,
\end{itemize}
and the semistability coincides with the stability. 
We write 
\[
\bMs(V,W) =\bMs_\theta(\widetilde{V}), \qquad  
Z^{\rm s}(V,W) = \Phi^{-1}(1) \cap \bMs(V,W),
\]
and
\[
\MM(\bv,\bw) =\MM_{1,\theta}(\bv,\bw), \qquad \MM_0(\bv,\bw)=\MM_{1,0}(\bv,\bw).
\]
Also we write $\MMreg_0(\bv,\bw) = \MMreg_{1,0}(\bv,\bw)$. 

The purpose of this section is to show that all the results proved by Nakajima 
in \cite[\S 4]{Nak-Duke2} can be shown analogously in the multiplicative case,
and that we can define the multiplicative version of 
Hecke correspondence of quiver varieties.

Since the projection $Z^{\rm s}(V,W) \to \MM(\bv,\bw)$ 
is a principal $G_V$-bundle
and each $V_i, W_i$ are representation spaces of $G_V$,
we can define associated vector bundles
\[
\mc{V}_i=Z^{\rm s}(V,W) \times_{G_V} V_i, \qquad
\mc{W}_i=Z^{\rm s}(V,W) \times_{G_V} W_i.
\]
We call these the {\em tautological bundles}.

Consider the following sequence of vector bundles:
\[
C^{\bullet}_i(\bv,\bw)\colon 
\begin{CD}
\mc{V}_i @>{\sigma_i}>> \bigoplus_{h\in H_i}\mc{V}_{\vout(h)} \oplus \mc{W}_i
@>{\tau_i}>> \mc{V}_i
\end{CD}\, ,
\]
where we have assigned the degree 0 to the middle term. 
$C^{\bullet}_i(\bv,\bw)$ is a complex by the multiplicative preprojective relation, 
and the degree (-1) cohomology vanishes by Proposition~\ref{3.2.8}.
Let $Q_i(\bv,\bw)$ denote the degree 0 cohomology; 
$Q_i(\bv,\bw) :=H_0(C_i^\bullet (\bv,\bw)) =\Ker \tau_i /\Im \sigma_i$.

We introduce the following subset of $\MM(\bv,\bw)$:
\begin{align*}
\MM_{i;n}(\bv,\bw) &:= \lset{[B,a,b] \in \MM(\bv,\bw)}%
{\corank \tau_i(B,a,b) = n}, \\
\MM_{i;\leq n}(\bv,\bw) &:= \bigcup_{m \leq n} \MM_{i;m}(\bv,\bw).
\end{align*}

Since $\MM_{i;\leq n}(\bv,\bw)$ is an open subvariety of $\MM(\bv,\bw)$,
$\MM_{i;n}(\bv,\bw)$ is a locally closed subvariety.
The restriction $Q_{i;n}(\bv,\bw):=Q_i(\bv,\bw) |_{\MM_{i;n}(\bv,\bw)}$ 
is a vector bundle of rank $\langle h_i ,\bw -\bv \rangle +n$. 
$\MM_{i;n}(\bv,\bw)$ is a multiplicative analogue of the 
Brill-Noether locus of the quiver variety. 

Replacing $V_i$ to $\Im \tau_i$,
we have a natural morphism
\[
p\colon \MM_{i;n}(\bv,\bw) \to \MM_{i;0}(\bv -n\alpha_i,\bw).
\]

Similar to \cite[Proposition 4.5]{Nak-Duke2}, we have:

\begin{proposition}\label{6.3.1}
Let $G(n,Q_{i;0}(\bv -n\alpha_i,\bw))$ be 
the Grassmann bundle of $n$-planes in 
$Q_{i;0}(\bv -n\alpha_i,\bw)$. 
Then we have the following commutative diagram:
\[
\begin{CD}
G(n,Q_{i;0}(\bv -n\alpha_i,\bw)) @>{\text{projection}}>> \MM_{i;0}(\bv -n\alpha_i,\bw) \\
@VV{\simeq}V  @| \\
\MM_{i;n}(\bv,\bw) @>{p}>> \MM_{i;0}(\bv -n\alpha_i,\bw) \, .
\end{CD}
\]
The kernel of the natural surjective homomorphism 
$p^* Q_{i;0}(\bv -n\alpha_i,\bw) \to Q_{i;n}(\bv,\bw)$ 
is isomorphic to the tautological vector bundle 
of $G(n,Q_{i;0}(\bv -n\alpha_i,\bw))$ via the isomorphism of the first row.  
In particular,
\begin{align*}
\dim \MM_{i;n}(\bv,\bw) &=\dim \MM_{i;0}(\bv -n\alpha_i,\bw) + 
n\left(  \langle h_i ,\bw -\bv \rangle +n \right)   \\
&= \dim \MM(\bv,\bw) -n\left(  \langle h_i ,\bw -\bv \rangle +n \right)   .
\end{align*}
\end{proposition}

\begin{proof}
The proof is almost the same as \cite[Proposition 4.5]{Nak-Duke2}.

The vector bundle $p^* Q_{i;0}(\bv -n\alpha_i,\bw)$ 
is given by 
$\Ker \tau_i / \sigma_i (\Im \tau_i)$. 
Considering the natural surjection 
$\Ker \tau_i / \sigma_i (\Im \tau_i) \to \Ker \tau_i /\Im \sigma_i$,
we have a surjective homomorphism 
$p^* Q_{i;0}(\bv -n\alpha_i,\bw) \to Q_{i;n}(\bv,\bw)$.
Its kernel $\Im \sigma_i /\sigma_i (\Im \tau_i)\simeq V_i /\Im \tau_i$ 
has a constant rank $n$. 
Thus we get a morphism from $\MM_{i;n}(\bv,\bw)$ to the Grassmann bundle.

Conversely suppose that a point $\phi$ in the Grassmann bundle is given.
Take a subspace $V' \subset V$ such that $\dim V' = \bv -n\alpha_i$. 
Let $(B',a',b') \in Z^{\rm s}(V',W)$ be 
a representative of the image of $\phi$ under the projection, and 
$\sigma'_i, \tau'_i$ denote $\sigma_i(B',a',b'), \tau_i(B',a',b')$ respectively.
Take an injective homomorphism 
$\sigma_i\colon V_i \to \widehat{V}_i \oplus W_i$ such that 
$\sigma_i |_{V'_i} =\sigma'_i$ and $\Im \sigma_i /\Im \sigma'_i = \phi$. 
Now we define   
\begin{align*}
B_h &:=B'_h \quad \text{for}\  h \notin H_i \cup \ov{H_i}, \\
a_j &:=a'_j, ~b_j :=b'_j \quad \text{for}\ j\neq i, \\
B_h &:= B'_h\colon V_{\vout(h)} \to V'_i \hookrightarrow V_i \quad \text{for}\ h \in H_i, \\ 
a_i &:= a'_i\colon W_i \to V'_i \hookrightarrow V_i,
\end{align*}
and define $b_i$ and $B_{\ov h}$ for $h\in H_i$ by the condition 
$\sigma_i(B,a,b) = \sigma_i$. 
Since $\sigma_i |_{V'_i} =\sigma'_i$, one can prove inductively 
that $b_i |_{V'_i} = b'_i$ and $B_{\ov h} |_{V'_i} =B'_{\ov h}$ for $h\in H_i$. 
Thus $\tau_i = \tau'_i$ and hence 
\[
\Im \sigma_i / \sigma_i(\Im \tau_i) 
= \Im \sigma_i /\Im \sigma'_i =\phi.
\] 
By definition we have $\tau_i \sigma_i = 0$, which implies $\Phi_i(B,a,b)=1$. 
Moreover $b_i |_{V'_i} = b'_i$ and $B_{\ov h} |_{V'_i} =B'_{\ov h}$ implies 
$b_i a_i = b'_i a'_i$ and $B_{\ov h}B_h = B'_{\ov h}B'_h$, 
and hence $\Phi_j(B,a,b)=1$ for all $j \neq i$. 
Thus $(B,a,b) \in \Phi^{-1}(1)$.

To check the stability condition, 
suppose that there is a $B$-invariant subspace $S$ 
contained in $\Ker b$.
We define a subspace $S' \subset V'$ by
\[
S'_j=\begin{cases} S_j & \text{if}\ j \neq i, \\
S_i \cap \Im \tau'_i =S_i \cap \Im \tau_i & \text{if}\ j=i. \end{cases}
\]
Then one can easily check that  
$S'$ is $B'$-invariant and contained in $\Ker b'$ 
using Lemma~\ref{3.2.3}.
Thus $S'=0$ by the stability condition for $(B',a',b')$.
In particular $S_j=0$ for $j\neq i$,
which implies $\sigma_i (S_i)=0$. 
Since we have taken $\sigma_i$ to be injective,
$S_i$ must be zero.
Thus $(B,a,b)$ is stable. 
Taking a quotient by the $G_V$-action 
we obtain a morphism from the Grassmann bundle to $\MM_{i;n}(\bv,\bw)$.
It is the inverse of the previous morphism.

To prove the last equality we compute that
\begin{align*}
\dim \MM(\bv,\bw) -\dim \MM(\bv -n\alpha_i,\bw) 
&= \langle \bv^{\vee} ,2\bw -\bv \rangle - 
\langle \bv^{\vee} -nh_i,  2\bw -\bv +n\alpha_i \rangle \\
&= 2n\left(  \langle h_i ,\bw -\bv \rangle +n \right)   .
\end{align*}
\end{proof}

Let $\bv_1, \bv_2 \in Q^+$ and $\bw \in P^+$.  
Let $\pi \colon \MM(\bv^i,\bw) \to \MM_0(\bv^i,\bw)$ $(i=1,2)$ be the 
projection. 
Recall that $\MM_0(\bv^i, \bw)$ is naturally embedded in 
$\MM_0(\bv^1 + \bv^2,\bw)$ by Proposition~\ref{2.2.7}. 
Thus we can regard $\pi$'s as maps to $\MM_0(\bv^1+\bv^2,\bw)$. 
Following Nakajima~\cite{Nak-Duke2}, we define  
\begin{align*}
\SS(\bv^1,\bv^2;\bw) &:= 
\lset{ (x^1,x^2) \in \MM(\bv^1,\bw) \times \MM(\bv^2,\bw)}
{\pi(x^1)=\pi(x^2)} \\ 
&= \MM(\bv^1,\bw) \times_{\MM_0(\bv^1+\bv^2,\bw)} \MM(\bv^2,\bw),
\end{align*}
which is an analogue of the Steinberg variety.

\begin{definition}\label{6.3.2}
For $n \in \Z_{>0}$ and $\bv \in Q^+$, 
the {\em Hecke correspondence} $\PP^{(n)}_i(\bv,\bw)$
is the variety defined as 
\[
\PP_i^{(n)}(\bv,\bw) := \left\{\,  (B,a,b,S) \, \,
\begin{array}{|l}  
(B,a,b) \in Z^{\rm s}(V,W) , \, S \subset V, \\
\text{$S$ is $B$-invariant, $\Im a \subset S$ and    
$\dim S=\bv -n\alpha_i$} 
\end{array} \right\} /G_V.
\]
We denote $\PP_i(\bv,\bw) = \PP^{(1)}_i(\bv,\bw)$.  
\end{definition}

We have the following diagram:
\begin{equation}\label{6.3.3}   
\MM(\bv-n\alpha_i,\bw) \stackrel{p_1}{\longleftarrow} 
\PP_i^{(n)}(\bv,\bw) 
\stackrel{p_2}{\longrightarrow} \MM(\bv,\bw)  .
\end{equation}
The first map is given by the restriction of $(B,a,b)$ to $S$, 
and the second is given by forgetting $S$
(It is clear that $(B,a,b)|_S \in Z^{\rm s}(S,W)$). 
  
Note that 
$p_1 \times p_2 \colon \PP_i^{(n)}(\bv,\bw) 
\to \MM(\bv-n\alpha_i,\bw) \times \MM(\bv,\bw)$
is an embedding whose image consists of all pairs 
$([B'',a'',b''],[B,a,b])$ such that there exists $\xi \in \bigoplus_{i\in I}\Hom(V''_i,V_i)$ 
satisfying
\[
\xi B'' = B \xi, \qquad \xi a'' =a, \qquad b''=b \xi. 
\] 
Here we fix an $I$-graded vector space $V''$ such that $\sum (\dim V''_i)\, \alpha_i =\bv -\bv'$. 
Indeed, if such a $\xi$ exists, then $\Ker \xi$ is zero by the stability condition 
and $\Im \xi$ is $B$-invariant and contains $\Im a$. 
Moreover $\xi$ is unique if we fix representatives $(B'',a'',b''), (B,a,b)$. 
Thus the point $[(B,a,b),\Im \xi] \in \PP_i^{(n)}(\bv,\bw)$ is well-defined. 

It is clear that this subvariety is contained in $\SS(\bv-n\alpha_i,\bv;\bw)$. 
So we may regard $\PP_i^{(n)}(\bv,\bw)$ as a subvariety of 
$\SS(\bv-n\alpha_i,\bv;\bw)$. 

Similar to \cite[Lemma 5.12]{Nak-Duke2}, we can prove the following.

\begin{proposition}\label{6.3.4}
Consider the diagram \eqref{6.3.3} with $n=1$. 

{\rm (i)} $p_1^{-1}(\MM_{i;n}(\bv -\alpha_i,\bw))$ can be identified with the
projective bundle $\P \left( Q_{i;n}(\bv -\alpha_i ,\bw) \right)$. 

{\rm (ii)} $p_2^{-1}(\MM_{i;n}(\bv,\bw))$ can be identified with the 
projective bundle $\P \left( H_1(C_i^\bullet(\bv,\bw))^* |_{\MM_{i;n}(\bv,\bw)} \right)$. 
\end{proposition}

\begin{proof}
The proof is similar to Proposition~\ref{6.3.1}
\end{proof}

\subsection{Constructible functions}\label{6.4}

Let $X$ be an algebraic variety.
A $\Q$-valued {\em constructible function} on $X$ 
is a function $f\colon X \to \Q$ such that 
$f(X)$ is finite and $f^{-1}(a)$ is constructible for all $a \in \Q$.
Let $\CF(X)$ be the $\Q$-vector space consisting of 
all $\Q$-valued constructible functions on $X$.
If $Y \subset X$ is a subvariety, we regard $\CF(Y)$ as a subspace of $\CF(X)$ 
by extending with zero on the complement. 
A typical example of constructible functions is the characteristic 
function $[A]$ of a constructible subset $A \subset X$;
\[
[A](x) := \begin{cases} 1 & \text{if}\ x \in A, \\
0 & \text{otherwise},
\end{cases}
\] 
and by the definition any constructible functions can be written 
as a linear combination of characteristic functions.

Any morphism $p\colon X \to Y$ induces the {\em pull-back} and the {\em push-forward}
between the vector spaces of constructible functions:
\begin{align*}
p^* \colon \CF(Y) &\to \CF(X); \qquad (p^* g)(x) := g(p(x)), \\
p_! \colon \CF(X) &\to \CF(Y); \qquad (p_! f)(y) := 
\sum_{c \in \Q} c \, \chi (p^{-1}(y) \cap f^{-1}(c)),
\end{align*}
where $\chi$ denotes the Euler characteristic.
Regarding $\chi$ as a ``measure" of constructible subsets, 
$p_!$ is also written as 
\[
(p_! f)(y) = \int_{x \in p^{-1}(y)} f(x).
\]

Let $M_1, M_2, M_3$ be three varieties and 
$p_{ij}\colon M_1 \times M_2 \times M_3 \to M_i \times M_j$ $(i,j=1,2,3)$ 
be the projection to the $i$-th and $j$-th factors.
For $f \in \CF(M_1 \times M_2)$ and $f' \in \CF(M_2 \times M_3)$, 
the {\em convolution product} $f * f'$ of $f$ and $f'$ is defined by
\[
f * f' := (p_{13})_! (p_{12}^*(f) p_{23}^*(f')) \in \CF(M_1 \times M_3).
\]
Note that it can be written as
\[
(f * f')(x_1,x_3)=\int_{x_2 \in M_2} f(x_1,x_2)f'(x_2,x_3).
\]
It is easy to see that the convolution product is associative.

Suppose there are morphisms $p_i \colon M_i \to M_0$ 
$(i=1,2,3)$ to some variety $M_0$. 
Then it is clear that if $f \in \CF(M_1 \times_{M_0} M_2)$ and 
$f' \in \CF(M_2 \times_{M_0} M_3)$, 
then $f * f' \in \CF(M_1 \times_{M_0} M_3)$, i.e., the support of $f * f'$ 
is contained in $M_1 \times_{M_0} M_3$.

\subsection{A geometric construction of the universal enveloping algebra}\label{6.5}

Let $\mathcal{A}(\bw)$ be the subspace of the direct product
\[
\prod_{\bv^1,\bv^2} \CF(\SS(\bv^1,\bv^2;\bw))
\]
consisting of all elements $(F_{\bv^1,\bv^2})$ such that the following 
two conditions are satisfied:
\begin{enumerate}
\item[\rm (i)] For fixed $\bv^1$, $F_{\bv^1,\bv^2}=0$ for all but finitely many $\bv^2$. 
\item[\rm (ii)] For fixed $\bv^2$, $F_{\bv^1,\bv^2}=0$ for all but finitely many $\bv^1$.
\end{enumerate}    
By the convolution product, it is an associative algebra with 
$1=\sum_{\bv} [\Delta(\bv,\bw)]$, where $\Delta(\bv,\bw)$ denotes 
the diagonal subset of $\MM(\bv,\bw) \times \MM(\bv,\bw)$.

Let 
$(\bullet)^{\dagger} \colon \MM(\bv-\alpha_i,\bw) \times \MM(\bv,\bw) 
\to \MM(\bv,\bw) \times \MM(\bv-\alpha_i,\bw)$
be the flip of the components.
The main theorem in this section is the following:

\begin{theorem}[cf.\ {\cite[Theorem 9.4]{Nak-Duke2}}]\label{6.5.1}
There is an algebra homomorphism $\bU \to \mathcal{A}(\bw)$ 
such that
\[
h \mapsto \sum_{\bv} \langle h,\bw -\bv \rangle \cdot [\Delta(\bv,\bw)], \qquad 
e_i \mapsto \sum_{\bv^2} [\PP_i(\bv^2,\bw)], \qquad 
f_i \mapsto \sum_{\bv^2} [\PP_i(\bv^2,\bw)^{\dagger}].
\] 
\end{theorem}

The relations \eqref{6.1.1}, \eqref{6.1.2}, \eqref{6.1.3} are obviously satisfied. 
We check the relations \eqref{6.1.4} and \eqref{6.1.5}, \eqref{6.1.6} 
by the same method as in \cite{Nak-Duke1, Nak-Duke2}. 

\subsection{The relation $[e_i ,f_j] = \delta_{ij} h_i$}\label{6.6}

In this subsection we check the relation $[e_i ,f_j] = \delta_{ij} h_i$. 
Fix $\bv \in Q^+$ and consider the following diagram:
\[
\begin{CD}
\MM(\bv-\alpha_i,\bw) @<<< \SS(\bv-\alpha_i,\bv;\bw) \\
@. @VVV \\ 
@. \MM(\bv,\bw) @<<< \SS(\bv, \bv-\alpha_j;\bw) \\
@. @. @VVV \\ 
@. @. \MM(\bv-\alpha_j,\bw) ,
\end{CD}
\]
where the arrows are the natural morphisms.
Set $\bv^1 := \bv-\alpha_i,~ \bv^2:=\bv,~ \bv^3:=\bv-\alpha_j$ 
and let $\SS(\bv^1,\bv^2,\bv^3;\bw)$ be the fiber product 
of $\SS(\bv^1,\bv^2;\bw)$ and $\SS(\bv^2, \bv^3;\bw)$ 
over $\MM(\bv^2,\bw)$. 
Then the above diagram induces the natural morphisms
\[
\MM(\bv^i,\bw) \stackrel{p_i}{\longleftarrow} \SS(\bv^1,\bv^2,\bv^3;\bw) 
\stackrel{p_j}{\longrightarrow} \MM(\bv^j,\bw) 
\]
for $i,j=1,2,3$. 
We set $p_{ij} := p_i \times p_j \colon 
\SS(\bv^1,\bv^2,\bv^3;\bw) \to \MM(\bv^i,\bw) \times \MM(\bv^j,\bw)$. 
Then $e_i f_j$ is described as 
\begin{align*}
e_i f_j &= \sum_{\bv^2} (p_{13})_! 
\left(  p_{12}^*[\PP_i(\bv^2,\bw)] p_{23}^*[\PP_j(\bv^2,\bw)^{\dagger}] \right)   \\
&= \sum_{\bv^2} (p_{13})_! \left[ p_{12}^{-1}(\PP_i(\bv^2,\bw)) \cap 
p_{23}^{-1}(\PP_j(\bv^2,\bw)^{\dagger}) \right] . 
\end{align*}

Next consider the following diagram: 
\[
\begin{CD}
\MM(\bv-\alpha_i,\bw) @<<< \SS(\bv-\alpha_i,\bv-\alpha_j-\alpha_i;\bw) \\
@. @VVV \\
@. \MM(\bv-\alpha_j-\alpha_i,\bw) @<<< \SS(\bv-\alpha_j-\alpha_i, \bv-\alpha_j;\bw) \\
@. @. @VVV \\  
@. @. \MM(\bv-\alpha_j,\bw).
\end{CD}
\]
Set $\bv^4 := \bv-\alpha_j-\alpha_i$, 
and define $\SS(\bv^1,\bv^4,\bv^3)$ and 
$q_{ij}\colon \SS(\bv^1,\bv^4,\bv^3) \to \MM(\bv^i,\bw) \times \MM(\bv^j,\bw)$ 
as above.  
Then $f_j e_i$ is described as
\begin{align*}
f_j e_i &= \sum_{\bv^4} (q_{13})_! \left( 
q_{14}^*[\PP_j(\bv^1,\bw)^{\dagger}] q_{43}^*[\PP_i(\bv^3,\bw)]) \right)   \\
&= \sum_{\bv^4} (q_{13})_! \left[ q_{14}^{-1}(\PP_j(\bv^1,\bw)^{\dagger}) \cap 
q_{43}^{-1}(\PP_i(\bv^3,\bw)) \right].
\end{align*}

The following lemma can be proved by the same way as \cite[Lemma 9.10]{Nak-Duke2}.

\begin{lemma}\label{6.6.1}
Let $U \subset \MM(\bv^1,\bw) \times \MM(\bv^3,\bw)$ denotes the outside of 
the diagonal when $i=j$, and the whole set otherwise.
Then there is an isomorphism 
\[
\Pi \colon p_{13}^{-1}(U) \cap p_{12}^{-1}(\PP_i(\bv^2,\bw)) \cap 
p_{23}^{-1}(\PP_j(\bv^2,\bw)^{\dagger}) 
\longrightarrow 
q_{13}^{-1}(U) \cap 
q_{14}^{-1}(\PP_j(\bv^1,\bw)^{\dagger}) \cap 
q_{43}^{-1}(\PP_i(\bv^3,\bw))
\]
such that $q_{13} \circ \Pi = p_{13}$.
\end{lemma}

Thus $e_i f_j - f_j e_i =0$ if $i \neq j$, and 
the support of $e_i f_i - f_i e_i$ is contained in 
$\sqcup_{\bv}\Delta(\bv,\bw)$.  
So it is sufficient to prove   
\[
(e_if_i -f_i e_i)(x,x) = \langle h_i, \bw -\bv \rangle
\] 
for all $x \in \MM(\bv,\bw)$. 

Suppose that $x \in \MM_{i;n}(\bv,\bw)$.
Then using Proposition~\ref{6.3.4} we compute 
\begin{align*}
\chi \left(  p_{13}^{-1}(x,x) \cap p_{12}^{-1}(\PP_i(\bv,\bw)) \cap 
p_{23}^{-1}(\PP_j(\bv,\bw)^{\dagger}) \right)    
&= \chi \left(  \P \left(  Q_{i,n}(\bv-\alpha_i,\bw) |_x\right)   \right)   \\
&= \langle h_i, \bw -\bv \rangle + r + 1.
\end{align*}
Similarly,
\begin{align*} 
\chi \left(  {q_{13}}^{-1}(x,x) \cap 
q_{14}^{-1}(\PP_j(\bv-\alpha_i,\bw)^{\dagger}) \cap 
q_{43}^{-1}(\PP_i(\bv-\alpha_i,\bw)) \right)  
&= \chi \left(  \P \left(  H_1 ( C^\bullet_i(\bv-\alpha_i,\bw) ) |_x\right)   \right)   \\
&= r+1.
\end{align*}
Thus
\[
(e_if_i -f_i e_i )(x,x) = \langle h_i, \bw -\bv \rangle + r +1 - (r+1) 
= \langle h_i, \bw -\bv \rangle.
\]

\subsection{The Serre relations}\label{6.7}

In this subsection we check the relations \eqref{6.1.5} and \eqref{6.1.6}.

Fix vertices $i, j$ with $i \neq j$, and set $N :=-c_{ij}$. 
For $n= 0,1, \dots ,N+1$, 
let $P_n$ be the fiber product
\[
\PP_j^{(N+1-n)}(\bv-n\alpha_j-\alpha_i,\bw) \times_{\MM(\bv-n\alpha_j-\alpha_i,\bw)}
\PP_i(\bv-n\alpha_j,\bw) \times_{\MM(\bv-n\alpha_j,\bw)} \PP_j^{(n)}(\bv,\bw).
\]
Consider the variety consisting of 
all tuples $(B,a,b,S^1,S^2,S^3)$, where 
$(B,a,b) \in Z^{\rm s}(V,W)$ and 
each $S^k$ is $B$-invariant subspace of $V$ such that 
\begin{align*}
&S^1\supset S^2 \supset S^3 \supset \Im a, \quad \text{and} \\  
&\dim S^1 = \bv -n\alpha_i,~\dim S^2=\bv -n\alpha_j -\alpha_i,
~\dim S^3=\bv -(N+1)\alpha_j -\alpha_i.
\end{align*} 
Then the quotient of it by the $G_V$-action is naturally isomorphic to $P_n$.

Set
\[
P := \left\{\,  (B,a,b,S) \, \,
\begin{array}{|l}  
(B,a,b) \in Z^{\rm s}(V,W) ,\, S \subset V, \\
\text{$S$ is $B$-invariant, $\Im a \subset S$ and   
$\dim S=\bv -\alpha_i -(N+1)\alpha_j$} 
\end{array} \right\} /G_V.
\]
It is a subvariety of $\SS(\bv -\alpha_i -(N+1)\alpha_j,\bv;\bw)$, 
and we have a natural morphism 
$r_n \colon P_n \to P$ 
which sends $[(B,a,b,S^1,S^2,S^3)]$ to $[(B,a,b,S^3)]$. 

\begin{lemma}\label{6.7.1}
We have
\[
e_j^{N+1-n}e_ie_j^n[\Delta(\bv,\bw)]=
f_j^{n}f_if_j^{N+1-n}[\Delta(\bv,\bw)]
= (N+1-n)!n!(r_n)_! [P_n] [\Delta(\bv,\bw)].
\]
\end{lemma}

\begin{proof}
Consider the variety 
consisting of all tuples 
$(B,a,b,\{ S_k \}_{k=1}^n)$, where 
$(B,a,b) \in Z^{\rm s}(V,W)$ and  
each $S^k$ is a $B$-invariant subspace of $V$ such that 
\begin{align*} 
&S^1 \supset S^2 \supset \cdots \supset S^n \supset \Im a, \quad \text{and} \\ 
&\dim S^k = \bv -k\alpha_i.
\end{align*}
Let $\PP_i^n(\bv,\bw)$ be the quotient of it modulo $G_V$-action.  
Then $e_j^{N+1-n}e_ie_j^n[\Delta(\bv,\bw)]$ is given by 
the push-forward of 
\[
[\PP_j^{N+1-n}(\bv-n\alpha_j-\alpha_i,\bw)] *[\PP_i(\bv-n\alpha_j,\bw)] *[\PP_j^n(\bv,\bw)]
\] 
by the obvious morphism. 

Let $\pi_n \colon \PP_i^n(\bv,\bw) \to \PP^{(n)}_i(\bv,\bw)$ 
be the morphism $[(B,a,b,S^1, \dots ,S^n)] \mapsto [(B,a,b,S^n)]$. 
Then the fiber of $\pi_n$ is isomorphic to 
the full flag variety of $n$-dimensional vector space, 
and hence its Euler characteristic is just $n!$.
Thus
\[
(\pi_n)_![\PP_i^n(\bv,\bw)] = n! [\PP_i^{(n)}(\bv,\bw)].
\]
This proves the assertion.
\end{proof}

By the above lemma, it is enough to show that
\[
\sum_{n=0}^{N+1} (-1)^n \chi (r_n^{-1}(x)) =0 \qquad \text{for any}\ \ x 
\in \PP(\bv,\bw; \alpha_i+(N+1)\alpha_j).
\]
Take a representative $(B,a,b,S)$ of $x$.

Recall the complex 
\[
\begin{CD}
V_j @>{\sigma_j}>> \widehat{V}_j \oplus W_j @>{\tau_j}>> V_j \, .
\end{CD}
\]
For $k \in I$ with $k\neq j$, 
let $\sigma_j^k$ be the projection of  
$\sigma_j$ to $\bigoplus_{h\in H_j; \vout(h)=k} V_k$, 
and $\tau_j^k$ be the restriction of $\tau_j$ on  
$\bigoplus_{h\in H_j; \vout(h)=k} V_k \subset \widehat{V}_j$. 
Similarly, let $\sigma_j^W$ be the projection of $\sigma_j$ to $W_j$ 
and $\tau_j^W$ be the restriction of $\tau_j$ on $W_j$.  

\begin{lemma}\label{6.7.2}
Define $T_1 :=S_j +\Im \tau_j^i$ and $T_2 := (\sigma_j^i)^{-1}(S_i^{\oplus N})$. 
Then the fiber $r_n^{-1}(x)$ is isomorphic to the variety consisting of 
all codimension $n$ subspaces $T \subset V_j$ such that 
$T_1 \subset T \subset T_2$.
\end{lemma}

\begin{proof}
For given $(B,a,b,S^1,S^2,S^3) \in r_n^{-1}(x)$, 
we set $T:= S^1_j = S^2_j$. 
Since $S^1,S^2$ are $B$-invariant and $V_i=S^1_i$, 
$T_1 \subset T \subset T_2$ is clearly satisfied.

Conversely suppose that a codimension $n$ subspace $T \subset V_j$ with 
$T_1 \subset T \subset T_2$ is given. 
Then we set 
\[ S^3 := S, \qquad S^2_k :=
\begin{cases}
T & \text{if}\ k =j, \\
S_k & \text{if}\ k \neq j, 
\end{cases}
\qquad
S^1_k :=
\begin{cases}
T & \text{if}\ k =j, \\
V_k & \text{if}\ k \neq j. 
\end{cases}
\]
$T_1 \subset T$ implies $S^3 =S \subset S^2$. 
We prove that $S^1,S^2$ are $B$-invariant. 
Since $S^2_k = S_k$ for $k \neq j$ and $S^1_k =S^2_k$ for $k \neq i$, 
it is enough to show that 
\[
\sigma_j(S^2_j) \subset \widehat{S}_j \oplus W_j, 
\qquad
\Im \tau_j \subset S^1_j
\]
by Lemma~\ref{3.2.3}. 
$T \subset T_2$ implies 
$\sigma_j^i (S^2_j) \subset S_i^{\oplus N}$, 
and $S^2_k =V_k$ for $k\neq i,j$ implies $\sigma_j^k(S^2_j) \subset S_k^{\oplus -c_{kj}}$. 
Thus $\sigma_j(S^2_j) \subset \widehat{S}_j \oplus W_j$.  
Also, $\Im \tau_j^k \subset S^1_j$ for $k \neq i,j$ follows from 
$V_k = S_k$ and $S_j \subset S^1_j$, 
and $T_1 \subset T$ implies $\Im \tau_j^i \subset S^1_j$.
Moreover $\Im a \subset S \subset S^1$ implies $\Im \tau_j^W \subset S^1_j$.   
So we get $\tau_j(\widehat{S^1}_j \oplus W_j) \subset S^1_j$. 

We complete the proof. 
\end{proof}

\begin{lemma}\label{6.7.3}
$T_1 \neq T_2$.
\end{lemma}

\begin{proof}
Note that
$\tau_j^i \sigma_j^i = -\tau_j^W \sigma_j^W - \sum_{k \neq i,j} \tau_j^k \sigma_j^k$, 
and hence 
\[
\Im \tau_j^i \sigma_j^i = \Im \tau_j^W +\sum_{k \neq i,j} \tau_j^k \sigma_j^k (V_j) 
\subset S_j + \sum_{k \neq i,j} \tau_j^k \left(  V_k^{\oplus -c_{kj}} \right)   
= S_j +\sum_{k \neq i,j} \tau_j^k \left(  S_k^{\oplus -c_{kj}} \right)   
\subset S_j.
\]
So by the definitions of $T_1, T_2$ we have a complex
\[
\begin{CD}
0 @>>> V_j/T_2 @>{\sigma_j^i}>> (V_i/S_i)^{\oplus N} @>{\tau_j^i}>> T_1/S_j @>>> 0\, ,
\end{CD}
\]
which is exact except possibly at the middle term.
Hence we have $\dim V_j/T_2 \leq N - \dim T_1 /S^3_j$, 
and hence
\[
\dim T_2 - \dim T_1 \geq \dim V_j -\dim S^3_j -N =1.
\]
Thus $T_1 \neq T_2$.
\end{proof}

Set $d_i := \codim T_i~(i=1,2)$. 
Then the fiber $r_n^{-1}(x)$ is empty unless $T_1 \subset T_2$ and 
$d_1 \leq n \leq d_2$, 
in which case $r_n^{-1}(x)$ is a Grassmannian manifold of 
$(d_1 -n)$-dimensional subspaces in a $(d_1 -d_2)$-dimensional space. 
Thus
\[
\sum_{n=0}^{N+1} (-1)^n \chi (r_n^{-1}(x)) 
= \sum_{n=d_2}^{d_1} (-1)^n \binom{d_1 -d_2}{d_1 -n} =0.
\] 

We complete the proof.

\subsection{Construction of irreducible highest weight representations}\label{6.8}

Let $\LL(\bv,\bw)$ denote the nilpotent subvariety $\pi^{-1}([0]) \subset \MM(\bv,\bw)$.  
The vector space $\bigoplus_{\bv} \CF(\LL(\bv,\bw))$  
becomes a representation space of the algebra $\mathcal{A}(\bw)$ 
by the following way:
\[
(F * f)(x^1) := \int_{x^2 \in \LL(\bv^2,\bw)} F(x^1,x^2) f(x^2)
\quad \text{for}\ \  
F \in \CF(\SS(\bv^1,\bv^2;\bw)),\ f \in \CF(\LL(\bv^2,\bw)).
\]

Note that $\MM(0,\bw)$ (and hence $\LL(0,\bw)$) consists of a single point.  
Set 
\[
L(\bw) := \bU^- \cdot [\LL(0,\bw)] \subset \bigoplus_\bv \CF(\LL(\bv,\bw)),
\qquad
L(\bv,\bw) := \CF(\LL(\bv,\bw)) \cap L(\bw).
\]

By the same way as \cite[Lemma 10.13]{Nak-Duke1},
one can easily show that $f_i^{w_i +1} [\LL(0,\bw)] =0$ for all $i\in I$, 
where $w_i = \langle \bw ,h_i \rangle$.
Thus we get the following corollary:

\begin{corollary}[cf.\ {\cite[Theorem 10.14]{Nak-Duke1}}]\label{6.8.1}
$L(\bw)$ is the irreducible highest weight integrable $\g$-module 
with the highest weight $\bw$, and $L(\bv,\bw)$ is 
the $(\bw -\bv)$-weight space of $L(\bw)$. 
\end{corollary}



\end{document}